\documentclass[10pt,a4paper]{amsart}

\usepackage{enumerate}
\usepackage{amsmath,amsfonts,amsthm,latexsym,amssymb,amscd,epsf}
\usepackage{flexisym}\usepackage{breqn}
\usepackage{amsrefs}
\usepackage{colortbl}
\usepackage{tablestyles}
\usepackage[table]{xcolor}
\usepackage{graphicx}

\usepackage{tikz}
\usepackage{caption}
\usepackage{subcaption}
\usepackage{float}
\usepackage{booktabs}
\usepackage{tcolorbox}
\usepackage{tabularx}
\usepackage{array}

   \usepackage{hyperref}
   \usepackage{cleveref}
\usepackage{adjustbox}
\newtheorem{theorem}{Theorem}[section]
\newtheorem{definition}[theorem]{Definition}
\newtheorem{lemma}[theorem]{Lemma}
\newtheorem{corollary}[theorem]{Corollary}
\newtheorem{proposition}[theorem]{Proposition}
\newtheorem{remark}[theorem]{Remark}
\newtheorem{example}[theorem] {Example}
\newtheorem{problem}{Problem}

\newtheorem{conjecture}[theorem]{Conjecture}
\usepackage[top=2cm, bottom=2cm, left=3cm, right=3cm]{geometry}

\usepackage[parfill]{parskip}    

\makeatletter
\def\subsubsection{\@startsection{subsubsection}{3}%
  \z@{.5\linespacing\@plus.7\linespacing}{.1\linespacing}%
  {\normalfont\itshape}}
\makeatother

\newcommand{\al}{\alpha}

\newcommand {\bC} {\mathbb {C}}
\newcommand {\bR} {\mathbb R}

\newcommand {\cD} {\mathcal D}

\newcommand {\Ga} {\Gamma}

\newcommand {\eps} {\epsilon}

\newcommand {\cA} {\mathcal A}

\newcommand {\cC} {\mathcal C}

\newcommand {\R}{\mathbb R}
\newcommand {\RR}{\mathcal R}

\DeclareMathOperator{\Ker}{Ker}

\DeclareMathOperator{\Coker}{Coker}
\DeclareMathOperator{\Ve}{Vert}
\DeclareMathOperator{\Diam}{Diam}

\theoremstyle{plain}

 \theoremstyle{definition}

             \numberwithin{equation}{section}

\makeatother
     
\begin{document}

             \title [Return of the plane evolute]
             {Return of the plane evolute}
             
 \author[R.~Piene]{Ragni Piene}
 \address{Department of Mathematics, University of Oslo, NO-0316 Oslo, Norway}  
 \email{ragnip@math.uio.no}

 \author[C.~Riener]{Cordian Riener}
\address{Department of Mathematics and Statistics, UiT The Arctic University of Norway, NO-9037 Troms\o, Norway} 
\email {cordian.riener@uit.no}

\author[B.~Shapiro]{Boris Shapiro}
\address{Department of Mathematics, Stockholm University, SE-106 91, Stockholm,
            Sweden}
\email{shapiro@math.su.se}

\dedicatory{``If I have seen further it is by standing on the shoulders of Giants." Isaac Newton, from a letter to Robert Hooke}

\date{\today}

\begin{abstract}
Below we consider the evolutes of  plane real-algebraic curves and discuss some of  their  complex and real-algebraic properties. In particular, for a given degree $d\ge 2$, we provide lower bounds for the following four numerical invariants: 
1)  the maximal number of times a real line can intersect the evolute of a  real-algebraic curve of degree $d$; 
2) the maximal number of  real cusps which can occur on  the evolute of a real-algebraic curve of degree $d$; 
3) the maximal number of  (cru)nodes which can occur on  the dual curve to the evolute of a  real-algebraic curve of degree $d$;
4) the maximal number of  (cru)nodes which can occur on the evolute of a real-algebraic curve of degree $d$.
\end{abstract}

\maketitle

\section{Short historical account}
\label{sec:int}

As we usually tell our students in  calculus classes, the evolute of a   curve in the Euclidean plane is the locus of  its centers of curvature. The following intriguing information about evolutes can be found on Wikipedia \cite{Wi}:
 ``Apollonius  (c. 200 BC) discussed evolutes in Book V of his treatise Conics. However, Huygens is sometimes credited with being the first to study them, see \cite{HuyB}.  Huygens formulated his theory of evolutes sometime around 1659 to help solve the problem of finding the tautochrone curve, which in turn helped him construct an isochronous pendulum. This was because the tautochrone curve is a cycloid, and cycloids have the unique property that their evolute is a cycloid of the same type. The theory of evolutes, in fact, allowed Huygens to achieve many results that would later be found using calculus, see \cite{Ar1, Yo}."  

Notice that \cite{Huy},  originally published in 1673 and freely available on the internet, contains a large number of beautiful illustrations including those of evolutes.  
Further exciting pictures of evolutes  can be found in the small book \cite{Ya} written about hundred years ago for high-school teachers.    

Among several dozens of books on (plane) algebraic curves available now, only very few  \cites{Co, Hi, Ya, Sa} mention evolutes at all, the best of them being \cite{Sa}, first published more than one and half century  ago. 
Some properties of evolutes have been studied in connection with the so-called 4-vertex theorem of Mukhopadhyaya--Kneser as well as its generalizations, see e.g. \cites{Fu, Ta}. Their definition has been generalized from the case of plane curves to that of plane fronts and also from the case of Euclidean plane to that of the Poincar\'e disk, see e.g. \cite{FuTa}. Singularities of evolutes and involutes have been discussed in details by  V.~Arnold and his school, see e.g. \cite{Ar1, Ar2} and more recently in  \cite{ST, ScT}. 

In recent years the notion of Euclidean distance degree of an algebraic variety studied in e.g. \cite{DHOST} and earlier in \cites{CT, JoPe} has attracted substantial attention of the algebraic geometry community. In  the case when the variety under consideration is a plane curve, the ED-discriminant in this theory is exactly the standard evolute. In our opinion, this connection calls for more studies of the classical evolute since in spite of more than three hundred years passed since its mathematical debut, the evolute of a  plane algebraic curve is still far from being well-understood. Below we attempt to develop some real algebraic geometry around the evolutes of  real-algebraic curves and their duals hoping to attract the attention of the fellow mathematicians to this beautiful and classical topic.

\section{Initial facts about evolutes and problems under consideration} \label{sec:init} 
From the computational point of view the most useful presentation of the evolute of a plane curve  is as follows.  
Using a local parametrization of a curve $\Ga$ in $\bR^2$, one can parameterize its evolute $E_\Ga$ as 
\begin{equation}\label{eq:basic}
E_\Ga(t)=\Ga(t)+\rho(t) \bar n(t),
\end{equation}
where $\rho(t)$ is its curvature radius at the point $\Ga(t)$ (assumed non-vanishing) and $\bar n(t)$ is  the unit normal at $\Ga(t)$  pointing towards the curvature center. In Euclidean coordinates, for $\Ga(t)=(x(t),y(t))$ and $E_\Ga(t)=(X(t),Y(t))$, one gets  the following explicit expression
\begin{equation}\label{eq:basic2}
\left\{
X(t)=x(t)-\frac{y^\prime(t)(x^\prime(t)^2+y^\prime(t)^2)}{x^\prime(t)y^{\prime\prime}(t)-x^{\prime\prime}(t)y^\prime(t)},\; Y(t)=y(t)+\frac{x^\prime(t)(x^\prime(t)^2+y^\prime(t)^2)}{x^\prime(t)y^{\prime\prime}(t)-x^{\prime\prime}(t)y^\prime(t)}
\right\}.
\end{equation}

If a curve $\Ga$ is given by an equation $f(x,y)=0$, then (as noticed in e.g., \cite{Hi},  Ch. 11, \S~2) the equation of its evolute can be obtained as follows. 
Consider the system 
\begin{equation}\label{eq:basic3}
\begin{cases}
f(x,y)=0\\
X=x+\frac{f^\prime_x((f^\prime_x)^2+(f^\prime_y)^2)}{2f^\prime_xf^\prime_yf^{\prime\prime}_{xy} -(f^\prime_y)^2f^{\prime\prime}_{xx}-(f^\prime_x)^2f^{\prime\prime}_{yy}}\\
Y=y+\frac{f^\prime_y((f^\prime_x)^2+(f^\prime_y)^2)}{2f^\prime_xf^\prime_yf^{\prime\prime}_{xy} -(f^\prime_y)^2f^{\prime\prime}_{xx}-(f^\prime_x)^2f^{\prime\prime}_{yy}}\\
\end{cases}
\end{equation}
defining the original curve and the family of centers of its curvature circles. Then eliminating the variables $(x,y)$ from \eqref{eq:basic3} one obtains a single equation defining the evolute in variables $(X,Y)$. For concrete bivariate polynomials $f(x,y)$ of lower degrees, such an elimination procedure can be carried out in, for example,  Macaulay 2 \cite{M2}. 

\begin{example}\label{ex:basic}{\rm Two basic examples of evolutes are as follows. 
\begin{enumerate}
\item For the parabola $\Ga=(t,t^2)$, its evolute is given by \[\textstyle E_\Ga(t)=(-4t^3, \frac{1}{2}+3t^2)\] which is a semicubic parabola satisfying  the equation $27X^2=16(Y-\frac{1}{2})^3$, see Fig.~  \ref{fig:conic1a}. 

\item For the ellipse  $\Ga=(a\cos t, b \sin t)$, the evolute is given by  \[\textstyle E_\Ga(t)=(\frac{a^2-b^2}{a}\cos^3 t, \frac{b^2-a^2}{b}\sin^3t),\]
 which is an astroid satisfying the equation $(aX)^{2/3}+(bY)^{2/3}=(a^2-b^2)^{2/3}$, see Fig.~\ref{fig:conic1b}. 
 \end{enumerate}}
\end{example}

\begin{figure}
     \centering
     \begin{subfigure}[t]{0.3\textwidth}
         \centering
      \includegraphics[scale=0.25]{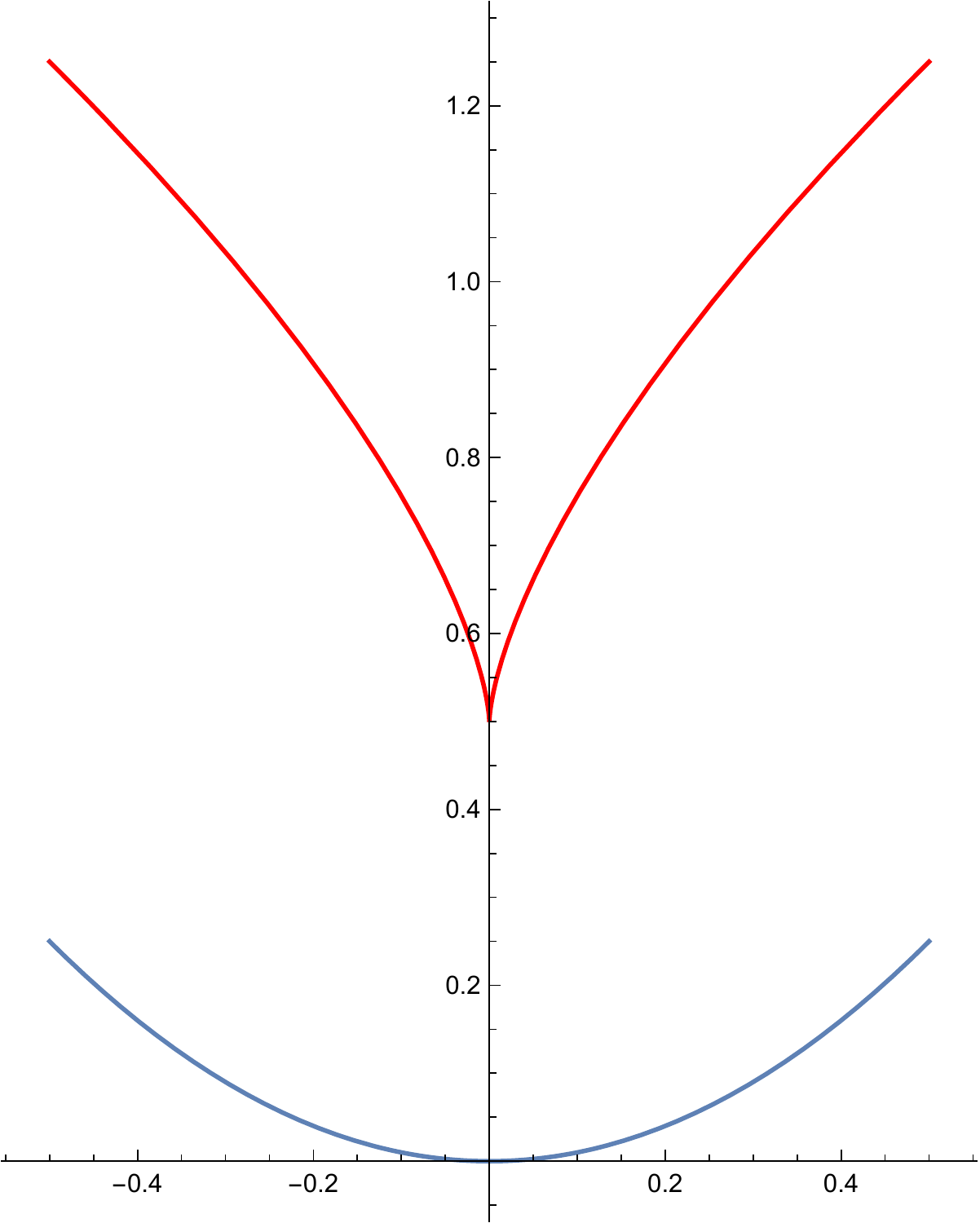} 
       \caption{Evolute of a parabola.}
        \label{fig:conic1a}
     \end{subfigure}
     \hfill
     \begin{subfigure}[t]{0.3\textwidth}
         \centering
        \includegraphics[scale=0.28]{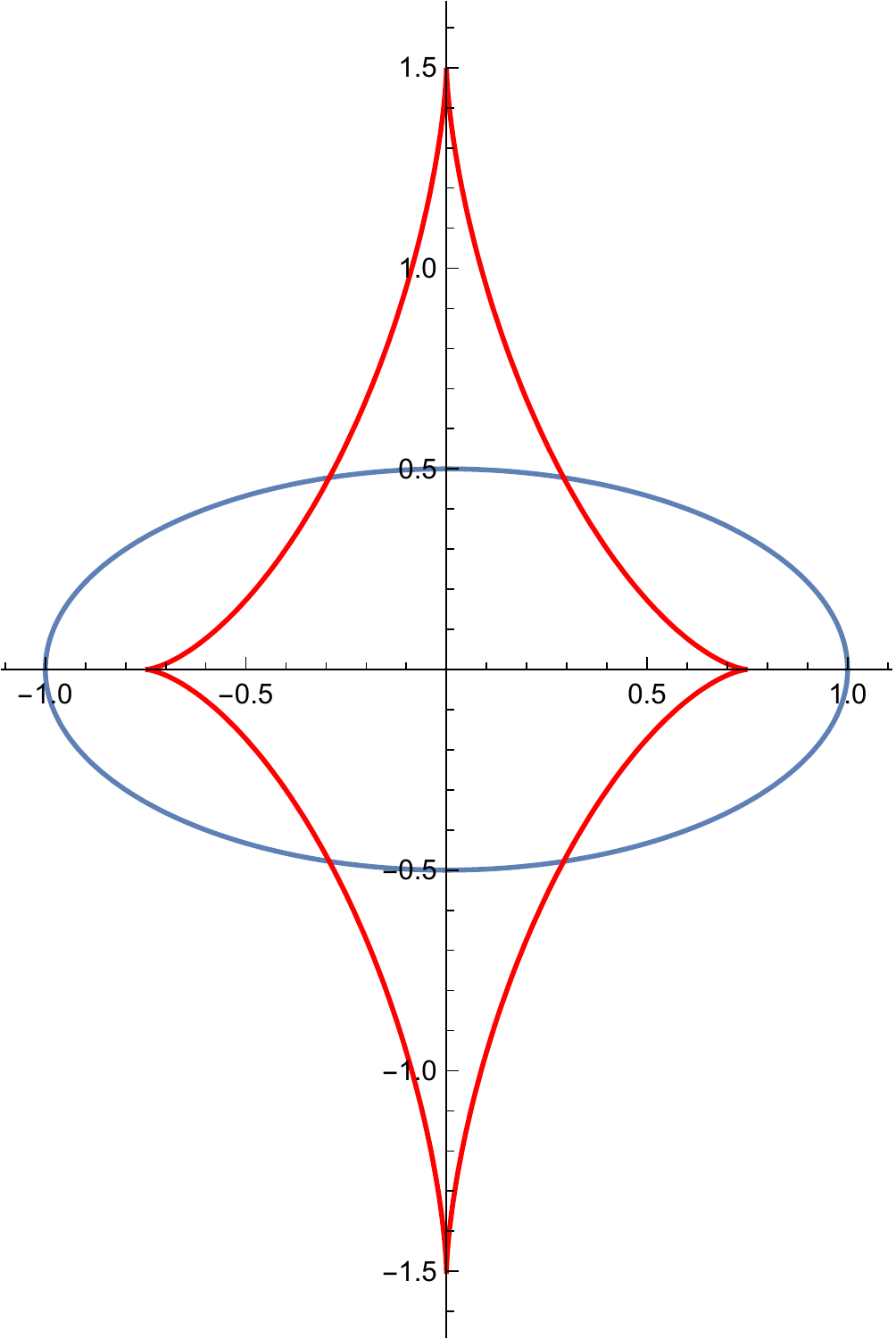}
        \caption{Evolute of an ellipse.} \label{fig:conic1b}
     \end{subfigure}
     \hfill
         \begin{subfigure}[t]{0.3\textwidth}
         \centering
          \includegraphics[scale=0.28]{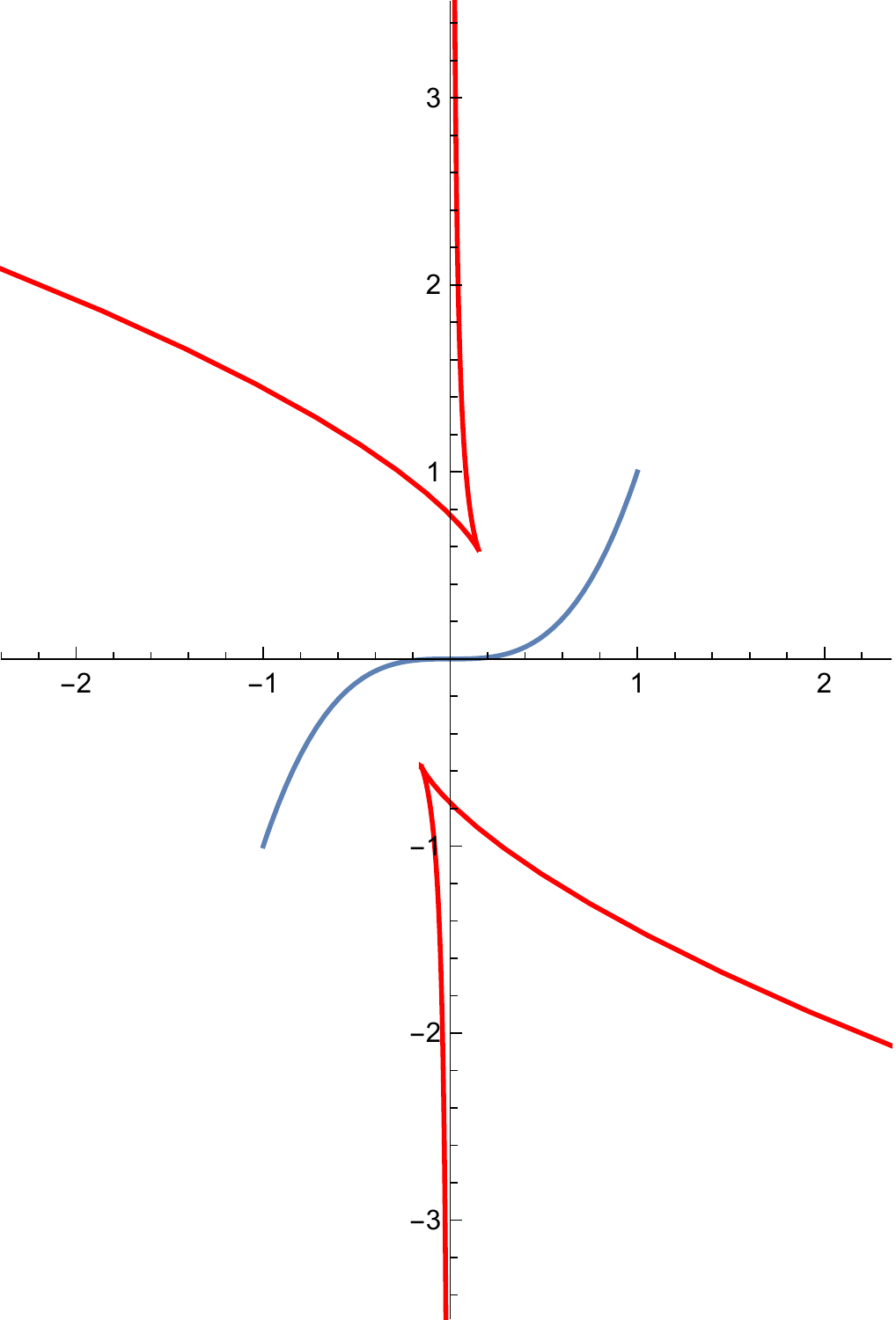}
        \caption[short]{Behavior of the evolute near a simple inflection point.}
     \label{fig:x^3}
     \end{subfigure}
     
       \caption{First examples of evolutes.}
\label{fig:one}
\end{figure}

If $\Ga$ has an inflection point, then the curvature radius becomes infinite, which means that the evolute $E_\Ga$ goes to infinity with its asymptote being the straight line passing through the inflection point and orthogonal to $\Ga$ (i.e., the normal to $\Ga$ at the inflection point), see Fig.~\ref{fig:x^3}. In case the point is a higher order inflection point, this point at infinity will be a critical point of the curvature and give a cusp on the evolute (see Example IX in Subsection \ref{subs:higher}). Observe that if $\Ga$ is a rational algebraic curve, then the above recipe provides the global parametrization of $E_\Ga$.  

\smallskip
Given a plane curve $\Ga$, the alternative definition of its evolute $E_\Ga$ which will be particularly useful for us is that $E_\Ga$ is the envelope of the family of normals to $\Ga$, where a \emph{normal}  is an affine line perpendicular to (the tangent line to) $\Ga$ at some point of $\Ga$. In other words, each normal to $\Ga$ is a tangent line to $E_\Ga$ and each tangent to $E_\Ga$ is a normal to (the analytic continuation) of $\Ga$.  

From this definition it follows that  the evolute $E_\Ga\subset \bR^2$  is a \emph{caustic}, i.e.,  the critical locus  of the Lagrangian projection of the cotangent bundle $T^\ast \Ga\subset T^\ast \bR^2$ of the initial curve $\Ga$   to the (phase) plane $\bR^2$. This circumstance explains, in particular,  why typical  singularities of the evolutes behave differently from those of (generic) plane algebraic curves and why generic evolutes have no inflection points. 

\smallskip
For an affine or projective real curve $\Ga$, we denote by $\Ga^\bC$ its complexification. For an affine real curve $\Ga \subset \bR^2$, we let $\tilde\Ga\subset \bR P^2$ denote its projective closure and $\tilde \Ga ^\bC \subset \bC P^2$ the complexification of $\tilde \Ga$ (equal to the projective closure of $\Ga^\bC \subset \bC^2$).

\begin{definition} {\rm For a plane algebraic curve $\Ga\subset \bR^2 \subset \bR P^2$, define its \emph{ curve of normals} $\tilde N_\Ga\subset (\bR P^2)^\vee$ as the curve in the dual projective plane whose points are the normals of $\Ga$.  (We start with the quasiprojective curve $N_\Ga$ of all normals to the affine $\Ga$ and take its projective closure  in $(\bR P^2)^\vee$.)}
\end{definition} 

Similarly to the above, for a (locally) parameterized curve $\Ga(t)=(x(t),y(t))$ and $N_\Ga(t)=(u(t),v(t))$, one gets 
\begin{equation}\label{eq:basic4}
\left\{ u(t)=\frac{x^\prime(t)}{y^\prime(t)},\;v(t)=-\frac{x(t)x^\prime(t)+y(t)y^\prime(t)}{y^\prime(t)}\right\}.
\end{equation}
(Here we assume that the equation of the normal line to $\Ga$ at the point $(x(t),y(t))$ is taken in the standard form $y+u(t)x+v(t)=0$.)

\smallskip
If a curve $\Ga$ is given by an equation $f(x,y)=0$, then  the equation of its curve of normals can be obtained as follows. 
Consider the system 
\begin{equation}\label{eq:eqnorm}
\begin{cases}
f(x,y)=0\\
u=-\frac{f^\prime_y}{f^\prime_x}\\
v=\frac{xf^\prime_y-yf^\prime_x}{f^\prime_x}\\
\end{cases}
\end{equation}
defining the original curve and the coefficients of the family of normals. Eliminating the variables $(x,y)$ from \eqref{eq:eqnorm} one obtains a single algebraic equation defining the curve of normals in the variables $(u,v)$. 

By the above alternative definition, for a plane algebraic curve $\Ga$, we get that  $(\tilde N_\Ga)^\vee=\tilde E_\Ga$ where $\vee$ stands for the usual projective duality of the plane projective curves.  \medskip

Two types of  real singularities of the evolute/curve of normals have natural classical interpretations in terms of the original curve $\Ga$.  Recall that real nodes of real-algebraic curves are  subdivided into \emph{crunodes} and \emph{acnodes}, the former being transversal intersections of two real branches and the latter being transversal intersections of two complex conjugate branches. 
 
 Now observe that  a crunode of $N_\Ga$  (i.e., the real node with two real branches)  corresponds to a  \emph{diameter} of $\Ga$ which is a straight segment  connecting pairs of points on $\Ga$ and  which is perdendicular to the tangent lines to $\Ga$ at both endpoints. On the other hand, a real cusp of $E_\Ga$ (resp. an inflection point on $N_\Ga$) corresponds to  a \emph {vertex} of $\Ga$ which is, by definition,  a critical point of $\Ga$'s curvature.  

As we mentioned above, vertices of plane curves appear, for example, in the classical $4$-vertex theorem and its numerous generalizations.  Beautiful lower bounds on the number of diameters of plane curves, plane wavefronts as well as their higher dimensional generalizations have been obtained in  symplectic geometry, see e.g. \cites{Pu1, Pu2}. 
\medskip

To formulate the problems  we consider below, let us  recall the following notion which deserves to be better known \cite[Def.~1]{LSS}. 

 \begin{definition}{\rm Given a real algebraic hypersurface $H \subset  \bR^n$, we define its \emph{$\R$-degree} as the supremum of the cardinality of $H \cap L$ taken over all lines $L \subset  \bR^n$ such that $L$ intersects $H$ transversally.}  \end{definition}
 
In what follows, we denote  the $\R$-degree of $H$ by $\R \deg(H)$.  Obviously,  $\R \deg(H)\le \deg(H)$, where $\deg(H)$ is the usual degree of $H$. 
 \medskip
 
 In what follows,  we discuss  four real-algebraic questions related to the evolutes and curves of normals of  plane real-algebraic curves. 

\begin{problem}\label{R-deg EG-NG}
For a given positive integer $d$, what are the maximal possible $\R$-degrees of the evolute  $E_\Ga$ and of the curve of normals $N_\Ga$ where $\Ga$ runs over the set of all real-algebraic curves of degree $d$?
\end{problem}

\begin{problem}\label{cusps EGa}
For a given positive integer $d$, what is the maximal possible number of real cusps   on $E_\Ga$ where $\Ga$ runs over the set of all real-algebraic curves of degree $d$? In other words, what is the maximal possible number of vertices  a real-algebraic curve $\Ga$ of degree $d$ might have? 
\end{problem}

To make Problem~\ref{cusps EGa} well-defined  we have to assume that $\Ga$ does not have a circle as its irreducible component since in the latter case the answer is infinite. Note that since by  \emph{vertices} of $\Ga$ we understand the critical points of the curvature,  it is possible that a vertex can be a real cusp of the projective closure of $E_\Ga$. This happens, in particular,  when the curvature  vanishes at a critical point  (see Example IX in Subsection \ref{subs:higher}).

\begin{problem}\label{nodesNGa}
For a given positive integer $d$, what its the maximal possible number of crunodes  on  $N_\Ga$ where $\Ga$ runs over the set of all real-algebraic curves of degree $d$? In other words, what is the maximal possible number of diameters $\Ga$ might have? 
\end{problem}

Here   again  we have to assume that $\Ga$ does not have a circle as its irreducible component since in this case the answer is infinite. 

\begin{problem}\label{nodesEGa}
For a given positive integer $d$, what is the maximal possible number of crunodes on  $E_\Ga$ where $\Ga$ runs over the set of all real-algebraic curves of degree $d$? In other words, what is the maximal possible number of points in $\bR^2$ which are the centers for at least two distinct  (real) curvature circles of  $\Ga$? 
\end{problem}

\begin{remark}{\rm As we mentioned above, questions similar to  Problems~\ref{cusps EGa} and \ref{nodesNGa} have been  studied in classical differential geometry and symplectic geometry/topology. They can also be connected to the study of plane curves of constant breadth which has been carried out by such celebrities as L.~Euler, A.~Hurwitz, H.~Minkowski, and W.~Blaschke, see e.g. \cite{Ba} and references therein. To the best of our knowledge, Problem~\ref{nodesEGa} has not been  previously discussed in the literature.}
\end{remark}

Our results related to Problems~1--4 can be found in Sections \ref{sec:rdeg}--\ref{sec:Ecru} below. They are mostly obtained by using small deformations of line arrangements in the plane. In all cases the lower bounds for the above quantities  which we obtain are polynomials of the same degree as their complex counterparts. However, in all cases but one their leading coefficients are smaller than those of the complex answers. At the moment we do not know whether  our estimates are sharp even on the level of the leading coefficients.

\section{Various preliminaries}

\subsection{Basic complex algebraic facts} 
We first summarize some information about the evolute and the curve of normals mainly borrowed from the classical treatise \cite{Sa}.   

\begin{proposition}[see Art. 111, 112, p.~94--96, in \cite {Sa}]\label{prop:1} For an affine real-algebraic curve $\Ga\subset \bR^2$ of degree $d$, which is in general position with respect to the line at infinity and has 
 only $\delta$ nodes and $\kappa$ ordinary cusps as singularities, the curves $\tilde \Ga^\bC$, $\tilde E_\Ga^\bC$ and $\tilde N_\Ga^\bC$ are birationally equivalent.  The degree of $\tilde E_\Ga$ equals $3d(d-1)-6\delta-8\kappa$, and the degree of $\tilde N_\Ga$ equals $d^2-2\delta-3\kappa$. 
\end{proposition} 

The genericity assumption for the birationality can be substantially weakened (but not completely removed).

\begin{lemma}[see Art. 113, p.~96, in \cite {Sa}] For a generic affine real-algebraic curve $\Ga\subset \bR^2$ of degree $d$, $\tilde E_\Ga^\bC$ has no inflection points. 
\end{lemma}

\begin{proposition}[see  Art. 113, p.~97, of \cite{Sa}]\label{prop:cusps+nodes} For an affine real-algebraic curve $\Ga\subset \bR^2$ as in Proposition \ref{prop:1}, the only singularities of $\tilde E_\Ga^\bC$ and $\tilde N_\Ga^\bC$ are nodes and cusps, except that $\tilde N_\Ga^\bC$ has an ordinary $d$-uple point (corresponding to the line at infinity).    
The number $\kappa_E$ of cusps of $\tilde E_\Ga^\bC$  equals $3(d(2d-3)-4\delta-5\kappa)$. 

If $\Ga$ is nonsingular, the number $\delta_E$ of nodes of $\tilde E_\Ga^\bC$  equals $\frac{d}{2} (3d-5)(3d^2-d-6)$, and the number $\delta_N$ of nodes of $\tilde N_\Ga^\bC$  equals 
$\binom{d^2-1}{2}\binom{d-1}{2}-\binom{d}{2}=(d^2+d-4)\binom{d}{2}$. 
  The curve $\tilde N_\Ga^\bC$ has no cusps   (since $\tilde E_\Ga^\bC$ has no inflection points). 
\end{proposition} 

For the sake of completeness and for the convenience of our readers, we included in Appendix~\ref{sec:Ragni} below  (some) modern proofs of the above claims, further results, and discussions related to  the enumerative geometry of envelopes and evolutes over the field of complex numbers.

\subsection {Klein's formula}
For further use, let us  recall a classical real-algebraic result of F.~Klein, see \cites{Kl, Co}.

\begin{theorem}\label{th:Klein}
If a real-algebraic curve $\Ga$ has no singularities except nodes, cusps, bitangents and inflection points, then 
$$d+2\tau^\prime+i^\prime=d^\vee+2\delta^\prime+\kappa^\prime,$$
where $d$ is the degree, $\tau^\prime$ the number of conjugate tangents, $i^\prime$ the number of real inflections, $d^\vee$ is the class, i.e., the degree of the dual curve, $\delta^\prime$ the number of real conjugate points, and $\kappa^\prime$ the number of real cusps of $\Ga$. 
\end{theorem}

Klein's theorem was generalized by Schuh \cite{Sch}, see  \cites{Vi, Wa}. In particular, the beautiful paper \cite{Wa} contains the following result (usually referred to as Klein--Schuh's theorem)  together with its detailed proofs and references to the earlier literature.

\begin{theorem}\label{th:genKlein} For any real-algebraic plane curve $\Ga$, 
\[\textstyle d-d^\vee=\sum_p (m_p(\Ga)-r_{p}(\Ga))
-\sum_q (m_q(\Ga^\vee)-r_{q}(\Ga^\vee)),
\]
where $\Ga^\vee$ is the dual curve, $d$ is the degree and $d^\vee$ is the class of $\Ga$, $m_p(\Ga)$ (resp. $m_q(\Ga^\vee)$) is the multiplicity of a real singular point $p\in \Ga$ (resp. $q\in \Ga^\vee$), and $r_{p}(\Ga)$ (resp. $r_{q}(\Ga^\vee)$) is the number of local real branches of $\Ga$ at $p$ (resp.  of $\Ga^\vee$ at $q$). 
\end{theorem}

\begin{example}{\rm 
Let $\Gamma$ be the curve of normals of an ellipse. Then $\Gamma^\vee$ is the evolute of the ellipse. We have $d=4$ and $d^\vee =6$. The singularities of $\Gamma$ are two crunodes, corresponding to  the two diameters of the ellipse, and an acnode, corresponding to the line at infinity. The evolute $\Gamma^\vee$ has four real cusps and no other real singular points. This gives $4-6=2-4$, which checks with the formula, see Fig.~\ref{fig:conic1b}.}
\end{example}

\subsection{Brusotti's theorem and small deformations of real line arrangements} Let $\Ga\subset \bC^2$ be any (possibly reducible) plane real-algebraic curve  with only real and complex nodes as singularities. Denote by $\Ga_\bR\subset \bR^2$ the real part of $\Ga$. Observe that there are two topological types of small real deformations/smoothenings of a crunode. Namely,  being a transversal intersection of two smooth real branches a crunode has locally four connected components of its complement. Under a small real deformation  of a crunode either one or the other pair of opposite connected components will merge forming a single component. Similarly, there exist two topological types of small real deformations/smoothenings of a acnode.  Either it disappears in the complex domain or a small real oval will be created near the acnode. The following useful result was proved by  Brusotti in  \cite{Br}. 

\begin{theorem}\label{th:brusotti}
 Any real-algebraic curve $\Ga\subset \bC^2$ with only nodes as singularities  
admits a small real deformation of the same degree which realizes any collection of independently prescribed smoothing types 
of  its real nodes.
\end{theorem} 

 Thus there are $2^m$ topological types of such small deformations where $m$ is the number of real nodes of the real curve under consideration.  The easiest way to think about different types of small deformations is to fix a real bivariate polynomial $G(x,y)$ of minimal degree defining $\Ga$ as its zero locus and to assign a $\pm$-sign to each real node.  Then the  sign pattern uniquely determines the way to resolve each crunode and acnode as follows. If $v$ is a crunode of $\Ga$, then $G(x,y)$ has alternating signs in four local components of the complement to $\Ga$ near $v$. If we assign the sign $+$ to the node $v$, then under this resolution two  opposite local components in which $G(x,y)$ is positive  should glue together, and if we assign the sign $-$ then the other two opposite components (in which $G(x,y)$ is negative) should glue together.  Analogously, if $G(x,y)$ has a local maximum at the acnode and  we assign $+$, then we create a small oval. Otherwise the acnode disappears. For  a local minimum,  the signs exchange their roles.   The next statement follows from \cite{Br} and can be also found in \cite[p.~13]{Gu} and \cite{Gu2}.

\begin{proposition} \label{prop:signs} For each plane real-algebraic curve $\Ga$ of a given degree $d$ with only nodes as singularities and a given sign pattern at its crunodes and acnodes,  there exists a real bivariate polynomial of degree at most $d$ having the chosen sign at each crunode/acnode.  
\end{proposition}

\begin{corollary}\label{cor:signs}In the above notation, if $\Ga$ is the zero locus of a  real polynomial $G(x,y)$ and $H(x,y)$ is a polynomial realizing a chosen sign pattern for the crunodes/acnodes of $\Ga$ and such that $\Ga$ and the curve $H(x,y)=0$ have no common singularities (including the line at infinity), then there exists $\epsilon_{G,H}>0$ such that for any $0<\epsilon\le \epsilon (G,H)$, the curve given by $G(x,y)+\epsilon H(x,y)=0$ is non-singular and realizes the prescribed smoothening type of all crunodes/acnodes of $\Ga$. 
\end{corollary} 

We will mainly be applying Brusotti's theorem to \emph{generic} real line arrangements in $\bR^2$, where ``generic" means that no two lines are parallell and no  three lines intersect at one point.  In this case, one has only crunodes among the real nodes. The following claim will be useful in our considerations.

\begin {proposition}\label{lm:smallcomp} Given a line arrangement $\cA\subset \bR^2$ of degree $d$, a vertex (crunode) $v$ and a real polynomial $H(x,y)$ of degree at most $d$  which does not vanish at $v$, consider a sufficiently small disk $\mathbb D\subset \bR^2$ centered at $v$ which neither intersects the real zero locus of $H(x,y)$ nor any of the lines of $\cA$ except for those two whose intersection point is $v$. Fix additionally a product $G(x,y)$ of linear forms whose zero locus is $\cA$. Then there exists $\beta (G,H,v)>0$ such that for any $0<\epsilon^2 \le \beta(G,H,v)$, the curve given by $G(x,y)+\epsilon^2 H(x,y)=0$ restricted to $\mathbb D$ has the following properties: 
\begin{itemize}
\item[\rm (i)] it consists of two smooth real branches without inflection points;
\item[\rm (ii)] each branch contains a unique vertex, i.e., a critical point of the curvature. At this point the curvature attains its maximum.  
\end{itemize}
\end{proposition}

\begin{figure}
\begin{center}
\includegraphics[scale=0.2]{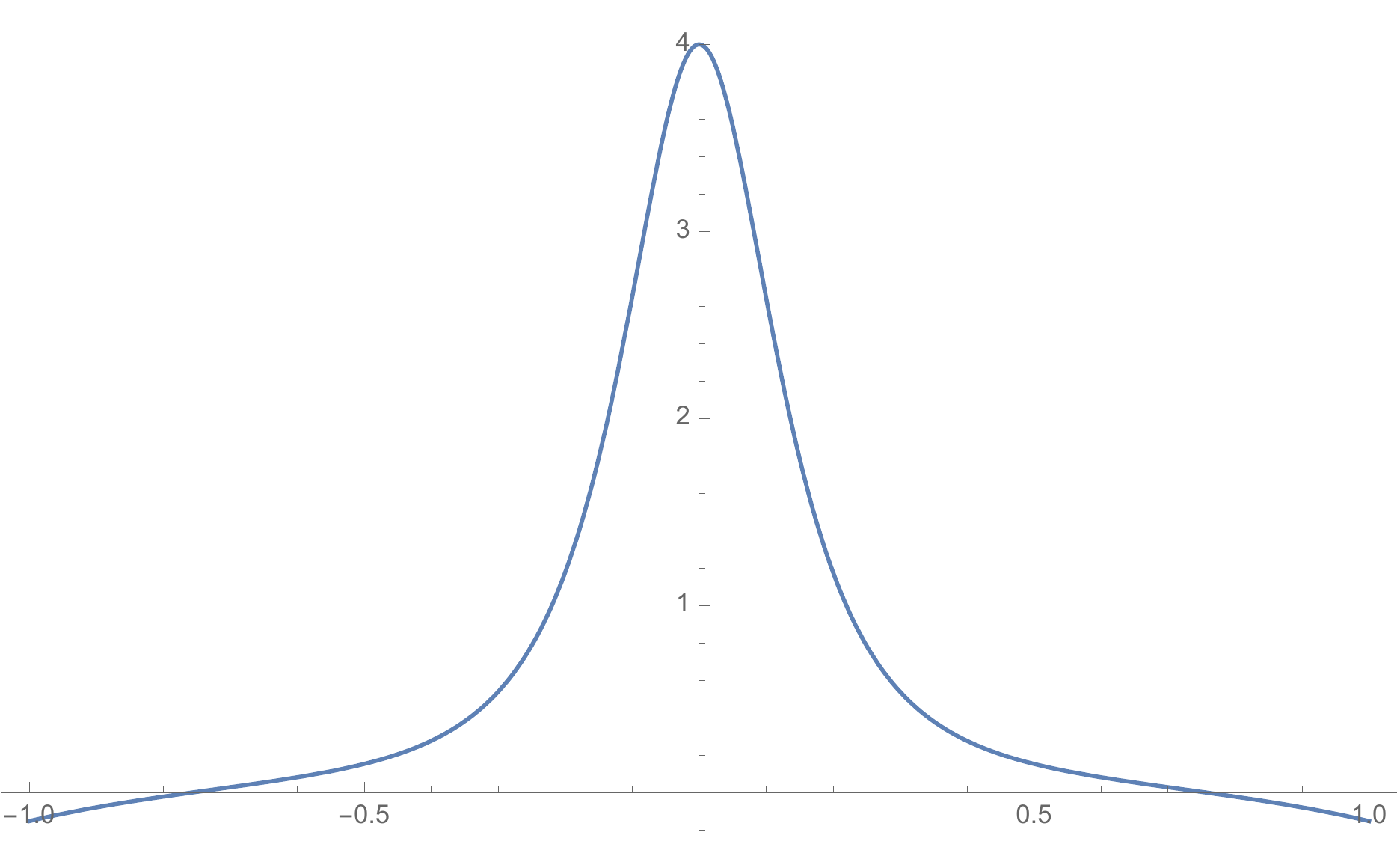}\hskip0.5cm
\includegraphics[scale=0.2]{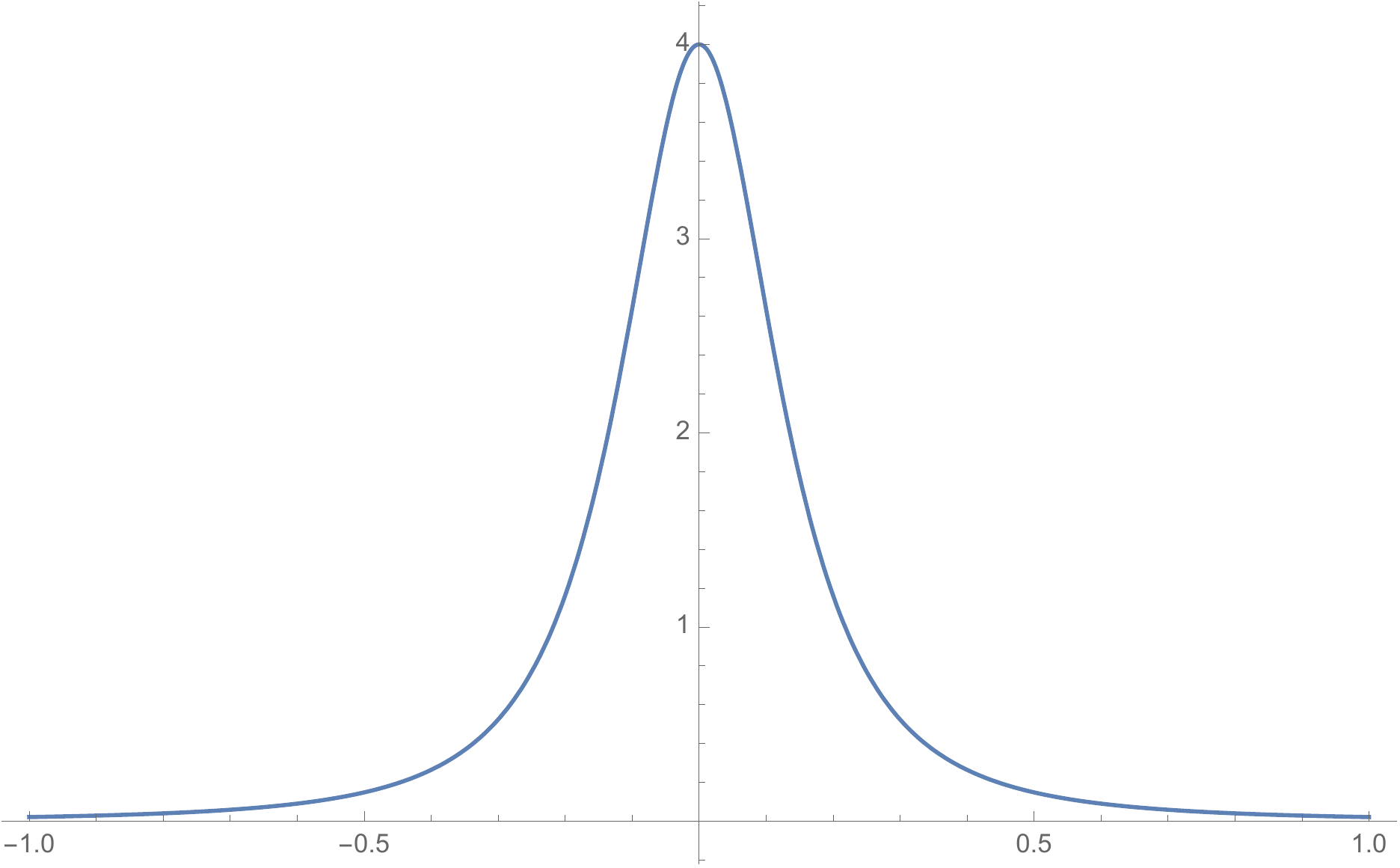}\hskip0.5cm
\includegraphics[scale=0.2]{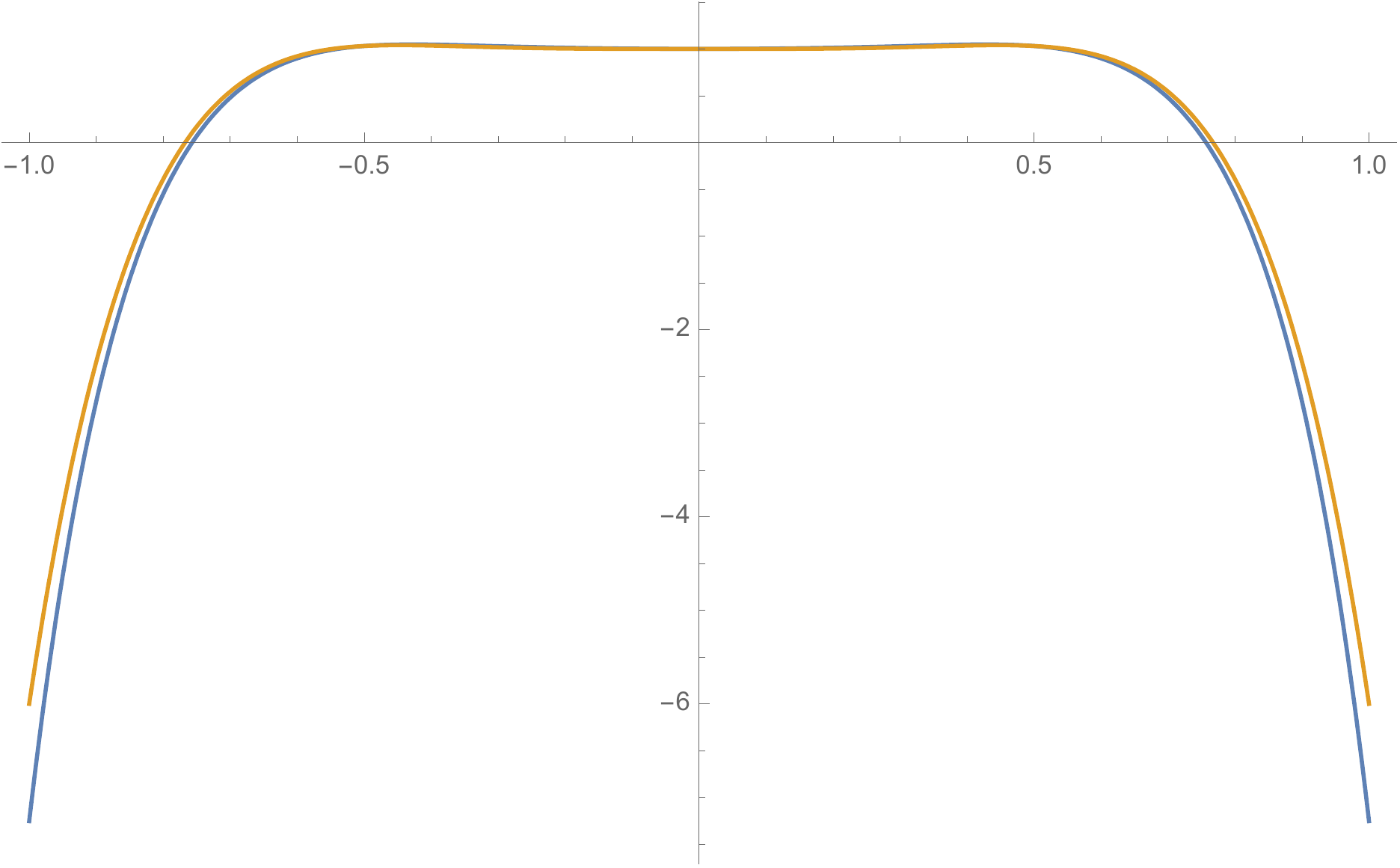}
\end{center}

\caption{The leftmost plot shows the curvature of $y_\eps=\sqrt{x^2+\eps^2(1+x^4-x^6)}$ for $\eps=1/4$, the central plot shows the standard curvature for the same $\eps$ and the rightmost plot shows their quotient together with the theoretical asymptotic quotient $\Psi(x)$, see Lemma~\ref{lm:curvature}.} \label{fig:Curvature}
\end{figure}

The proof of Proposition~\ref{lm:smallcomp}  is based on the next analytic lemma. Consider a family of functions $y_\eps(x)=\sqrt{x^2+\eps^2(1+O(x,\eps))}$, where $O(x,\eps)$ is a function vanishing at the origin and real-analytic in the variables $(x,\eps)$ in some open neighborhood of it.  Therefore $1+O(x,\eps)$ is positive in some sufficiently small fixed rectangle  $x\in [-A,A],\; \eps\in [-e,e]$. Here $\eps$ is  a small real parameter and for each fixed $\eps$, we think of $y_\eps(x)$ as a function of $x\in [-A,A]$. 

Let $K_\eps(x)=\frac{y^{\prime\prime}}{(1+(y^\prime_\eps(x))^2)^{3/2}}$ be the signed curvature of the function $y_\eps(x)$ and  let 
$k_\eps(x)=\frac{\eps^2}{(\eps^2+2x^2)^{3/2}}$ be the signed curvature of the (upper branch of) hyperbola $y= \sqrt{x^2+\eps^2}$, $x\in \bR$. (Notice that 
$k_\eps(x)>0$ for all $x\in \bR$ and that $\int_{-\infty}^\infty k_\eps(x)dx=\sqrt{2}$ for all  $\eps\in \bR$ with $\eps\neq 0$.)

\begin{lemma}\label{lm:curvature} In the above notation, when $\eps\to 0$,  then on the interval $[-A,A]$ the quotient $\frac{K_\eps(x)}{k_\eps(x)}$ uniformly converges   to the function 
\[\textstyle \Psi(x)=1+O(x,0)-xO^\prime(x,0)+\frac {1}2 x^2 O^{\prime\prime}(x,0).\]
\end{lemma}

The statement of Lemma~\ref{lm:curvature}  is illustrated in Fig.~\ref{fig:Curvature}. 

\begin{proof} Set $\Phi(x):=x^2+\eps^2(1+O(x,\eps))$. With this notation we get
\[y^\prime_\eps=\frac{\Phi^\prime}{2\Phi^{1/2}},\quad y^{\prime\prime}_\eps=\frac{2\Phi^{\prime\prime}\Phi-(\Phi^\prime)^2}{4\Phi^{3/2}},\quad 
K_\eps=\frac{y^{\prime\prime}}{(1+(y^\prime_\eps(x))^2)^{3/2}}=2\frac {2\Phi^{\prime\prime}\Phi-(\Phi^\prime)^2}  {(4\Phi+(\Phi^\prime)^2)^{3/2}}.
\]
(Recall that all derivatives are taken with respect to the variable $x$.) 
Substituting the expressions for $\Phi, \Phi^\prime$ and $\Phi^{\prime\prime}$ in the third formula, we obtain
\[K_\eps(x)=2\eps^2 \frac {4+4O+2x^2O^{\prime\prime}+2\eps^2O^{\prime\prime}+2\eps^2 O O^{\prime\prime}-4x O^\prime-\eps^2(O^\prime)^2}{(4\eps^2+8x^2+4\eps^2 O+4\eps^2x O^\prime+\eps^4(O^\prime)^2)^{3/2}}.
\]
Dividing by $k_\eps(x)$ we get
\begin{equation}\label{eq:eqKeps}
\frac{K_\eps(x)}{k_\eps(x)}=2\frac {(4+4O+2x^2O^{\prime\prime}+2\eps^2O^{\prime\prime}+2\eps^2 O O^{\prime\prime}-4x O^\prime-\eps^2(O^\prime)^2 )(\eps^2+2x^2)^{3/2}}{(4\eps^2+8x^2+4\eps^2 O+4\eps^2x O^\prime+\eps^4(O^\prime)^2)^{3/2}}.
\end{equation}
To prove the pointwise convergence, assume that $x\neq 0$ and let $\eps\to 0$. Then using real analyticity of $O(x,\eps)$ and $\Phi(x,\eps)$, one gets
\begin{align*}
\lim_{\eps\to 0} \frac{K_\eps(x)}{k_\eps(x)}=&2 \frac {(4+4O(x,0)-4xO^\prime(x,0)+2x^2O^{\prime\prime}(x,0)) (2x^2)^{3/2}} {(8x^2)^{3/2}}\\
=&1+O(x,0)-xO^\prime(x,0)+\frac{x^2O^{\prime\prime}(x,0)}{2}=\Psi(x).
\end{align*}
For $x=0$, the function $\Psi(x)$ continues by analyticity implying that $\Psi(0)=1$ since $O(0,0)=0$.  

\smallskip
To prove the uniform convergence, consider again \eqref{eq:eqKeps}. The factor \[(4+4O+2x^2O^{\prime\prime}+2\eps^2O^{\prime\prime}+2\eps^2 O O^{\prime\prime}-4x O^\prime-\eps^2(O^\prime)^2 )\]
in the right-hand side is real analytic and non-vanishing at the origin. Thus  if we can prove that the remaining  factor 
\[\frac{(\eps^2+2x^2)^{3/2}}{(4\eps^2+8x^2+4\eps^2 O+4\eps^2x O^\prime+\eps^4(O^\prime)^2)^{3/2}}\]
 in the right-hand side of \eqref{eq:eqKeps} is real analytic and non-vanishing at the origin, the uniform convergence will follow. Let us consider the inverse 
$\left(\frac{4\eps^2+8x^2+4\eps^2 O+4\eps^2x O^\prime+\eps^4(O^\prime)^2}{\eps^2+2x^2}\right)^{3/2}$ of the latter expression. To this end, it suffices to show that 
\[\kappa(x,\eps)=\frac{4\eps^2+8x^2+4\eps^2 O+4\eps^2x O^\prime+\eps^4(O^\prime)^2}{\eps^2+2x^2}=4\frac{\eps^2+2x^2+\eps^2 O+\eps^2x O^\prime+\eps^4(O^\prime)^2/4}{\eps^2+2x^2}\]
has a non-vanishing limit at the origin.  To prove this claim, consider 
\begin{multline*}
\lim_{(x,\eps)\to (0,0)}\frac{\eps^2+2x^2+\eps^2 O+\eps^2x O^\prime+\eps^4(O^\prime)^2/4}{\eps^2+2x^2}\\
=\lim_{\sqrt{\eps^2+2x^2}\to 0}\left(1+\frac{\eps^2 O}{\eps^2+2x^2}+\frac{\eps^2x O^\prime}{\eps^2+2x^2}+\frac{\eps^4(O^\prime)^2/4}{\eps^2+2x^2}\right).
\end{multline*}
Since $O(x,\eps)$ and $x$ vanish at the origin, the limits at the origin of the second and third terms in the right-hand side of the latter expression  are $0$. The last term also has limit $0$, since it contains $\eps^4$ in the numerator. Thus the limit of  $\kappa(x,\eps)$ at the origin exists and equals $4$. The result follows.  
\end{proof} 

\begin{remark} \label{rm:rm2}{\rm A  rather simple rescaling shows that a statement similar to Lemma~\ref{lm:curvature} holds  in a more general family of curves $y_\eps(x)=a\sqrt{x^2+\eps^2(1+O(x,\eps))}$
for the quotient of the curvature $K_\eps(x)$  and 
the standard curvature  $k_\eps(x)= \frac{a\eps^2} {(\eps^2+(1+a^2)x^2)^{3/2}}$ of (the branch of) the hyperbola $y=a\sqrt{x^2+\eps^2}$, where $a$ is a fixed positive constant. The original case corresponds to $a=1$. } 
\end{remark}  

\begin{proof}[Proof of Proposition~\ref{lm:smallcomp}]   In the above notation,  consider the family of real curves $\Upsilon_\eps: \{ G(x,y)+\eps^2 H(x,y)=0 \}$.  In a small neighborhood $\mathbb D$ of the chosen vertex $v$ and for sufficiently small $\eps$, the restriction of  $\Upsilon_\eps$ to $\mathbb D$ is a desingularization of the crunode at $v$ of a chosen type, since $\eps>0$ and $H(x,y)$ does not vanish av $v$. Let us choose affine  coordinates $(\tilde x, \tilde y)$ centered at $v$ which are  obtained from the initial coordinates $(x,y)$ by a translation and rotation such that the pair of lines belonging to $\cA$  and intersecting at $v$ will be given by $\tilde y=\pm a\tilde x$ for some positive $a$. These coordinates  will be uniquely defined if we additionally require that  in a small neighborhood of $v$ the family $\Upsilon_\eps$ will be close to the family of hyperbolas $\tilde y=\pm a\sqrt {\tilde x^2+\eps^2}$.  Now in these coordinates $(\tilde x, \tilde y)$, the curve $\Upsilon_\eps$ will be given by 
\[(\tilde y^2-a^2\tilde x^2)L(\tilde x, \tilde y)+\eps^2 H(\tilde x, \tilde y)=0.\] 
Here (after possible rescaling of the equation and parameter $\eps$) we can assume that 
$L(\tilde x,\tilde y)=1+ \dots$ and $H(\tilde x, \tilde y)=1 + \dots$, where $\dots$ stand for higher order terms. Notice that $L(\tilde x,\tilde y)$ is the product of linear forms defining the 
 lines of $\cA$ other than the two given by $(\tilde y^2-a^2\tilde x^2)$. From the latter equation we obtain 
\begin{equation}\label{eq:extra}
\tilde y^2=a^2\tilde x^2+\eps^2\frac {1 + C(\tilde x, \tilde y)}{1 + D(\tilde x, \tilde y)}, 
\end{equation}
where $C(\tilde x, \tilde y)$ and $D(\tilde x, \tilde y)$ are real analytic functions vanishing at the origin. 
Expanding the functional coefficient at $\eps^2$  in \eqref{eq:extra} near the origin in coordinates $(\tilde x, \tilde y)$ and using the implicit function theorem, we get that \eqref{eq:extra} determines the family of curves $\tilde y_\eps = \pm a \sqrt{ \tilde x^2+\eps^2(1+O(\tilde x,\eps))}$. Here $O(\tilde x,\eps)$ is a well-defined  function which  vanishes at the origin and is real analytic in some neighborhood of it.  Dropping tildas we see that the case $y_\eps =  \sqrt{x^2+\eps^2(1+O(x,\eps))}$ is exactly the one treated in Lemma~\ref{lm:curvature} and the case $a>0$ is mentioned in Remark~\ref{rm:rm2}. The case  $a<0$ is obtained from the case $a>0$ by a trivial variable change. 

Furthermore, by Lemma~\ref{lm:curvature}, the curvature $K_\eps(x)$ of $\Upsilon_\eps$ is strictly positive for all sufficiently small $\eps$ and $x$ lying in an a priori  fixed small neighborhood of $v$. Thus $\Upsilon_\eps$  has no inflection points in this neighborhood. Finally, near the origin and for small $\eps$,  $K_\eps(x)$ behaves very much like the standard curvature $k_\eps(x)$ since (i) their quotient tends to a positive constant at $0$;  (ii) for $\eps\to 0$, the standard curvature $k_\eps(x)$ tends to $\sqrt{2} \delta (0)$ ,  where $\delta(0)$ is Dirac's delta function supported at the origin. Therefore,  for all sufficiently small $\eps$,  $K_\eps(x)$ has a unique maximum near the origin.  Finally, for $a>0$, the situation is completely parallel. 
 \end{proof}

To move further, we need more definitions.  By an \emph{edge}  of a real line arrangement $\cA\subset \bR^2$ we mean a connected component of $\cA\setminus V(\cA)$ where $V(\cA)$ is the set of its vertices (nodes). An edge is called \emph{bounded} if both its endpoints are vertices, and \emph{unbounded} otherwise.  Given a small resolution $\RR$ of $\cA$ and  a bounded edge $e\in \cA$, we denote by $\RR_e\subset \RR$ the restriction of $\RR$ onto a sufficiently small neighborhood $U_e\subset \bR^2$ of $e$. Obviously, for any small resolution $\RR$ of $\cA$, $\RR_e$ consists of three connected components, see  Fig.~\ref{pic:edge}. We say that $\RR$ \emph{respects} $e$ if each of the three connected components of $\RR_e$   is close to the union of edges bounding a single connected domain of $U_e\setminus \cA$, see Fig.~\ref{pic:edge}\,(left); otherwise  we say that $\RR$ \emph{twists} the edge $e$, see Fig.~\ref{pic:edge}\,(right). In the first case,  the edge $e$ is called \emph{respected} by 
 $\RR$ and in the second case \emph{twisted} by  $\RR$.  The  three connected components of $\RR_e$ are divided into two \emph{short} and one \emph{long} where the long one is stretched along $e$ and each of the two short ones is completely located near the respective vertex of $e$. 
\smallskip 
Proposition~\ref{lm:smallcomp} has the following consequence.

\begin{figure}[t]
     \centering
     \begin{subfigure}[t]{0.49\textwidth}
         \centering
      \includegraphics[scale=0.35]{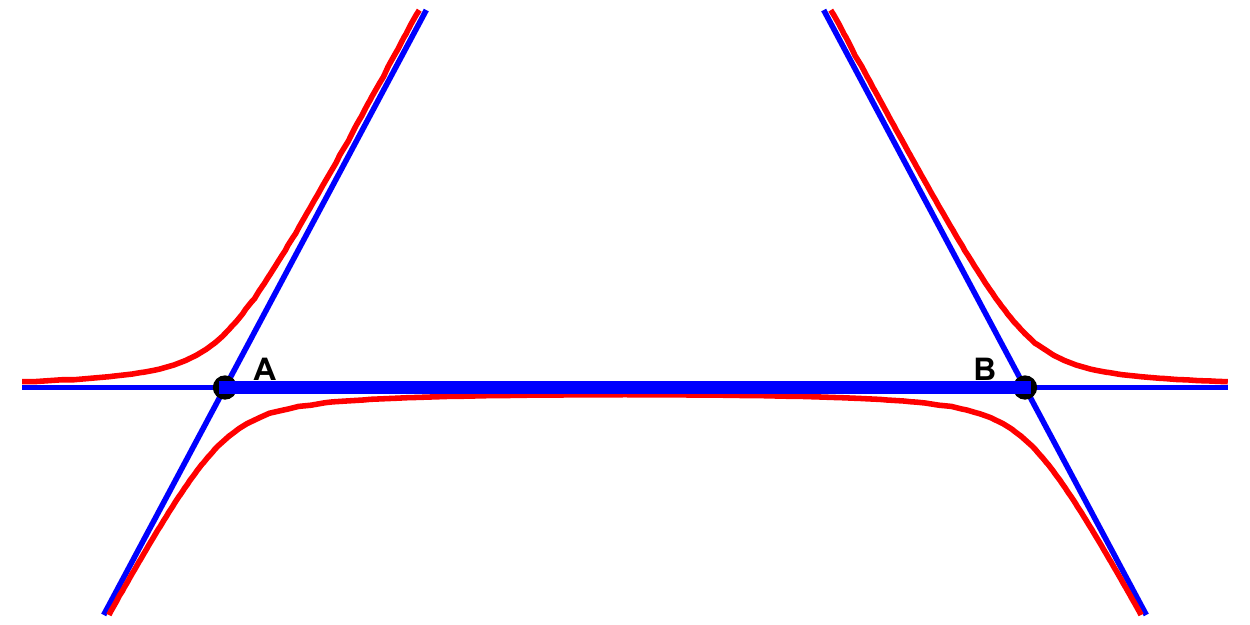} 
       \caption{The resolution respects the the bounded edge.}\label{pic:edgeleft}
        \label{fig:conic1a}
     \end{subfigure}
     \hfill
     \begin{subfigure}[t]{0.49999\textwidth}
         \centering
        \includegraphics[scale=0.35]{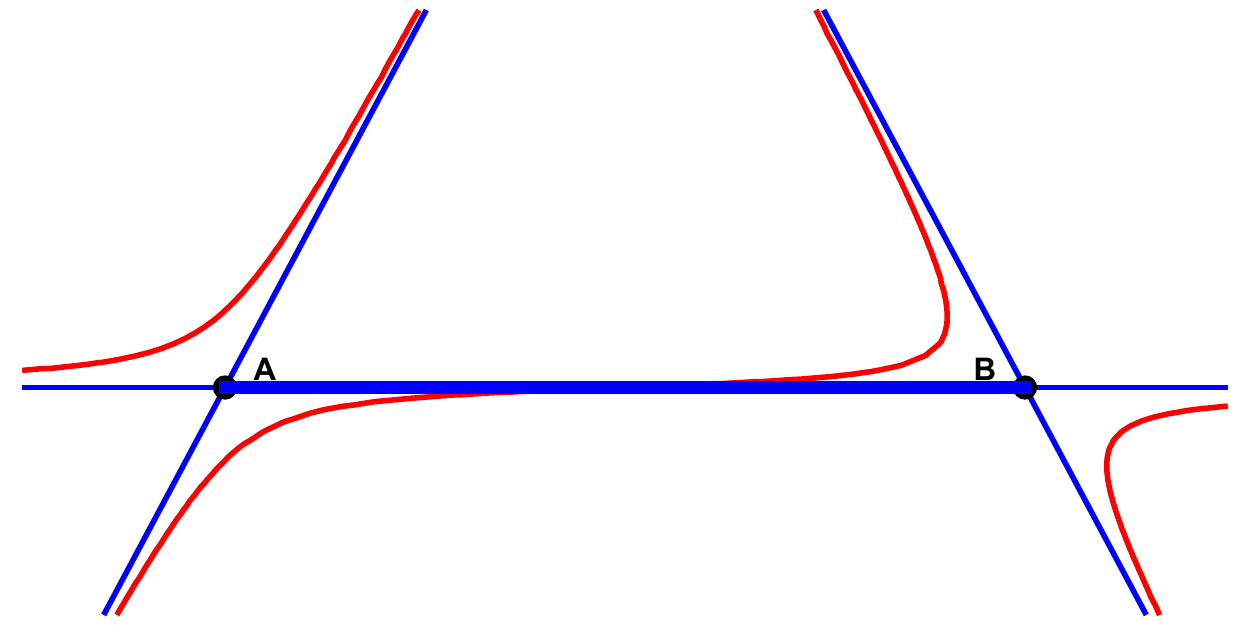}
        \caption{The resolution twists the the bounded edge.}\label{pic:edgeright}   
     \end{subfigure} \hfill
     \caption{Small resolutions of a line arrangement around the bounded edge $AB$.} \label{pic:edge}
\end{figure}

\begin{corollary}\label{lm:edges} Given a generic line arrangement $\cA\subset \bR^2$ and its sufficiently small real deformation $\RR$,  the following facts hold: 
\begin{itemize}
\item[\rm (i)] If $\RR$  respects the edge $e$, then the short components of $\RR_e$ have no inflection points while the long component has an even number inflection points (possibly none), see  Fig.~\ref{pic:edge} (left);
\item[\rm (ii)]  If $\RR$  twists $e$, then the short  components of $\RR_e$ have no inflection points while the long component has an odd number of inflection points, see Fig.~\ref{pic:edge} (right); 
\item[\rm (iii)]  If $\RR$  respects the edge $e$, then there are at least five extrema of curvature on the three components of $\RR_e$; namely  one maximum on each of the short components and  two maxima close to the vertices,  plus an odd number of additional extrema on the long component; 
\item[\rm (iv)]  If $\RR$  twists the edge $e$, then there are at least four extrema of curvature on the three components of $\RR_e$; namely,  one maximum on each short component and  two maxima (of the absolute value of the curvature) on the long component close to the vertices. 
\end{itemize}
\end{corollary}

\begin{proof} To settle (i) and (ii), we notice that by Proposition~\ref{lm:smallcomp} neither the short  nor the long components of $\RR_e$  have  inflection points near  the vertices of $e$.  Furthermore, if we parameterize the long component as $(x(t),y(t))$ with $(x^\prime)^2(t) +(y^\prime)^2(t)$ non-vanishing for all $t$, then the signed curvature  
\[k(t)=\frac{ x^\prime(t) y^{\prime\prime}(t)- y^\prime(t) x^{\prime\prime}(t)} {((x^\prime(t))^2+(y^\prime(t))^2)^{3/2}}\]
will have the same sign near the vertices  if $\RR$ preserves $e$ and will change the sign if $\RR$ twists $e$. Therefore the long segment acquires an even number of inflections counting multiplicities in the first case and an odd number of inflections in the  second case.

To settle (iii) and (iv), observe that by Proposition~\ref{lm:smallcomp}  both in the case when $\RR$ preserves $e$ or twists $e$,  $\RR_e$ will have  four maxima of the absolute value of the curvature near the vertices of $e$, 
one on each short component and two on the long component. Additionally, if $\RR$ preserves $e$, then the long segment contains an even number of inflection points. If there are no inflection points at all (as in the case shown in Fig.~ \ref{pic:edge} (left)) then there is at least one more minimum of the curvature on the long component between the two maxima. In other words, between any two consecutive inflection points on the long component there must be at least one maximum of the absolute value of the curvature. 
\end{proof}

\section{On the maximal $\R$-degrees of the evolute and of the curve of normals}\label{sec:rdeg}

Recall that the degree of the evolute of a generic curve of degree $d$ equals $3d(d-1)$, see Proposition~\ref{prop:1}. Already consideration of the first non-trival case of plane conics shows  that the question about the maximal $\bR$-degree of the evolute (the first part of Problem~\ref{R-deg EG-NG})   is non-trivial. Namely, if $\Ga$ is a generic real conic then the usual degree of $E_\Ga$ equals $6$ while the $\bR$-degree of its evolute (which for a generic conic $\Ga$ is an astroid) equals $4$, see Lemma~\ref{lm:astroid} below, Fig.~\ref{fig:one} b) and Fig.~\ref{fig:hypRot}. Our initial result in this direction is as follows. 
  
\begin{proposition}\label{prop:Rdegevol} For any $d\ge 3$, 
the maximal $\bR$-degree for the evolutes of algebraic curves of degree $d$ is not less than $d(d-2)$. 
\end{proposition}

\begin{proof} Recall that at each real inflection point of a real-algebraic curve $\Ga$, its evolute $E_\Ga$  goes to infinity and its asymptote at infinity is the line normal to $\Ga$ and passing through the respective inflection point. Notice that  Klein's formula combined with the usual Pl\"ucker formula for non-singular plane curves   imply that a real-algebraic curve of degree $d$ has at most one third of its total number of inflection points being real and this bound is sharp, see \cite{Kl}.  (The sharpness can be obtained by considering  small deformations of real line arrangements). The number of complex inflection points of a generic smooth plane curve of degree $d$ equals $3d(d-2)$. Thus there exists a smooth real-algebraic curve of degree $d$ with $d(d-2)$ real inflection points. The evolute of such curve hits the line at infinity (transversally) at $d(d-2)$ real points. Therefore the evolute intersects any affine line sufficiently close to the line at infinity $d(d-2)$ times as well. Thus the maximal $\bR$-degree among the evolutes of  plane curve of degree $d$  is at least $d(d-2)$. 
\end{proof} 

\begin{remark}{\rm At least for small $d$, the above lower bound is not sharp. Namely, for $d=2$, the maximal $\bR$-degree equals $4$, while the above bound is not applicable. For $d=3$, taking a usual cubic in the Weierstrass form, one has  an example of the evolute whose $\bR$-degree  is greater than or equal to  $6$, see Fig.~\ref{fig:WeierNC}. (Notice  the number of real inflections of a nonsingular real cubic  always equals $3$.)}
\end{remark}

Our second result solves the second part of Problem~\ref{R-deg EG-NG} about the maximal $\R$-degree of the curve of normals.   

\begin{proposition}\label{lm:Shustin} There exists a real-algebraic curve $\Ga$ of degree $d$ and a point  $z\in \bR^2$ such that all $d^2$ complex normals to $\Ga$ passing through $z$ are, in fact,  real. In other words, there exists $\Ga$ such that the maximal $\R$-degree of $N_\Ga$ equals $d^2$, which coincides with  its complex degree.  
\end{proposition}

\begin{proof}  
 Recall that a crunode (which is a transversal intersection of two smooth real local branches) 
admits two types of real smoothing. Given a crunode $c$ and a point $z$ such that the straight segment $L$ connecting $z$ with $c$ is not tangent to the real local branches at the curve at $c$, there exists a 
smoothing of the curve at $c$  such that 
 one obtains two real normals to this smoothing passing through $z$ and close to $L$, see illustration in Fig.~\ref{pic:brusotti}. 
 Now take an arrangement $\cA\subset \bR^2$ of $d$ real lines in general position and a point $z$ outside these lines.
By Brusotti's theorem, smoothing  
all $d(d-1)/2$ nodes in the admissible way shown in Fig.~\ref{pic:brusotti} we obtain $d(d-1)$ normals close to the
straight segments joining $z$ with the nodes of $\cA$. Additional $d$
normals are obtained by small deformations  of the altitudes connecting $z$ with  each of the $d$ given 
lines. Thus, for a small deformation of $\cA$ that resolves all its nodes in the admissible way with respect to $z$,  one gets $d^2$ real normals to the obtained curve  through the point $z$ which implies that the $\R$-degree of its curve of normals  is a least $d^2$. But the usual complex degree of this curve of normals is (at most) $d^2$. The result follows.  
\end{proof} 

\begin{figure}
  \centering
        \includegraphics[scale=0.44]{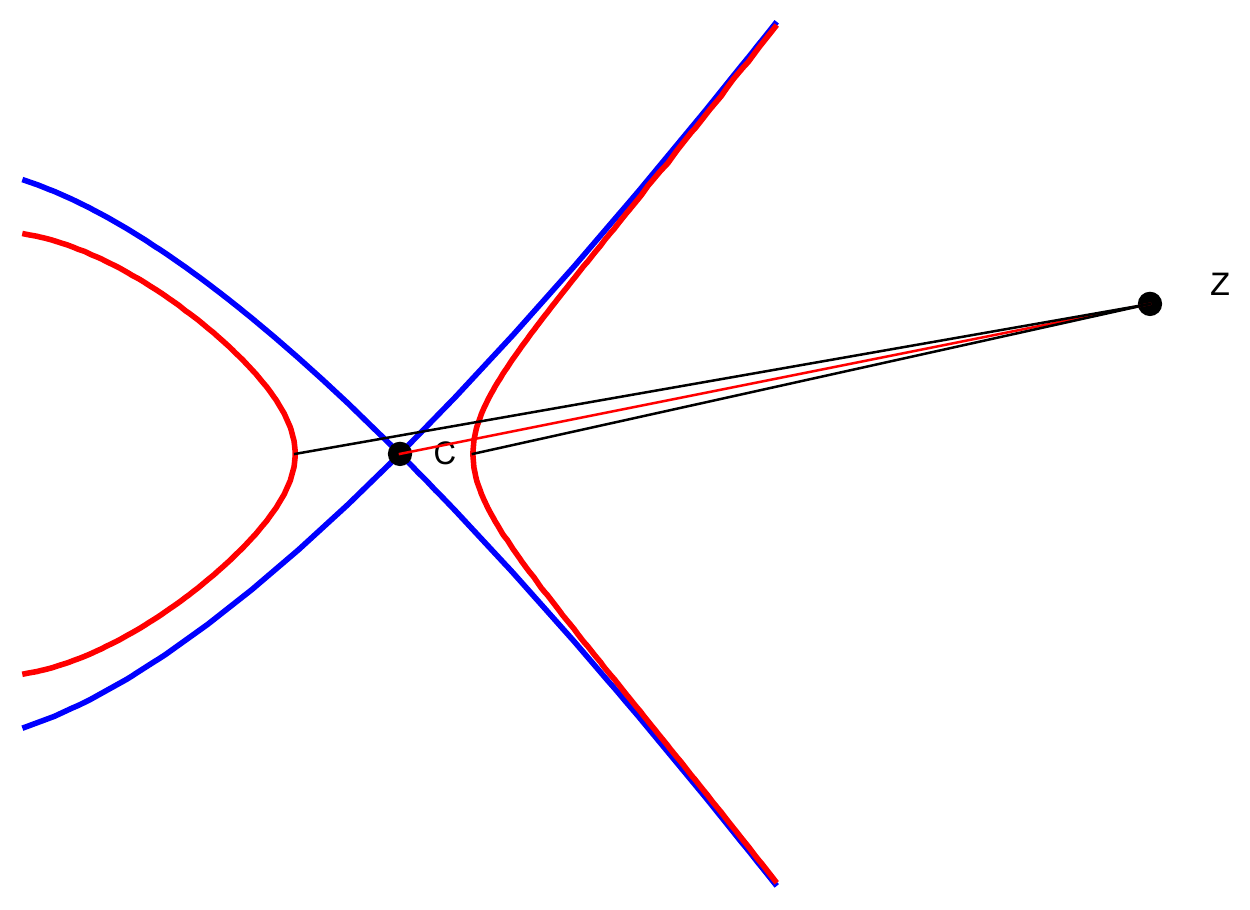}

\caption{An admissible resolution (in red) of a crunode $c$ w.r.t. a point $z$ and two normals (in black) to the two local branches of the resolution.}\label{pic:brusotti}
\end{figure}

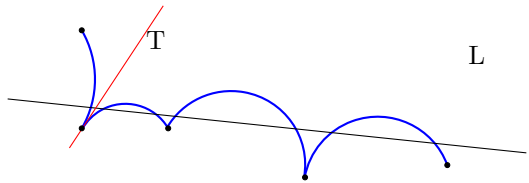
\begin{figure}

\begin{tikzpicture}[scale=0.65]

\draw [thick, blue]  (0,0) arc (-30:30:2) ;

\draw [thick, blue]  (1.75,0) arc (30:150:1);

\draw [thick, blue]  (4.5,-1.) arc (-10:150:1.5) ;

\draw [thick, blue]  (7.4,-0.75) arc (20:170:1.5) ;

\draw[thin, black] (-1.5,0.6) -- (9,-0.5);  

\draw[thin, red] (-.25,-0.4) -- (1.65,2.5);

\node at (8,1.5) {L};
\node at (1.5,1.8) {T};

\draw[fill] (0.0,0) circle [radius=0.05];

\draw[fill] (1.75,0) circle [radius=0.05];

\draw[fill] (0.0,2) circle [radius=0.05];

\draw[fill] (4.52,-1) circle [radius=0.05];

\draw[fill] (7.4,-0.75) circle [radius=0.05];

\end{tikzpicture}

\caption{Illustration of one possible case in the proof of Lemma~\ref{lm:astroid}.}\label{pic:illustr}
\end{figure}

\begin{lemma}\label{lm:astroid} The $\bR$-degree of the evolute of a non-empty generic real conic equals $4$. 
\end{lemma} 

\begin{proof}  Any ellipse in $\bR^2$ can be reduced by a translation and rotation to the standard ellipse
 $\frac{x^2}{a^2} + \frac{y^2}{b^2} = 1$ whose evolute is parameterized as $\{x = \frac{a^2-b^2}{a} \cos^3 t,\; y =\frac{b^2-a^2}{b} \sin^3 t\}$ and is given 
by the equation $(ax)^{2/3} + (by)^{2/3} = (a^2-b^2)^{2/3}$. Analogously, any hyperbola can be reduced by the same operations to the standard hyperbola $\frac{x^2}{a^2}-\frac{y^2}{b^2} = 1$ whose evolute is parameterized as $\{x = \frac{a^2+b^2}{a} \cosh^3 t, y =\frac{a^2+b^2}{b} \sinh^3 t\}$ and is given by the equation $(ax)^{2/3}- (by)^{2/3} = (a^2 + b^2)^{2/3}$. 

In both cases one can
find a line in $\bR P^2$ such that in the affine chart obtained as the complement to this line, the evolute $E$ of the conic becomes a plane closed $4$-gon $\mathcal F$ with smooth and concave sides, see Fig.~\ref{fig:one} and Fig.~\ref{fig:hypRot}. An additional property which is essential for our argument, is that no line can intersect each of the $4$ sides of $\mathcal F$  more than twice. To prove this in the case of a standard ellipse, observe that each side of the evolute can be considered as a convex/concave graph of a function $y(x)$ and therefore cannot intersect a real line more than two times, since otherwise there will necessarily exist an inflection point. As a consequence, we get that such a graph lies on one side with respect to the tangent line at any of its points. The case of a hyperbola is similar.

Let us show that no real line can intersect a concave closed $4$-gon $\mathcal F$ with smooth sides more than four times if any of its sides can not intersect a line more than twice. (Recall that we count intersection points without multiplicities.) Indeed, assume that there is a real line $L$ intersecting $\mathcal F$ at least six times transversally, see Fig. 5. (Notice that a line intersecting a closed simple curve transversally must intersect it an even number of times.) Then $L$ intersects $\mathcal F$ either six or eight times since it can not intersect any side of $\mathcal F$ more than twice. 

Firstly, notice that by concavity of $\mathcal F$, the line $L$ can not intersect it  $8$ times since in this case the $4$-tuple of concave sides will not be able to close up forming  a $4$-gon. A similar situation occurs in all cases of the intersection multiplicities of the line $L$ with the $4$  consecutive cyclicly ordered sides of $\mathcal F$ are  $2-2-1-1$. The final possibility is the intersection multiplicities $2-2-2$ with the three consecutive sides of  $\mathcal F$ which is illustrated in Fig.~\ref{pic:illustr}. Let us provide more details in this special case. Consider the leftmost of these three sides and draw a tangent line $T$ at its leftmost vertex, see Fig.~\ref{pic:illustr}. Then the last remaining side must lie in one halfplane of the complement to $T$ while the rightmost vertex of the rightmost side among the three interesting $L$ twice must lie in the other halfplane in the complement to $T$. Thus again the $4$ sides can not close up to form a $4$-gon.
\end{proof}

\section{On the maximal number of real vertices of a plane real-algebraic curve}\label{sec:vert}
In this section we discuss Problem 2, providing a lower bound for the maximal number $\bR\Ve(d)$ of real vertices of a real-algebraic curve of degree $d$. Recall that by Proposition~\ref{prop:cusps+nodes}, the number $\bC\Ve(d)$ of complex vertices of a generic curve $\Ga$ of degree $d$ (i.e., the number of cusps of its evolute $\tilde E_\Ga^\bC$) equals  $3d(2d-3)$.

Below we obtain a lower bound of $\bR\Ve(d)$ via small deformations of real line arrangements.  

\begin{proposition}\label{prop:realcusps}
\rm{(i)} The number of real cusps for the evolute of an arbitrary small deformation $\RR$ of any generic line arrangement $\cA\subset \bR^2$ of degree $d$ is at least $d(d-1)$ plus the number of bounded edges  of $\cA$ respected by $\RR$. 

\noindent
\rm{(ii)}  If the line arrangement $\cA$ is given by the equation $\prod_{i=1}^d L_i=0$, where the $L_i$'s  are linear equations describing the lines of $\cA$, then for any sufficiently small real number $\epsilon\neq 0$,  the deformation   $\RR$ given by $\prod_{i=1}^d L_i=\epsilon$ respects all the bounded edges of $\cA$. In this case the total number of real cusps on its evolute is greater than or equal to $d(2d-3)$.    
 \end{proposition}

\begin{remark} {\rm Apparently
 the number $d(2d-3)$ is the maximal possible number of cusps among the evolutes of small deformations of generic line arrangements of degree $d$. It is exactly equal to  one third of the number $\bC\Ve(d)$ of complex cusps. }
\end{remark}

\begin{proof}[Proof of Proposition~\ref{prop:realcusps}]
 Given a generic line arrangement $\cA\subset \bR^2$, consider its complement $\bR^2\setminus \cA$. It consists of $\binom {d-1}{2}$ bounded and  $2d$ unbounded convex polygons. Now take any small deformation $\RR$ of $\cA$ among real curves of degree $d$. By Proposition~\ref{lm:smallcomp}, in a sufficiently small neighborhood of  each vertex $v$ of $\cA$ the smooth curve $\RR$ consists of  two convex branches at each of which the (absolute value of the) curvature attains a local maximum  near $v$. These two local maxima  correspond to two cusps 
on the evolute $E_{\RR}$ of   $\RR$ which gives totally $2\binom {d}{2}=d(d-1)$ cusps corresponding to the local maxima of the absolute value of the curvature. 

Additionally, by Corollary~\ref{lm:edges} (iii), every bounded edge of $\cA$ respected by $\RR$  supplies at least  one additional vertex on $\RR$  which settles item (i). 
To settle (ii), observe that for a sufficiently small deformation $\RR$ of $\cA$ which twists a bounded edge  $e$ of $\cA$,  the long component of $\RR_e$ must intersect  $e$, see  Fig.~\ref{pic:edgeright}. On the other hand, the line arrangement $\cA$ given by $\prod_{i=1}^d L_i=0$  and its deformation $\prod_{i=1}^d L_i=\epsilon$ have no common points. Therefore $\RR$ respects  every bounded edge of $\cA$ and, in particular,  has no inflection points.  

The total number of bounded edges of any generic arrangement $\cA\subset \bR^2$ with $d$ lines equals $d(d-2)$. Thus we get at least $d(d-1)+d(d-2)=d(2d-3)$ extrema of the curvature on the smooth curve $\RR$ given by $\prod_{i=1}^d L_i=\epsilon$, for small real $\epsilon$. 
\end{proof} 

\begin{remark} {\rm Conjecturally, for any  small deformation $\RR$ of a generic line arrangement $\cA$, the number of the minima of the curvature on $\RR$ plus the number of inflection points on $\RR$ equals the number of bounded edges of $\cA$  (which is given by $d(d-2)$).}
\end{remark}

\section{On the maximal number of real diameters of a plane real-algebraic curve}\label{sec:diam}

In this section we provide some information about Problem~\ref{nodesNGa}. Let us denote by $\R\Diam(d)$ the maximal number of real diameters for real-algebraic curves of degree $d$ having no circles as irreducible components  and   by $\bC\Diam(d)$ the number of complex diameters of a generic curve $\Ga$ of degree $d$. By definition, $\bC\Diam(d)$   equals the number $\delta_N$ of complex nodes of $N_\Ga$ not counting the special $d$-uple point at $\infty$, which by Proposition~\ref{prop:cusps+nodes} implies that 
\[\textstyle \bC\Diam(d)=\binom{d}{2}(d^2+d-4)=\frac{1}{2}d^4-\frac{5}{2}d^2+2d.\]   
In particular, $\bC\Diam(2)=2$, and the number of real diameters  of an  ellipse is also $2$.   Further,  $\bC\Diam(3)=24$.  Based on our experiments,  we conjecture that  not all $24$ complex diameters can be made real. In other words,  there is probably no (generic) real cubic $\Ga$ for which all $24$ complex nodes of its curve of normals $N_\Ga$ are crunodes.

     Below we provide a lower bound for $\bR\Diam(d)$ by using small deformations of real line arrangements.  
Assume that we have a generic arrangement $\cA$ of $d$ lines in $\bR^2$, meaning that no two lines are parallel and no three intersect at the same point. 
Again by  Brusotti's theorem, we can find a small real deformation of $\cA$ within real curves of degree $d$ which  resolves each of the $\binom {d}{2}$ crunodes of  $\cA$ in a prescribed way. (For a generic $\cA\subset \bR^2$ of degree $d$, there exist $2^{\binom{d}{2}}$ topological types of its small real resolutions.)  It turns out that under some additional generality assumptions, one can estimate the number of real diameters of any such small resolution $\RR$. To move further, we need to introduce more notions related to line arrangements and their small resolutions.  
\medskip

\noindent{\bf Notation.}  A line arrangement $\cA\subset \bR^2$ is called \emph{strongly generic} if in addition to the conditions that no two lines are parallel and no three lines intersect at the same point, we require that no two lines are \emph{perpendicular}. By an \emph{altitude} of a given line arrangement $\cA\subset \bR^2$, we mean a straight segment connecting a vertex  $v$ of  $\cA$ with a point on a line of $\cA$  not containing $v$ such that this segment is perpendicular to the chosen line, see Fig.~\ref{pic:altitude}. The line of $\cA$ to which an altitude $\al$ is orthogonal, is called a \emph{base line}.  (Notice that if $\cA\subset \bR^2$ is strongly generic, then no altitude of $\cA$  connects its vertices.) Finally, we call a segment of a line belonging to $\cA$ and connecting its two vertices, a \emph{side} of the arrangement $\cA$. (In particular, every bounded edge is a side, but a side can consist of several bounded edges). 

Given two intersecting lines in $\bR^2$, we say that a pair of opposite sectors of its complement form a \emph{cone}. Thus the complement to the union of two lines consists of two disjoint cones. (The closure of a cone will be called a \emph{closed cone}).   We will mainly be interested in cones in a small neighborhood of their respective vertices. 

Assume that we have chosen some  type  of  a small real resolution  $\RR$ of a given strongly generic arrangement $\cA$, which means that at each vertex $v$ of $\cA$ we have (independently of other vertices) chosen which of two  local cones bounded by the lines whose intersection gives $v$, will merge, i.e., whose sectors will glue together after the deformation. (Two sectors of the other cone will stay disjoint under a local deformation.) We will call the first cone  \emph{merging}  and the second one \emph{persisting}. 

For a given line $\ell\subset \bR^2$ and point $p\in \ell$, denote by $\ell^\perp(p)$ the line passing through $p$ and orthogonal to $\ell$.  For a cone bounded by two lines $\ell_1$ and $\ell_2$ intersecting at some vertex  $v$, define its \emph{dual cone} as the union of all lines passing through $v$ and such that every line is orthogonal to some line passing through $v$ and belonging to the initial cone. (The dual cone is bounded by $\ell^\perp_1(v)$ and   $\ell^\perp_2(v)$.)  

Given a generic line arrangement $\cA\subset \bR^2$ of degree $d$,  define its \emph{derived arrangement} $\cD\cA\subset \bR^2$ of degree $d(d-1)$ as follows. For any pair of lines $\ell_1$ and $\ell_2$ from $\cA$, let $v$ denote their intersection point. Then $\cD\cA$ consists of all  lines $\ell_1^\perp(v)$ and $\ell_2^\perp(v)$ where  $\ell_1$ and $\ell_2$ run over all pairs of distinct lines in $\cA$. 
 
If we choose some resolution $\RR$ of $\cA$ then at each vertex $v$ of $\cA$, we get the  persistent cone $\cC_v(\RR)$ and its dual persistent cone $\cC_v^\perp(\RR)$ bounded by two lines of $\cD\cA$ which are perpendicular to the lines of $\cA$ those intersection gives $v$. If $\alpha$ is an altitude starting at $v$ and $\RR$ is some resolution, we say that $\alpha$ is \emph{admissible} with respect to $\RR$ if it lies inside $\cC_v(\RR)$ and \emph{non-admissible}  otherwise.  Finally, for a given strongly generic $\cA$, its two vertices $v_1$ and $v_2$ and any resolution $\RR$, we say that $v_1$ and $v_2$ \emph{see each other with respect to $\RR$} if   $v_2 \in \cC_{v_1}^\perp(\RR)$ and  $v_1\in \cC_{v_2}^\perp(\RR)$.

\begin{figure}

\begin{tikzpicture}[scale=1.5]

\draw[thin, blue] (0,1) -- (1.5,3);  

\draw[thin, blue] (-0.2,1.3) -- (1.7,2.7);

\draw[fill] (0.75,2) circle [radius=0.05];

\node at (0.8,2.4) {v};

\draw[fill] (3.25,2) circle [radius=0.05];

\node at (3.2,2.4) {w};

\draw[thin, blue] (2.5,3) -- (4,1);  

\draw[thin, blue] (4,3) -- (2.5,1);  

\draw[thin, black,dashed] (1.5,1.0) -- (0,3);  

\draw[thin, black,dashed] (1.5,1.4) -- (0,2.6);

\draw[thick, red, dashed] (0.75,2) -- (3.25,2);  

\draw[thin, black,dashed] (4,2.5) -- (2.5,1.5);  

\draw[thin, black,dashed] (4,1.3) -- (2.4,2.7);


 \draw[blue] (1.25,2.35)  arc (40:44:3);
 
  \draw[blue] (0.1,1.56)  arc (220:225:3);
  
   \draw[blue] (3.5,2.35)  arc (85:95:3);
   
      \draw[blue] (3.55,1.6)  arc (-82:-105:1.4);

\end{tikzpicture}

\caption{Two vertices $v$ and $w$ of a line arrangement $\cA$ which see each other  (along the red dashed line) with respect to $\RR$. The lines of $\cA$ are shown in solid blue, while the dashed lines belong to the derived arrangement. The persistent cones are marked by the blue arcs.} \label{fig:Proper}
\end{figure}
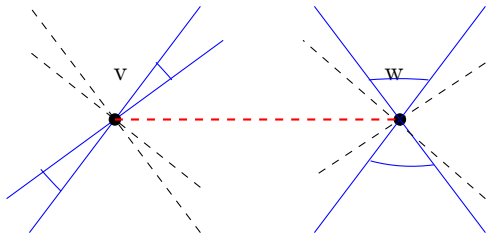

\begin{lemma}\label{lm:local} Given a strongly generic line arrangement $\cA$, the following holds:
\begin{itemize}
\item[(i)] any small resolution $\RR$ of a vertex $v\in \cA$ has (at least) one short diameter  near $v$; 
\item[(ii)] if an altitude $\alpha$ is admissible with respect to a small deformation $\RR$ then $\RR$ has (at least) two diameters close to $\alpha$; 
\item[(iii)] if $v_1$ and $v_2$ see each other with respect to a small deformation $\RR$ then $\RR$ has (at least) four diameters close to the straight segment $(v_1,v_2)$.
\end{itemize}
\end{lemma}

\begin{proof} 
To settle (i), observe that a small deformation $\RR$ of $\cA$ restricted to  a small neighborhood of a given vertex $v$ consists of two convex connected components which are convex  ``towards each other".  For each point on one of the local branches, the distance function to the other branch increases towards its endpoints.  Therefore there exists a global  minimum of the distance function between these branches, giving a diameter  attained on a pair of points lying inside both branches (and not on their boundary), see Fig.~\ref{pic:altitude}.   

To settle (ii), denote by $\ell_1$ and $\ell_2$ a pair of lines belonging to $\cA$ whose intersection gives the vertex $v$.  Recall that the Gauss map sends a point of a curve to the normal line to the curve passing through this point. Observe that a small resolution $\RR$ creates near the vertex $v$ two short connected components such that the Gauss map  on each of them moves from a line close to $\ell_1^\perp$ to  a line close to $\ell_2^\perp$. The Gauss map of a small deformation of the base line of $\al$ (to which the altitude $\al$ is orthogonal)  is close to the direction of the altitude. Thus by continuity and admissibility we will be able to find (at least) one normal on each of the short components which is also orthogonal to the small deformation of the base line, see Fig.~\ref{pic:altitude}. 

To settle (iii), observe  that short segments are very close to the corresponding hyperbolas (one for each crunode), see Fig.~\ref{pic:seeeach}. For the hyperbolas the statement is true and the four diameters are located close to the connecting segment $vw$.   To prove this, notice that we have four pairs of branches of hyperbolas to consider, lying in four pairs of respective sectors. These pairs of sectors can be of three possible types. We say that a pair of sectors ``look away from each other"  if each cone contains the apex of the other cone; ``look towards each other" if their intersection is empty; ``one looks after the other" if the apex of one cone belongs to the other, but not the other way around.  For a pair whose sectors ``look away from each other",  we get a maximum of the distance between the branches of hyperbolas near their apices, for the pair whose sectors ``look towards each other", we get  a minimum of the distance between the branches of hyperbolas near their apices, and for the pairs where ``one looks after the other", we get a saddle point. Finally,  the same phenomena will be present in a sufficiently small deformation of hyperbolas. To prove this, we consider the pairwise distance as the function on the small square, one side of which gives the position of the point on one considered branch and the other side gives the position of the point on the other side.  Then if the sectors ``look away from each other", the distance function is concave on each horizontal and vertical segment of the square; if the sectors ``look towards each other", the distance function is convex on each horizontal and vertical segment of the square;  if the sectors ``one looks after the other", the distance function is convex on each horizontal and concave on the each vertical segment of the square. The result follows. 
\end{proof} 

\begin{figure}\includegraphics[scale=0.65]{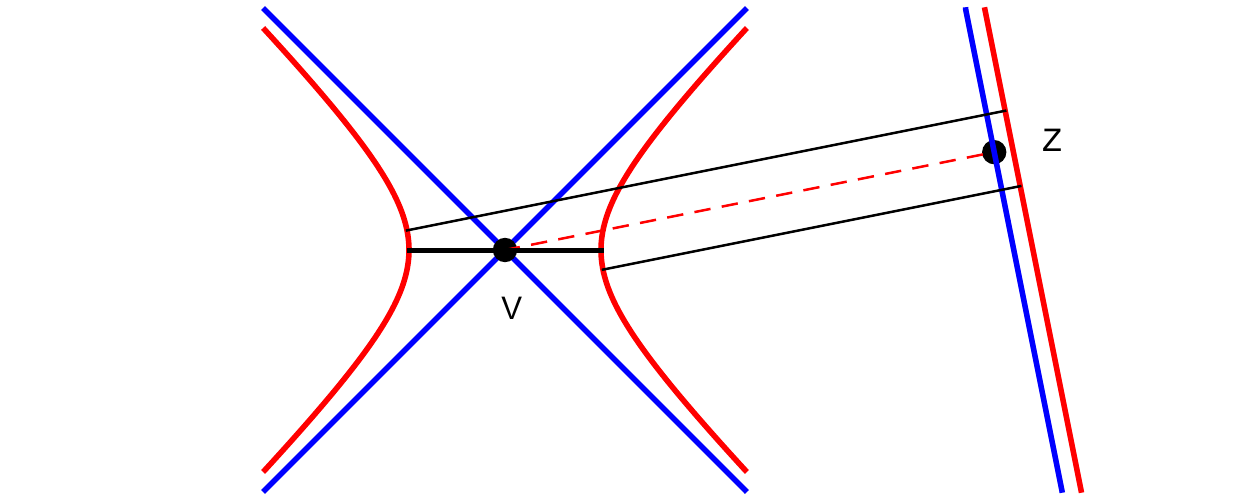}
\caption{A resolution of an admissible altitude (dashed) creating two diameters (black). We also include a short diameter created near the vertex $v$  (in black).}\label{pic:altitude}
\end{figure}

\begin{figure}

\includegraphics[scale=0.55]{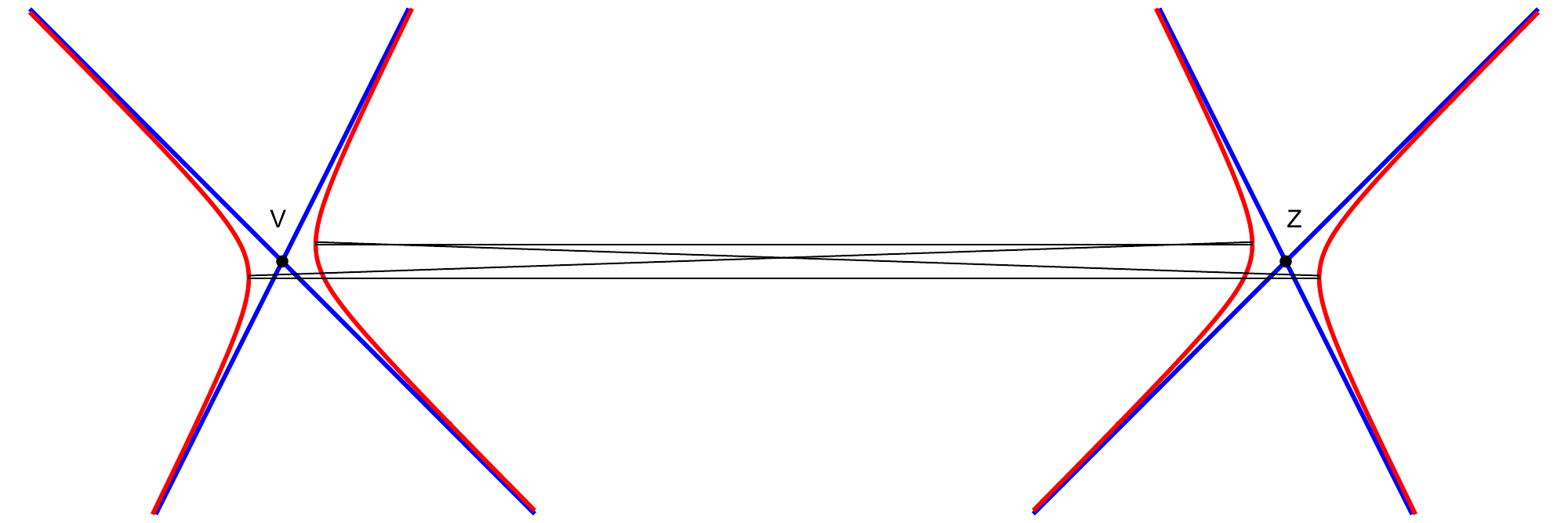}

\caption{A resolution of two vertices which see each other, resulting in four diameters (black).}\label{pic:seeeach}
\end{figure}

Lemma~\ref{lm:local} immediately implies the following claim.

\begin{proposition} \label{prop:Brus} Given any small resolution $\RR$ of a  strongly generic arrangement $\cA$ consisting of $d$ lines, the number of real diameters  of  $\RR$ is greater than or equal to  $\binom{d}{2}$ plus twice the number of admissible altitudes with respect to $\RR$, plus four times the number of pairs of vertices which see each other with respect to $\RR$. 
\end{proposition}
\begin{conjecture} The number of diameters of any real small resolution of any line arrangement of degree $d$ does not exceed the number
\begin{equation}\label{eq:kappa}
\binom{d}{2}+2\binom{d}{2}(d-2)+4\binom{\binom{d}{2}}{2}=\frac{1}{2}d^4-3d^2+\frac{5}{2}d.
\end{equation} 
\end{conjecture}
%
\begin{figure}[H]
     \centering
     \begin{subfigure}[t]{0.32\textwidth}
         \centering
      \includegraphics[width=0.95\columnwidth]{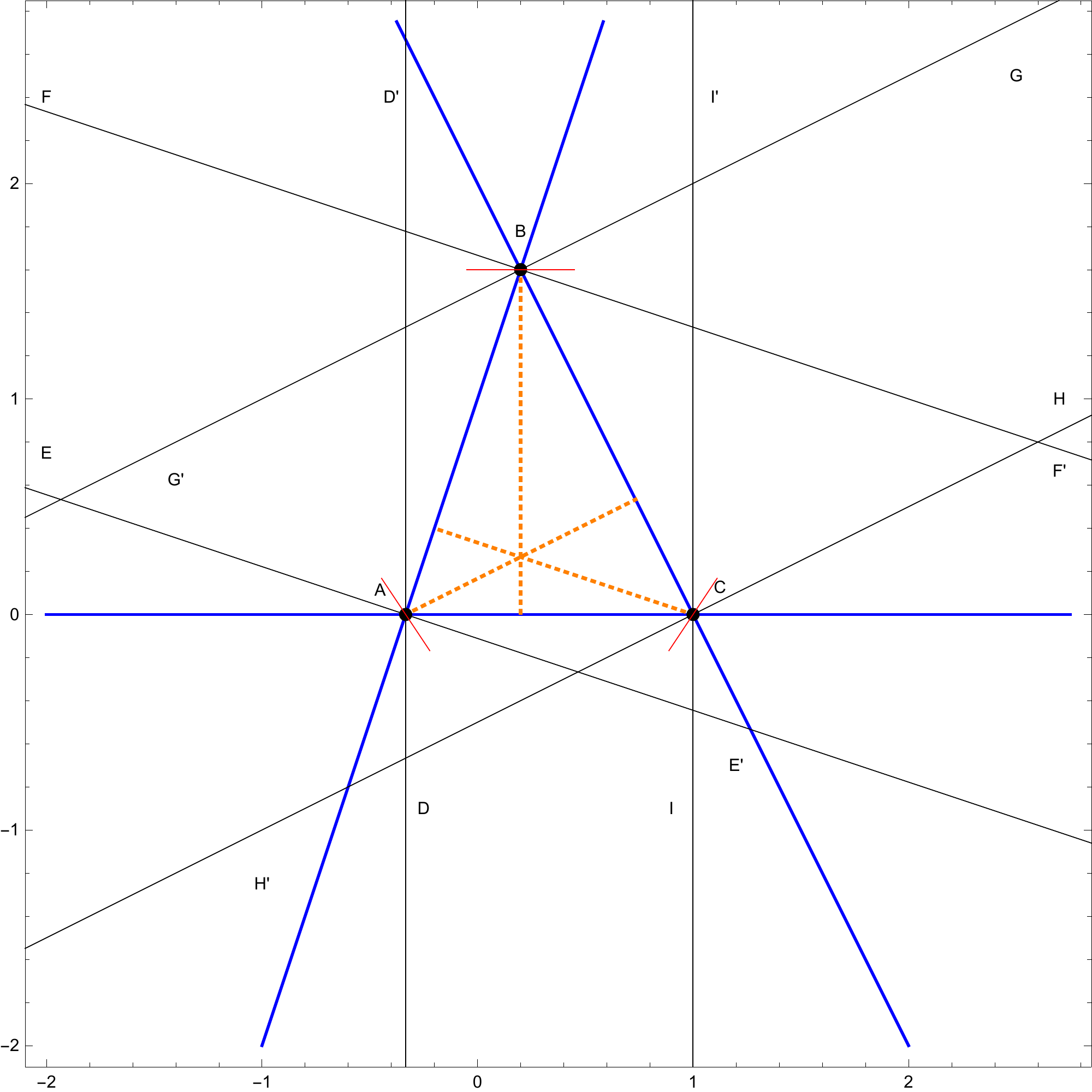} 
       \caption{Three lines (shown in blue), their altitudes (shown dotted in orange),  and their derived arrangement (shown in black). Red lines indicate the resolution.}\label{fig:fig9left}
     \end{subfigure}
     \hfill
     \begin{subfigure}[t]{0.32\textwidth}
         \centering
        \includegraphics[width=0.95\columnwidth]{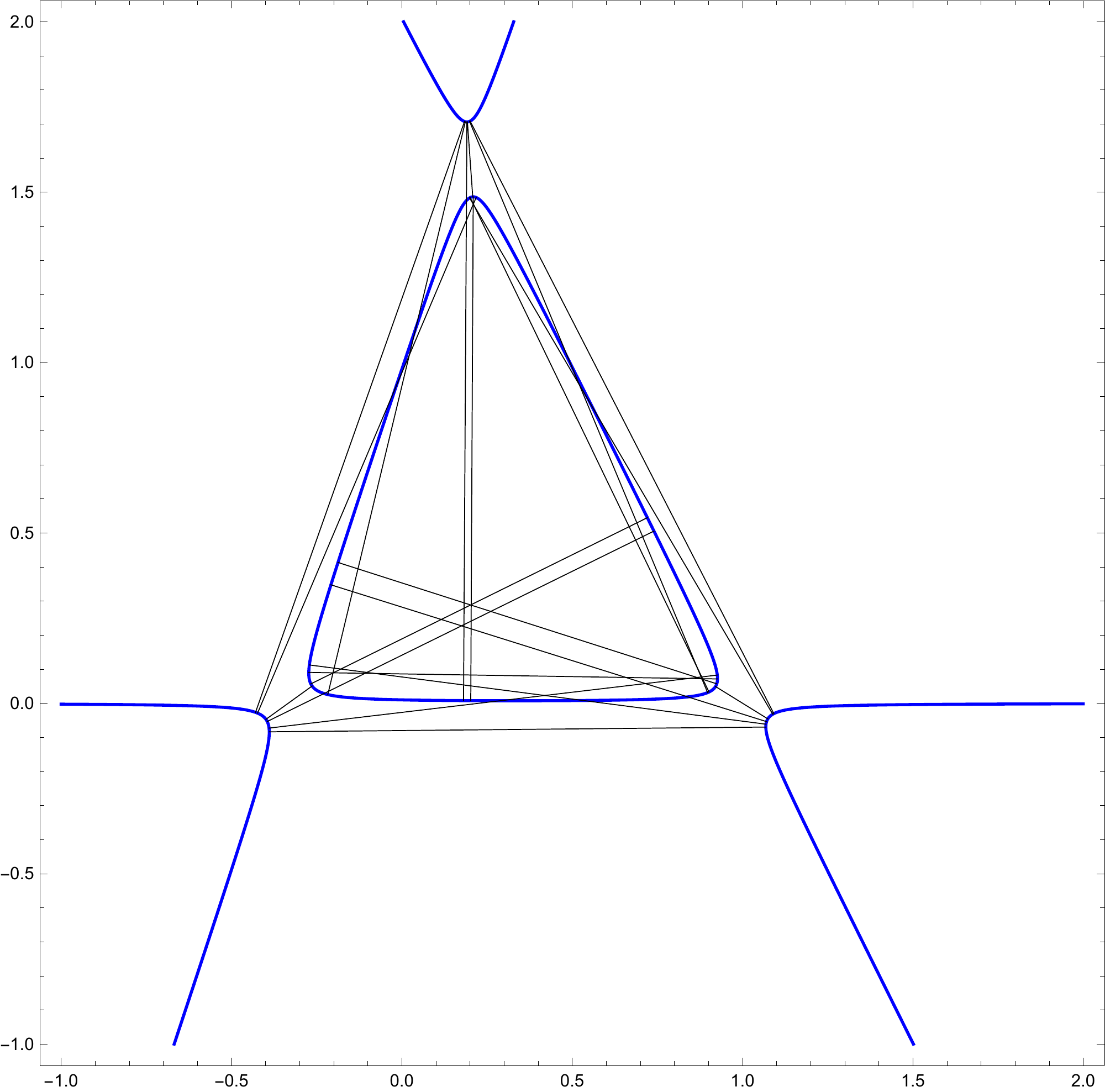}
        \caption{A small resolution  creating a smooth triangle-shaped oval with 3 infinite arcs having 21 diameters, which are drawn in black.}\label{fig:fig9middle}
     \end{subfigure} \hfill
      \begin{subfigure}[t]{0.32\textwidth}
         \centering
            \includegraphics[width=0.95\columnwidth]{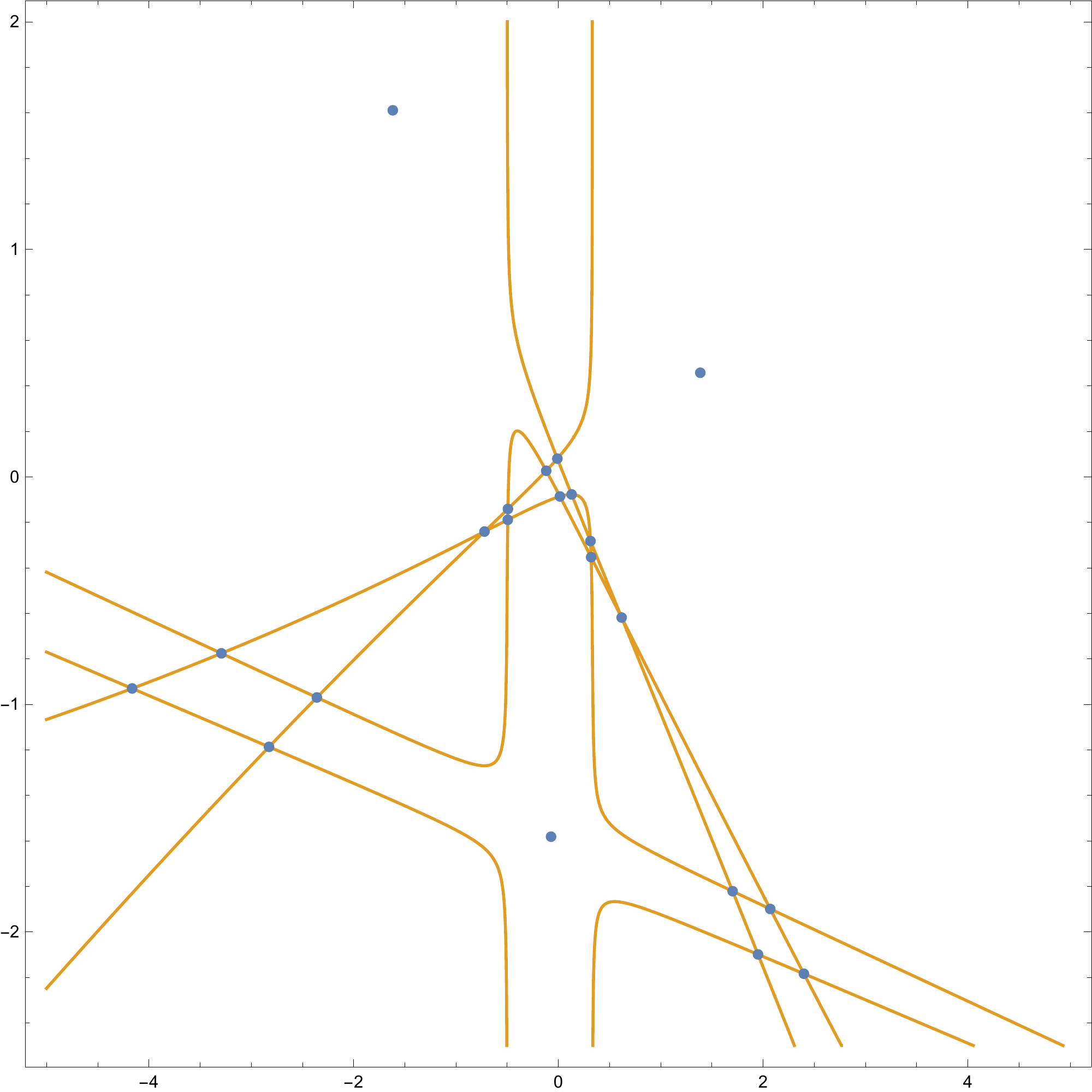}
        \caption{The curve of normals corresponding to the resolution has 24 real singularities in the affine plane. 21 of these are crunodes, of which 18 are visible in this picture -- there are two more  out left and one more to the right.}\label{fig:fig9right}
     \end{subfigure}\hfill
     \caption{An example of a resolution of three lines with 21 diameters.} \label{fig:Triangle}
     \end{figure}

\begin{example}{\rm Fig.~\ref{fig:Triangle} shows a special case of a line arrangement for $d=3$ and its small resolution which results in $21$ diameters. The arrangement itself is shown by $3$ blue lines and its derived arrangement is shown by $6$ black lines, the three altitudes are shown by $3$ red lines. The chosen resolution is shown by $3$ short red segments at the vertices $A, B,C$ which indicate how we cut out arrangement at its vertices. As the result of such a resolution we will obtain a compact triangle-formed oval inside together with $3$ infinite arcs. Three pairs of persistent sectors will include $EAD - D'AE'$, $FBG - F'BG'$, and $HCI - H'CI'$. Every two vertices see each other with respect to $\RR$ for this small resolution, and each altitude is admissible. Therefore we get $3+2\times 3 + 3\times 4=21$ real diameters on this small resolution $\RR$. We conjecture that 21 is the maximal number of real diameters which can be obtained by a small deformation of three real lines. 
}\end{example}

Proposition~\ref{prop:Brus}  implies the following lower bound for $\bR\Diam(d)$. 

\begin{proposition}\label{prop:thin} In the above notation,
\begin{equation}\label{eq:newbound}
\textstyle \bR\Diam(d)\ge \frac{1}{2}d^4-d^3+\frac{1}{2}d.
\end{equation}
\end{proposition}

To settle Proposition~\ref{prop:thin}, we need to introduce a certain class of real line arrangements. We say that an arrangement $\cA$ is \emph{oblate} if the slopes of all lines in $\cA$ are close to each other. As a particular example, one can take $d$ lines tangent to the graph of $\arctan x$ for $d$ values of the variable $x$ of the form $101, 102, 103, \dots , 100+d$. 

\begin{proof} For the special small resolution  of an oblate arrangement for which  we choose the narrow cone at each vertex to be the persistent,  the following diameters will be present. Every pair of persistent cones  will be in proper position and  contribute $4$ diameters and each vertex will contribute $1$ diameter. On the other hand, all altitudes will be non-admissible. Thus we get at least 
$4\binom{\binom{d}{2}}{2}+\binom{d}{2}=\frac{1}{2}d^4-d^3+\frac{1}{2}d$
diameters for this resolution. 
\end{proof} 

\begin{figure}[H]

\begin{tikzpicture}


\draw[thick, blue] (-0.7,2) -- (5.5,3);  

\draw[thick, blue] (-0.5,1.7) -- (5.5,3.5);  

\draw[thick, blue] (-0.3,2.4) -- (5.9,3);  

\draw[thick, blue] (-0.5,2.6) -- (4.7,2.6);  




\end{tikzpicture}

\caption{Example of an oblate line arrangement.} \label{fig:oblate}
\end{figure}
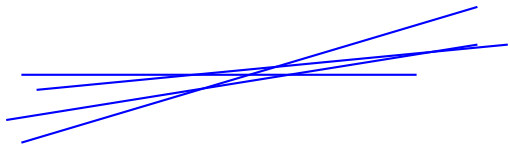

\begin{remark} {\rm It might happen that using some other types of line arrangements and their special resolutions, one can improve the lower bound \eqref{eq:newbound} and get closer to $\frac{1}{2}d^4-3d^2+\frac{5}{2}d$. For $d=3$, Fig. ~\ref{fig:Triangle} contains such a construction. However, for $d\ge 4$ it is not clear whether $\frac{1}{2}d^4-3d^2+\frac{5}{2}d$ is achievable by  small deformations of line arrangements. It seems difficult  to make all pairs of vertices to see each others  while simultaneously keeping all the altitudes admissible. }
\end{remark}

\section{On the maximal number of crunodes of the evolute}\label{sec:Ecru}
In this section we discuss Problem \ref{nodesEGa}.   
Recall that the number of complex nodes of the evolute of a generic curve of degree $d$ is given by 
\[ \textstyle \delta_E=\frac{1}{2}d(3d-5)(3d^2-d-6).\]

Denote by $\delta^{\rm cru}_E(d)$ the maximal number of crunodes for the evolutes of real-algebraic curves of degree $d$.
At the moment we have only a rather weak lower bound for this number of crunodes, which we again obtain by resolving a line arrangement.

\begin{proposition}\label{prop:Ecru}
We have $\delta^{\rm cru}_E(d)\ge(\lfloor \frac{d-2}{2}\rfloor +d-2)^4-\frac{1}{2}$.
\end{proposition}

\begin{proof} 
Take a line arrangement $\mathcal{A}$ and assume that in this arrangement we have a line which intersects two other lines in acute angles such that there is a bounded edge $e$ between the two resulting nodes $A$ and $B$. By Brusotti's theorem there exists a  resolution $\mathcal{R}$ in which $A$ and $B$ are resolved  in such a way that the long component of the resulting local branch $\mathcal{R}_e$ has curvatures of different signs near the two nodes, the tangent direction changes by more than 90 degrees near each of the nodes,  and such that it twists $e$. We call such a resolution a ``zig-zag''.  By  Corollary \ref{lm:edges} we know that the long component  of $\mathcal{R}_e$  will have at least one inflection point as well as two maxima of the absolute value of curvature close to the vertices. Considering the part of the evolute corresponding to $\mathcal{R}_e$,  we have  that these two maxima correspond to two cusps of the evolute, which, due to the different signs of curvature, are oriented towards each other. Furthermore, the existence of the inflection point results in a line, perpendicular to the compact line segment, which is an asymptote of the evolute, in such a way that from each of the two cusps one branch of the  evolute approaches this asymptote. Furthermore,  since the resolution can be chosen in such a way to ensure that the curvature radius becomes as large as necessary, the remaining  two branches are guaranteed  to intersect with the other branches (the described situation can also be seen in Fig.~\ref{fig:zigzac}). Now it follows that the two branches which tend to the asymptotical line are connected via the other two branches, and that the part of the evolute corresponding to $\mathcal{R}_e$  therefore contains a pseudo-line, Moreover,  $\mathcal{R}$ can be chosen in such a way that every bounded edge  that can be resolved in a zig-zag, contributes a pseudo-line in the evolute.  For $d\geq 3$, consider  a set of $\lfloor \frac{d}{2}\rfloor$ parallel lines as well as another set of $\lfloor \frac{d+1}{2}\rfloor$ parallel lines which are almost parallel to the first set of lines. The union of these two sets will yield a line arrangement with $(d-2)+\lfloor \frac{(d-2)^2}{2}\rfloor$ many bounded edges, and there exists a resolution of the line arrangement which resolves each of the bounded  edges in a zig-zag. The resulting evolute thus has  $(d-2)+\lfloor \frac{(d-2)^2}{2}\rfloor$ pseudo-lines, each of which will intersect pairwisely, yielding the announced lower bound of crunodes.    
\end{proof}

\begin{figure}
     \centering
     \begin{subfigure}[t]{0.45\textwidth}
         \centering
   \includegraphics[width=0.62\columnwidth]{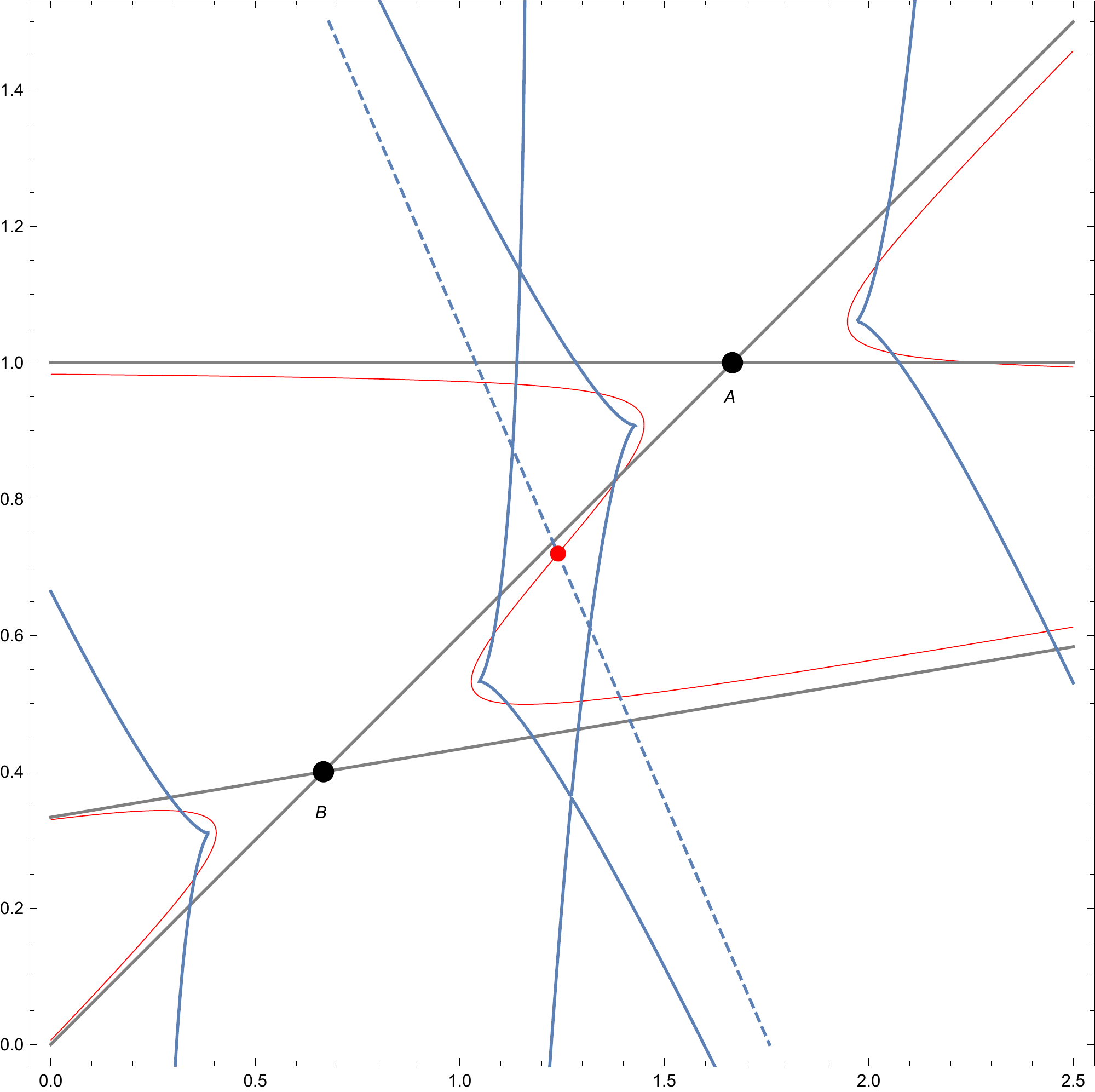}
    \caption{An arrangement of three lines (in gray) together with a zigzag deformation in red, with the inflection point  on the compact segment between {\bf A} and {\bf B}. The resulting evolute (in blue) has an assymptote (dashed in blue).} 
        \label{fig:zigzac}
     \end{subfigure}
     \hfill
     \begin{subfigure}[t]{0.45\textwidth}
         \centering
        \includegraphics[width=0.62\columnwidth]{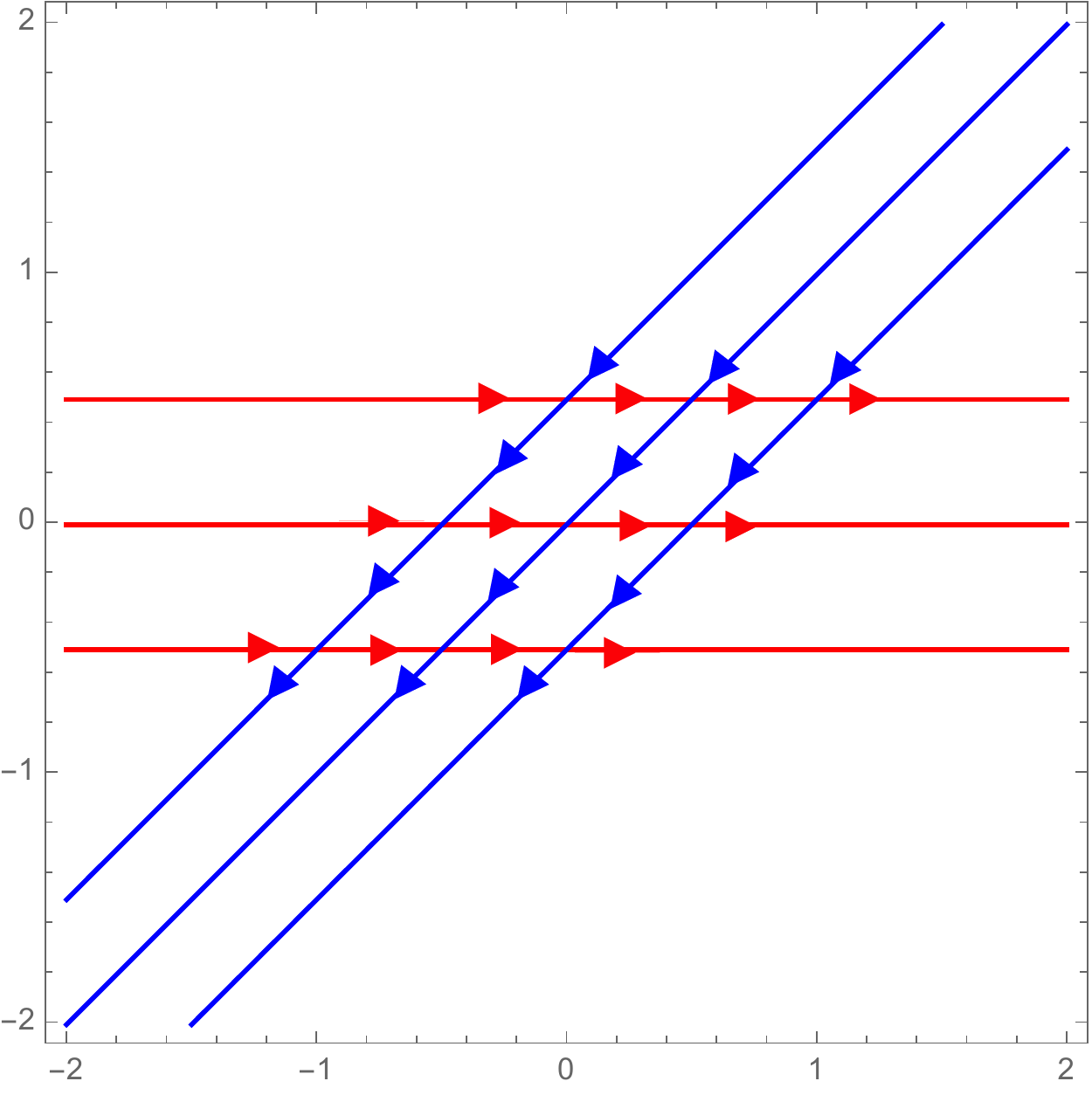}
        \caption{Two set of 3 lines. By resolving the 6 lines with changing the colours at every vertex and following the orientation  indicated by the arrows  one obtains the desired resolution.}
        \label{fig:lines}
     \end{subfigure}
       \caption{Visualization of zig-zag-resolutions.}
\label{fig:zigtac1}
\end{figure}

\begin{remark}{\rm The above lower bound is  obtained in a crude way, and it is surely possible to improve this construction. However, the bound is of the same degree as the complex bound.}
\end{remark}

\section{A florilegium of some real curves, their evolutes, and curves of normals}\label{sec:zoo}
 Here we present some classical examples to showcase the interplay between a real algebraic curve,   its evolute, and its curve of normals. Since many of the equations defining these curves are quite lengthy, we separately collect most of them in Appendix \ref{subsec:formulae}.
Many of our examples are known in the literature, see  e.g. \cite[\S~9]{Sa}, but, to the best of our knowledge, curves of normals have not been explicated earlier. For convenience of the readers, we collected the information about the standard invariants of the complexifications of the curves under consideration in Subsection~\ref{subs:tables}. Recall from Section \ref{sec:diam} that the real singular point (an ordinary $d$-uple point if the curve $\Ga$ is in general position) of the curve of normals corresponding to the line at infinity does not contribute to the count of the diameters of $\Ga$.
\bigskip

\subsection{Conics} \hfill\\

\noindent 
{\bf I.}  As we saw in Example \ref{ex:basic} (2), for the standard ellipse $\Ga$ given by the equation
\[\frac{x^2}{a^2}+\frac{y^2}{b^2}=1,\] 
its evolute $E_\Ga$ is a stretched astroid given
by the equation
\[(ax)^{2/3}+(by)^{2/3}=(a^2-b^2)^{2/3},\]
see Fig.~\ref{fig:conic1b}
\begin{figure}[H]\label{fig:ellRot}
     \centering
     \begin{subfigure}[t]{0.49\textwidth}
         \centering
      \includegraphics[width=0.62\columnwidth]{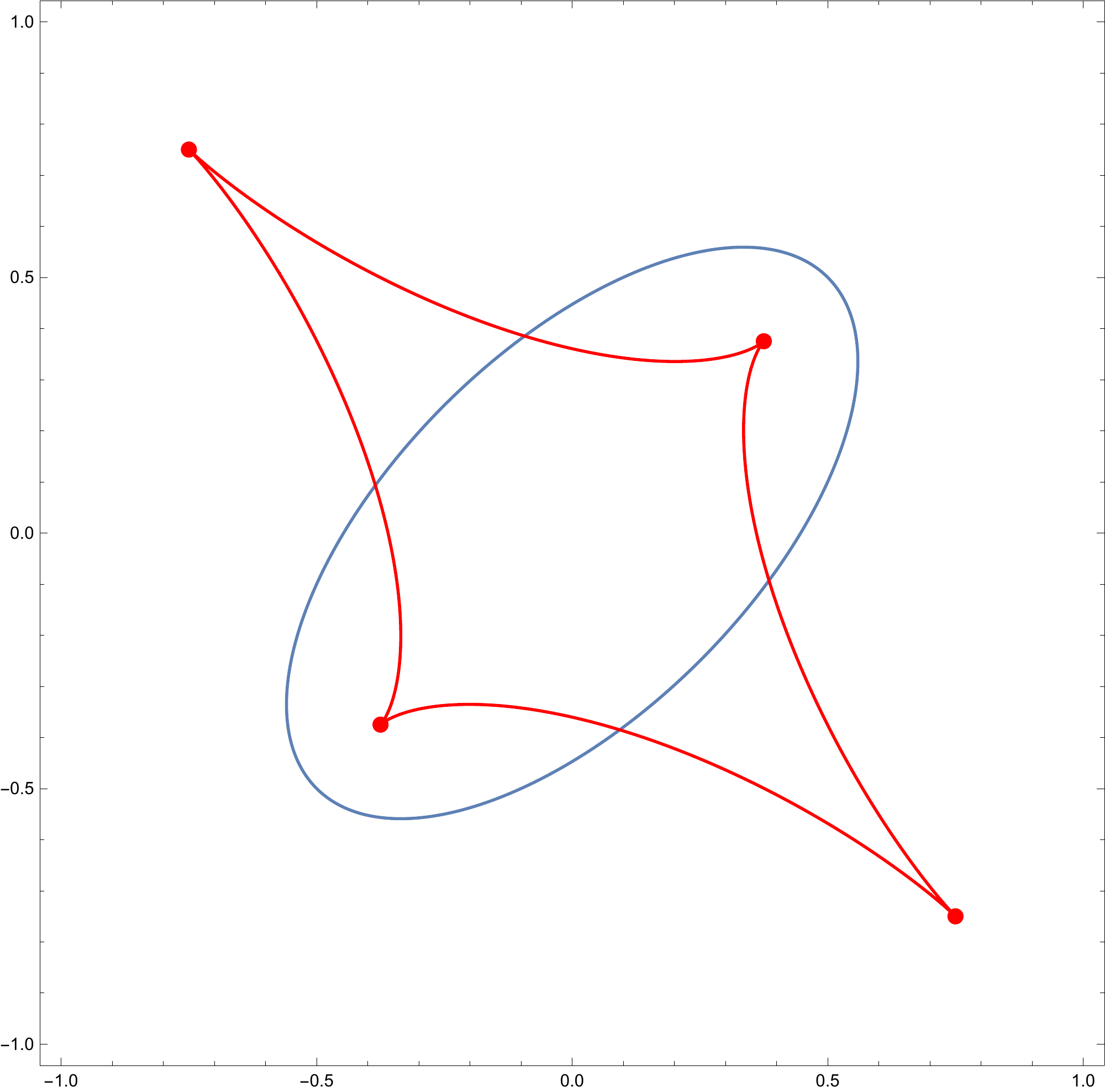} 
       \caption{A rotated ellipse in blue and its evolute in red.}\label{fig:ellRotleft}
        \label{fig:conic1a}
     \end{subfigure}
     \hfill
     \begin{subfigure}[t]{0.49\textwidth}
         \centering
        \includegraphics[width=0.62\columnwidth]{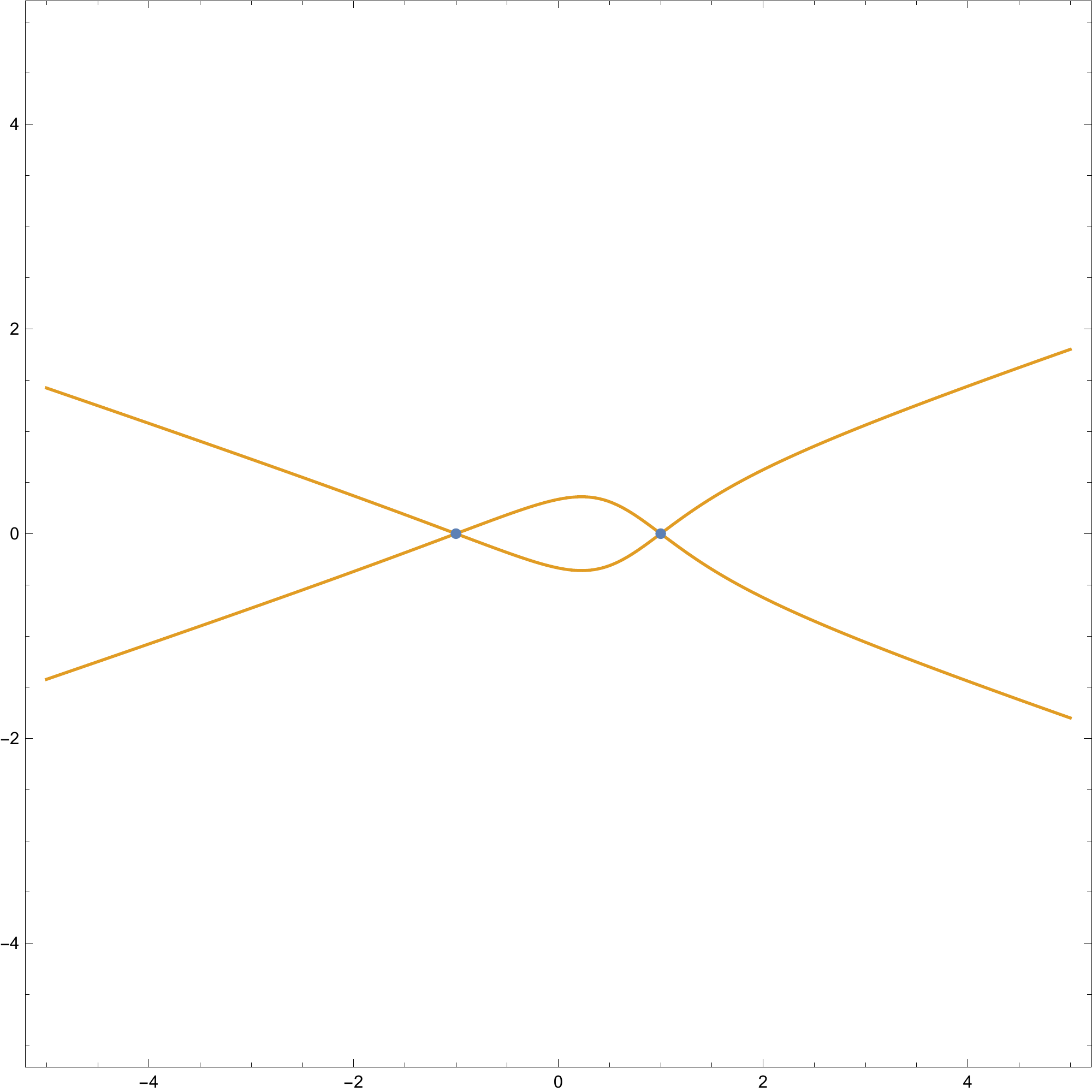}
        \caption{Corresponding curve of normals shown in gold.}\label{fig:ellRotright}
    
     \end{subfigure}
     \caption{A rotated ellipse, its  evolute, and its curve of normals.} \label{fig:ellRot}
        \end{figure}

The  evolute and the curve of normals are better seen in case of the rotated ellipse $(x+y)^2+4(y-x)^2-1=0$, shown in  Fig.~\ref{fig:ellRot}. 

 The evolute  is a sextic with 4 real cusps, 2 complex cusps, and 4 complex nodes. The curve of normals is a quartic with 2 real crunodes (corresponding to the diameters of the ellipse), and its double point corresponding to the line at infinity $z=0$ has no real branches, i.e.,  is an acnode, see  Fig.~\ref{fig:ellRotright}. Klein's formula for the evolute in this case reads as 
$6+2=4+4$, where the second $4$ in the right hand side comes from real cusps of the evolute and 2 in the left hand side comes from the acnode of the curve of normals. 

The ellipse has 4 vertices and 2 diameters.
\medskip

\noindent
{\bf II.}  For the hyperbola $\Ga$ given by the equation
\[\frac{x^2}{a^2}-\frac{y^2}{b^2}=1,\]
its evolute $E_\Ga$ is a version of a Lam\'e curve given 
by the equation
\[(ax)^{2/3}-(by)^{2/3}=(a^2+b^2)^{2/3}.\]

\begin{figure}[H]
     \centering
     \begin{subfigure}[t]{0.32\textwidth}
         \centering
      \includegraphics[width=0.95\columnwidth]{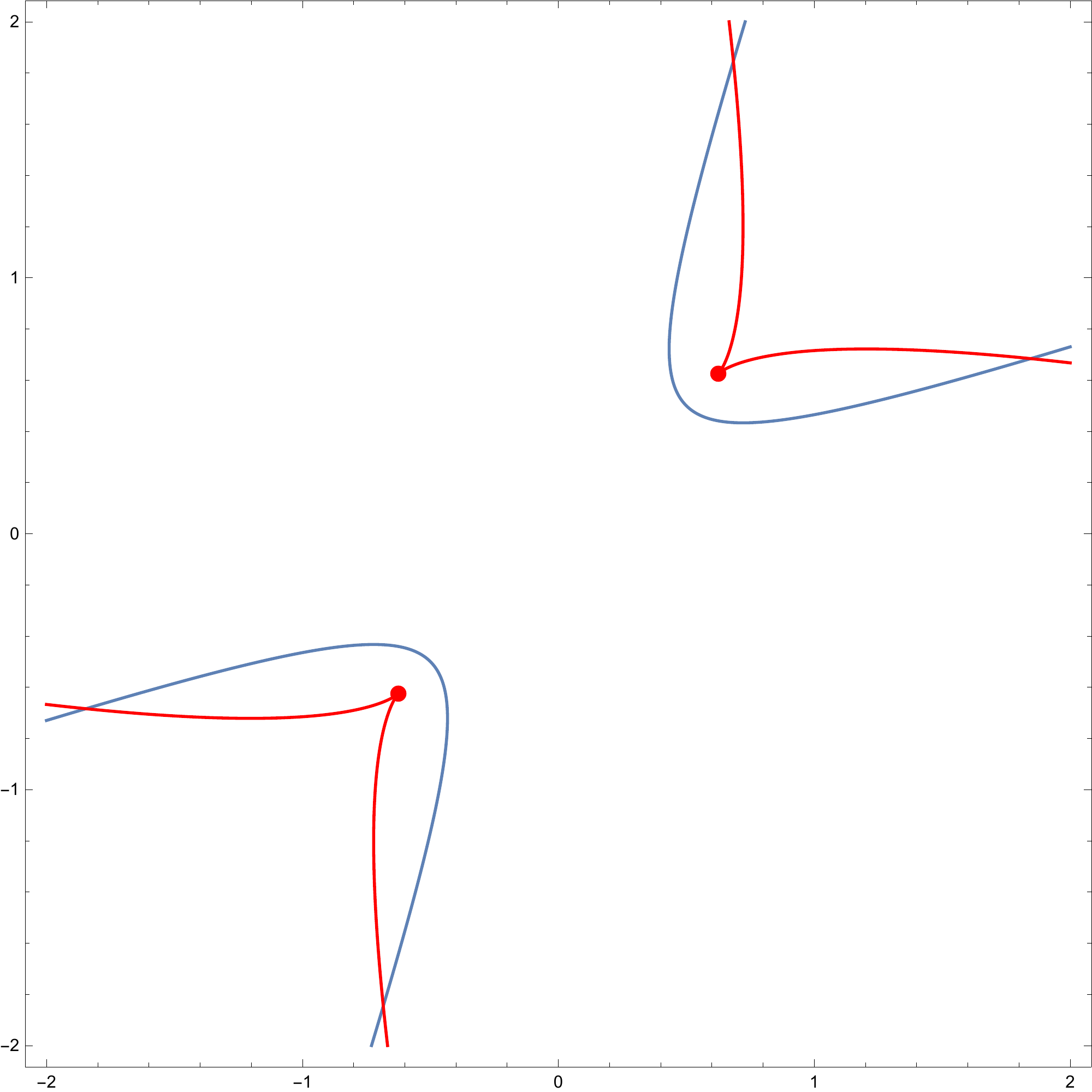} 
       \caption{A rotated hyperbola in blue and its evolute  in red.}\label{fig:hypRotleft}
        \label{fig:conic1a}
     \end{subfigure}
     \hfill
     \begin{subfigure}[t]{0.32\textwidth}
         \centering
        \includegraphics[width=0.95\columnwidth]{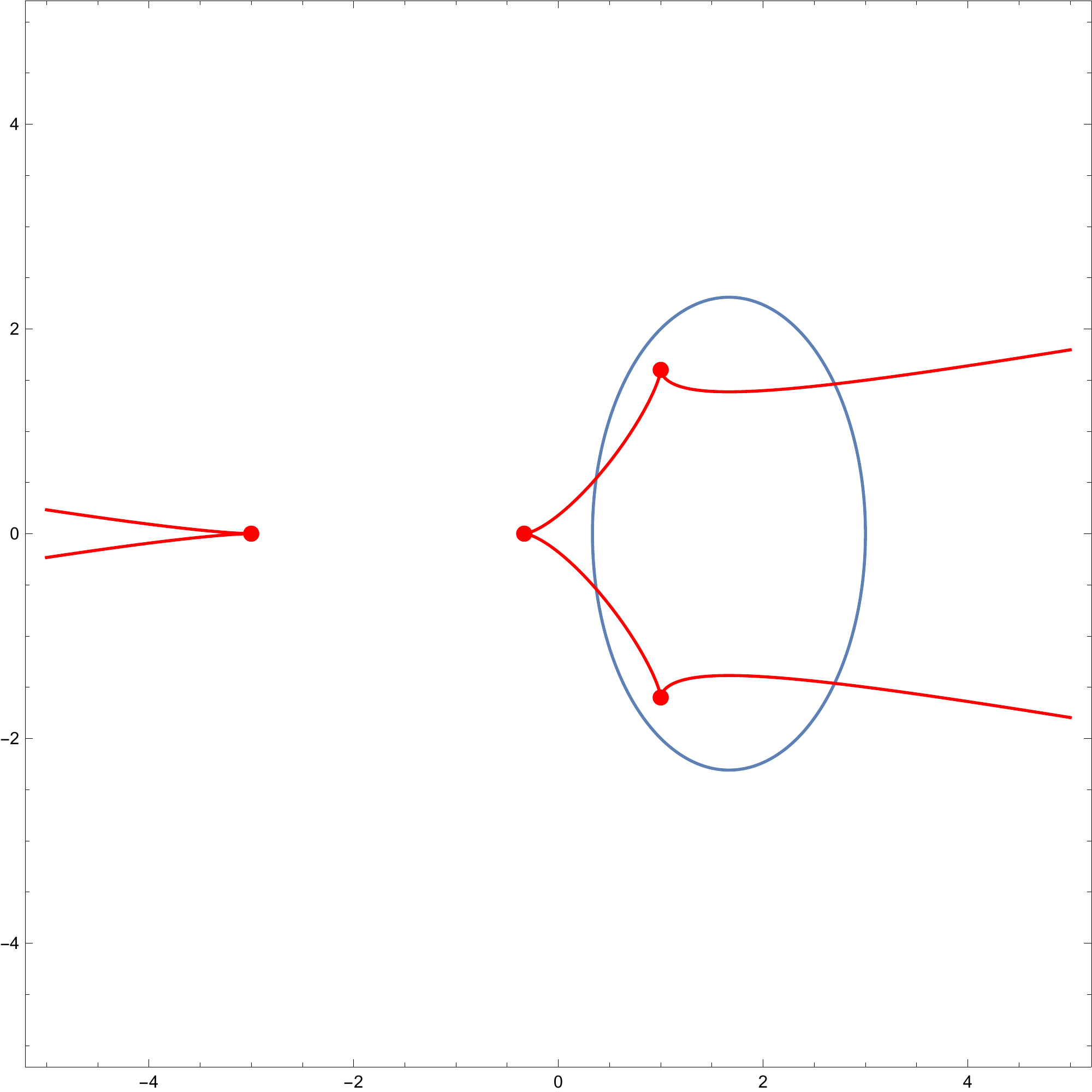}
        \caption{Same situation in a different affine chart.}\label{fig:hypRotmiddle}   
     \end{subfigure} \hfill
     \begin{subfigure}[t]{0.32\textwidth}
         \centering
        \includegraphics[width=0.95\columnwidth]{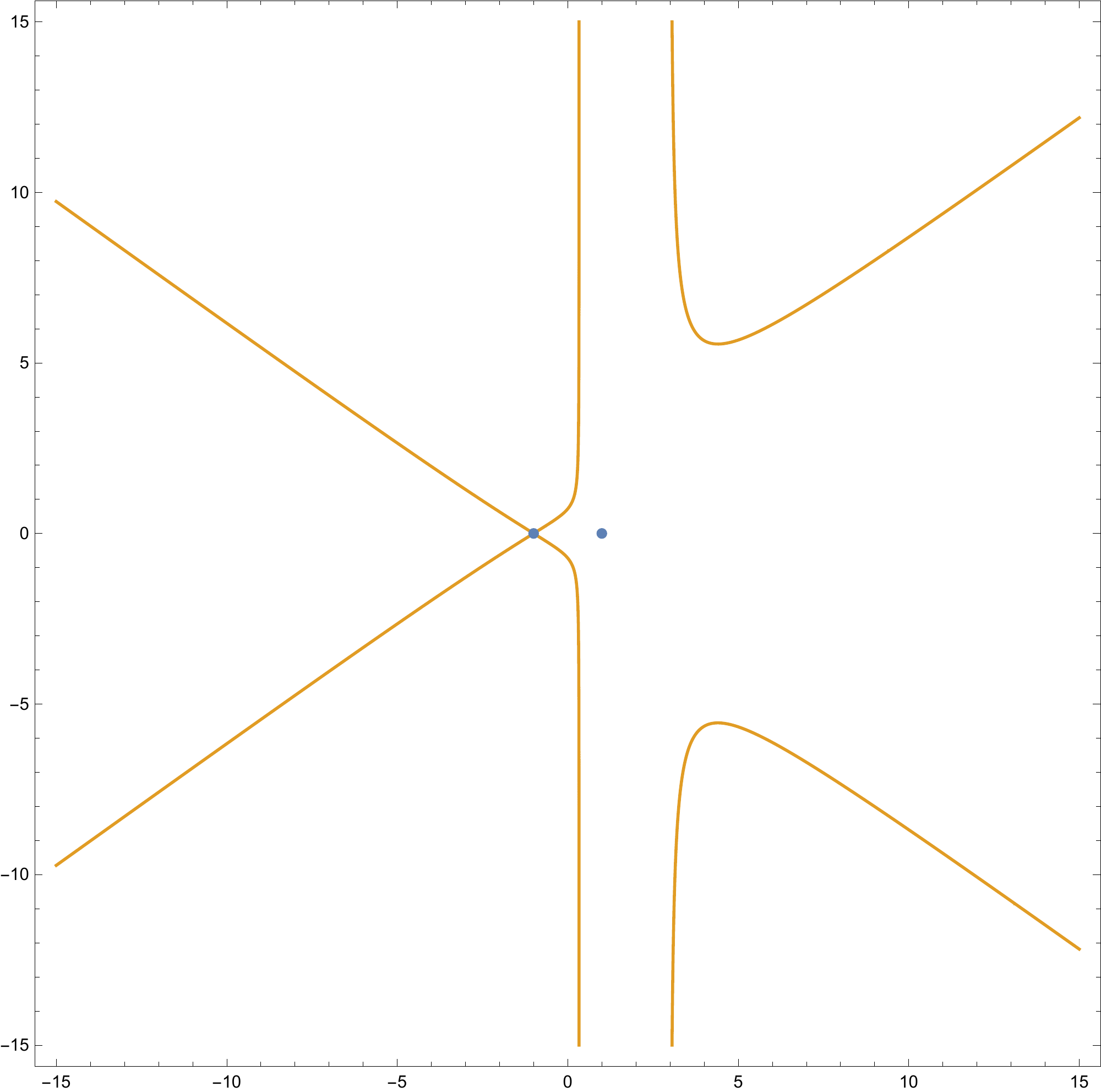}
        \caption{Corresponding curve of normals shown in gold.}\label{fig:hypRotright}
     \end{subfigure}
\caption{Rotated hyperbola, its evolute  and its curve of normals.}   \label{fig:hypRot}
\end{figure}

The evolute and the curve of normals  are better seen in case of the rotated hyperbola $(x+y)^2-4(y-x)^2-1=0$, see Fig.~\ref{fig:hypRot}. 
  As in the case of an ellipse, the evolute of a hyperbola is a sextic with 4 real cusps, 2 complex cusps, and 4 complex nodes. This  is best visible in a different affine chart (see Fig.~\ref{fig:hypRotmiddle}). The curve of normals again is a quartic with 1 crunode and 1  acnode in addition to its double point corresponding to the line at infinity $z=0$, which is a crunode. In Fig.~\ref{fig:hypRotright}  one can see the  (affine) crunode  corresponding to the unique real diameter of the hyperbola; this diameter connects the pair of  points on its two branches minimizing the distance function.  Klein's formula for the hyperbola is exactly the same as for the ellipse.
 
 The hyperbola has 2 vertices and 1 diameter.

\subsection{Cubics} \hfill\\

\noindent
{\bf III.} The cuspidal cubic classically referred to as \emph{cissoid} is given by the equation 
\[ (x^2+y^2)x=ay^2.\]
This rational curve is  \emph{circular}, meaning that it passes through the two circular points $(i:1:0)$ and $(-i:1:0)$ on the line at infinity. Besides the circular points, it intersects the line at infinity in the point $(0:1:0)$, which is an inflection point. The numerical characters of the curve are $d=3$, $\kappa =1$, $d^\vee= 2d-2-\kappa =3$, $\iota=1$.

The evolute has two inflectional tangents passing through the circular points, so $\iota_E=2$. Hence the degree of the curve of normals is $\deg N_\Gamma=d+d^\vee-\iota_E=4$, and the degree of its evolute is $\deg E_\Gamma=3d+\iota-3\iota_E=4$, see column III in Table~\ref{tb:inv}. The evolute is given by the equation $\textstyle 27y^4+288a^2y^2+512 a^3x=0.$ It is apparent from the equation  that the evolute has only one singular point, namely $(1:0:0)$, which has type $E_6$ (with tangent being the line at $\infty$), hence a point of multiplicity 3 and $\delta$-invariant  $3$. 
 
The curve of normals has an acnode $(-\frac{2}3:1:0)$ and two complex ordinary cusps  $(i:2i:1)$ and $(-i:-2i:1)$ (corresponding to the inflection points of $E_\Gamma$). 

For the Klein--Schuh formula applied to the evolute, we have that the degree and the class (equal to the degree of the curve of normals) are equal, and the curve of normals has one real singularity, an acnode, and the evolute has a cusp of multiplicity 3. Hence the formula checks:
$4-4=(2-0)-(3-1)=0$.
Since the curve of normals has no crunodes, and the evolute has no cusps in the affine real plane, the cissoid has no vertices or diameters, see Fig.~\ref{fig:conic}.

\begin{figure}[H]
     \centering
     \begin{subfigure}[t]{0.48\textwidth}
         \centering
      \includegraphics[width=0.52\columnwidth]{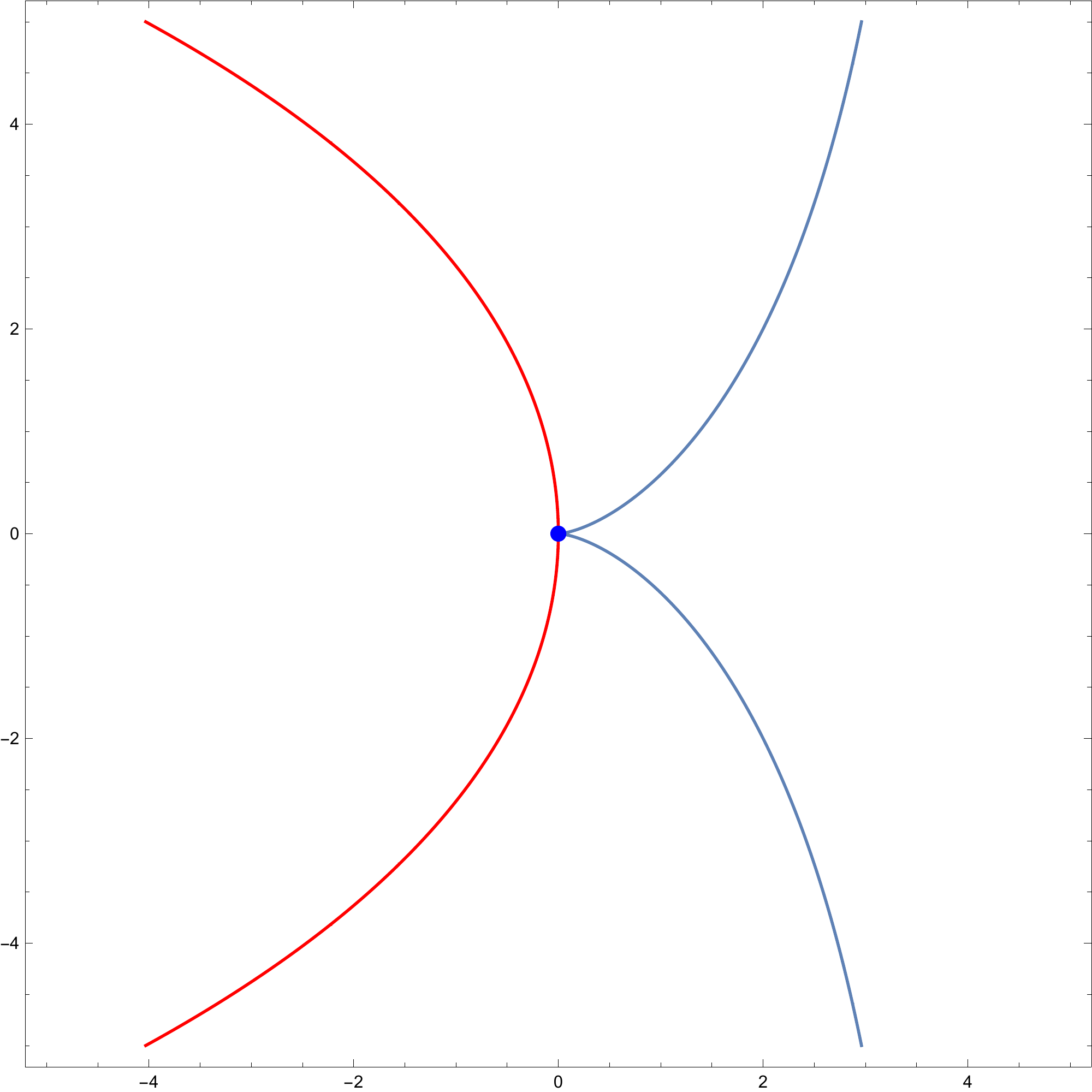} 
       \caption{A cissoid $(x^2+y^2)x=4y^2$ in blue and its evolute given by $27y^4+4608y^2+32768x=0$  in red.}\label{fig:conicleft}
        \label{fig:conic1a}
     \end{subfigure}
     \hfill
     \begin{subfigure}[t]{0.4\textwidth}
         \centering
        \includegraphics[width=0.62\columnwidth]{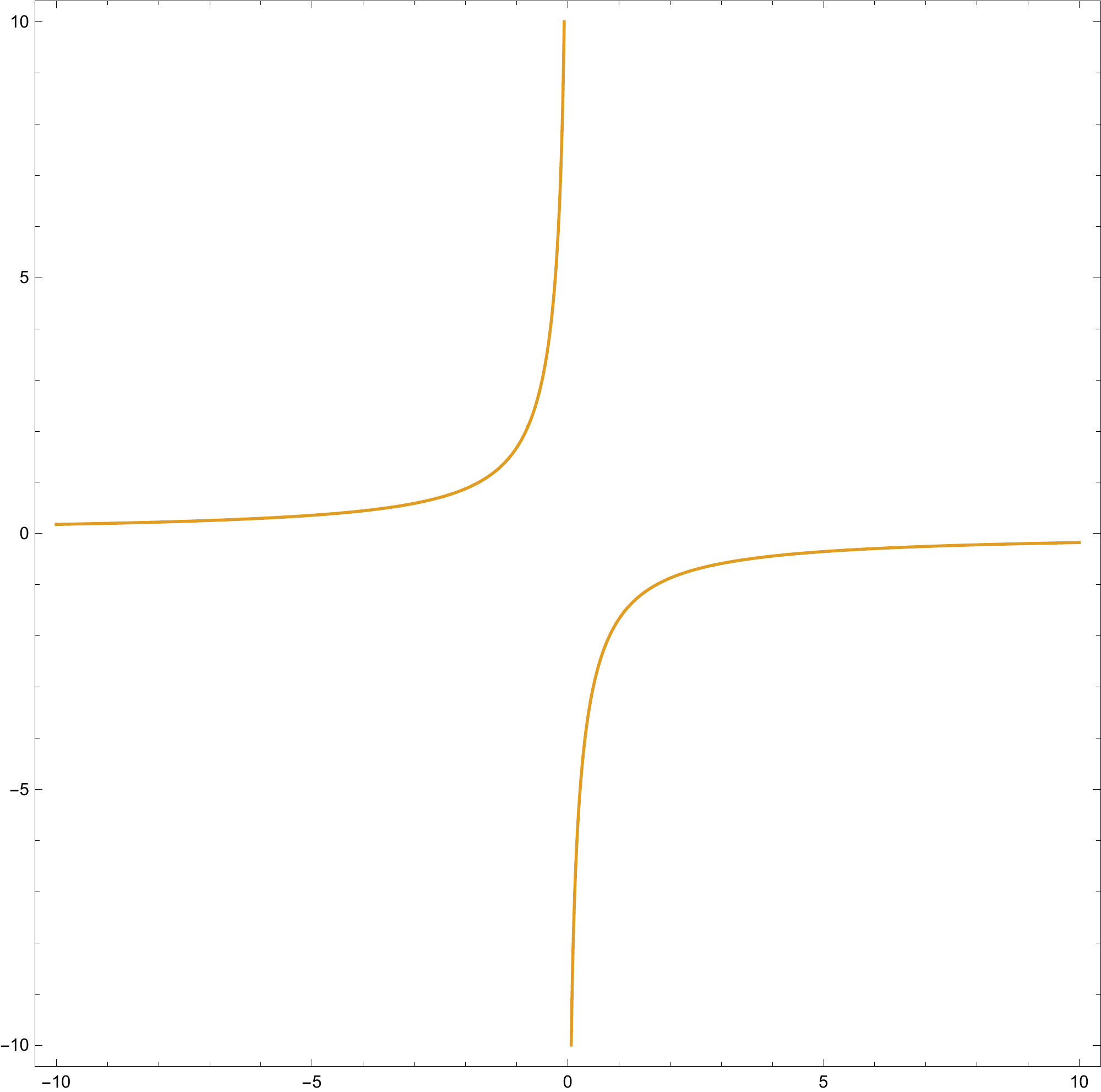}
        \caption{The corresponding curve of normals shown in gold is given by $36u^3v+uv^3+12u^2v^2+64u^2+96uv+128=0$.}\label{fig:conicright}
     \end{subfigure}
\caption{Cissoid, its evolute and its curve of normals.}   \label{fig:conic}
\end{figure}

 \smallskip

\noindent
{\bf IV.} The nodal cubic given by the equation
\[5(x^2-y^2)(x-1)+(x^2+y^2)=0\]
 has a crunode at the origin, and three real branches transversal to the line at $\infty$, hence is in general position with respect to the line at infinity and the circular points, except that the point $(0:1:0)$ is an inflection point. 
 Its numerical characters are $d=3$, $\delta=1$, $\kappa=0$, $\iota =3$, $d^\vee =3\cdot 2-2=4$. 
 
 The degree of its curve of normals is $\deg N_\Gamma=3+4=7$, and the degree of its evolute is $\deg E_\Gamma =3\cdot 3+3=12$, see column IV of Table~\ref{tb:inv}. 
 We find that the evolute has $11$ cusps in the affine plane, of which 5 are real. Furthermore, it has two real cusps and a real $E_6$-singularity at infinity. 
 (The reason the third cusp is an $E_6$-singularity rather than an ordinary cusp is that the curve has an inflection point at $(0:1:0)$, with tangent transversal to the line at infinity. This gives an ``extra'' tangent to the evolute passing through the point $(0:1:0)^\perp =(1:0:0)$, so that the formula for the degree of the evolute can be written $2\cdot 3 +(3+1)+ (\iota -1)=12$.)
 Furthermore $E_\Ga$ has $39$ complex nodes, of which $3$ are real (crunodes). 
 \begin{figure}[H]
     \centering
     \begin{subfigure}[t]{0.32\textwidth}
         \centering
      \includegraphics[width=0.95\columnwidth]{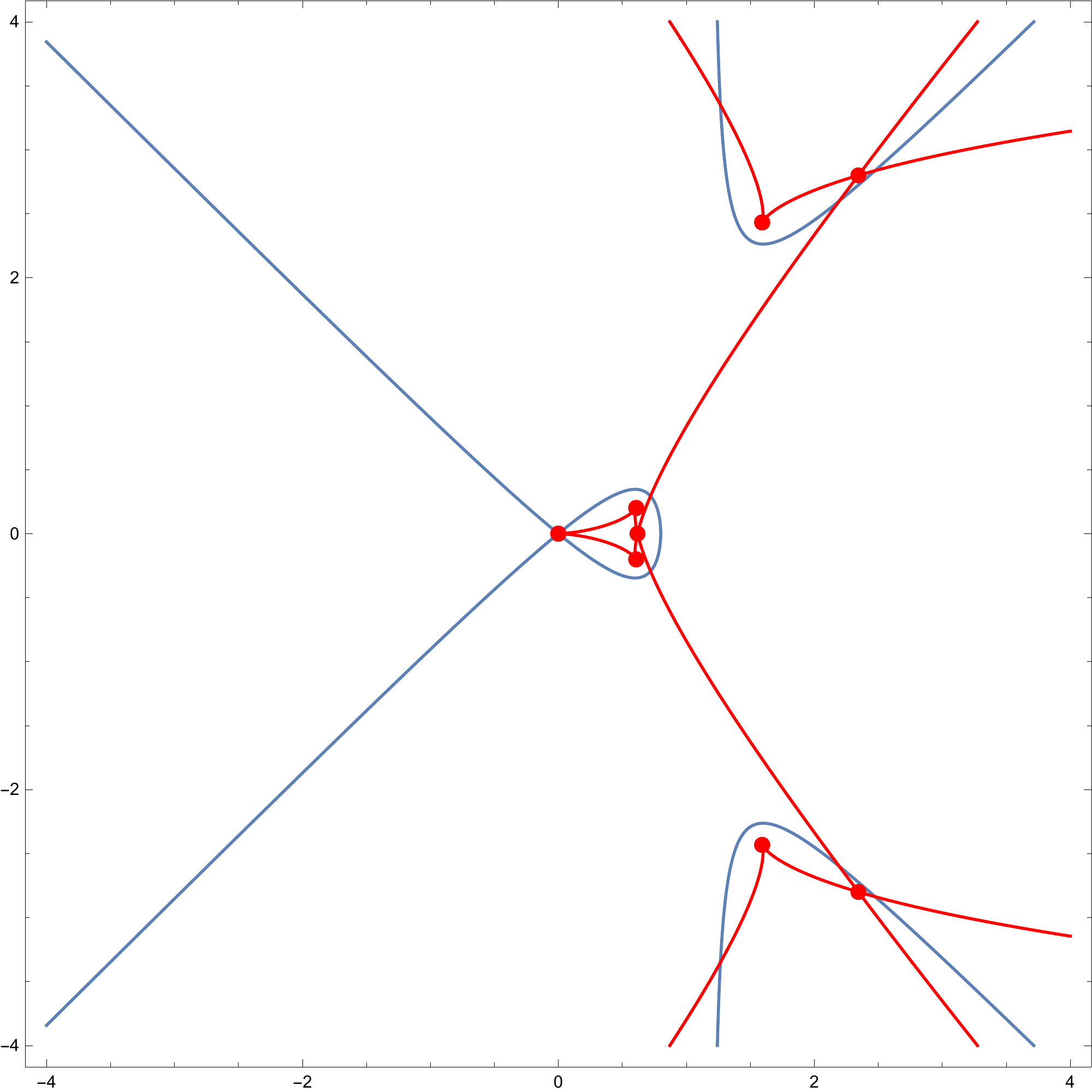} 
       \caption{The nodal cubic  in blue and its evolute  in red with marked singular points.}
       \label{fig:nodalcubic}
     \end{subfigure}
     \hfill
     \begin{subfigure}[t]{0.32\textwidth}
         \centering
        \includegraphics[width=0.95\columnwidth]{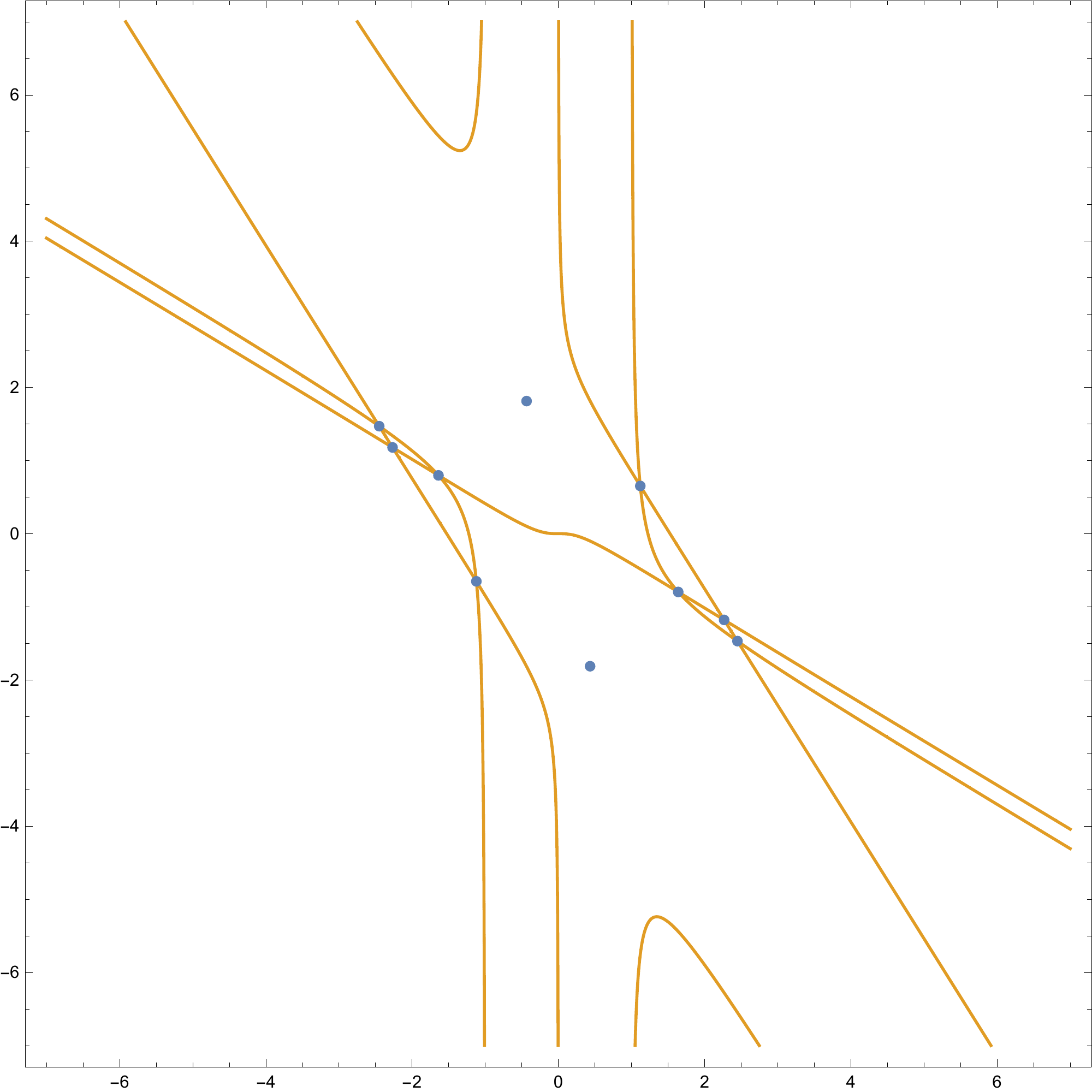}
        \caption{The curve of normals for the latter nodal cubic in the $(u,v)$ chart, with marked singular points.}\label{fig:nodalcubicmiddle}   
     \end{subfigure} \hfill
     \begin{subfigure}[t]{0.32\textwidth}
         \centering
        \includegraphics[width=0.95\columnwidth]{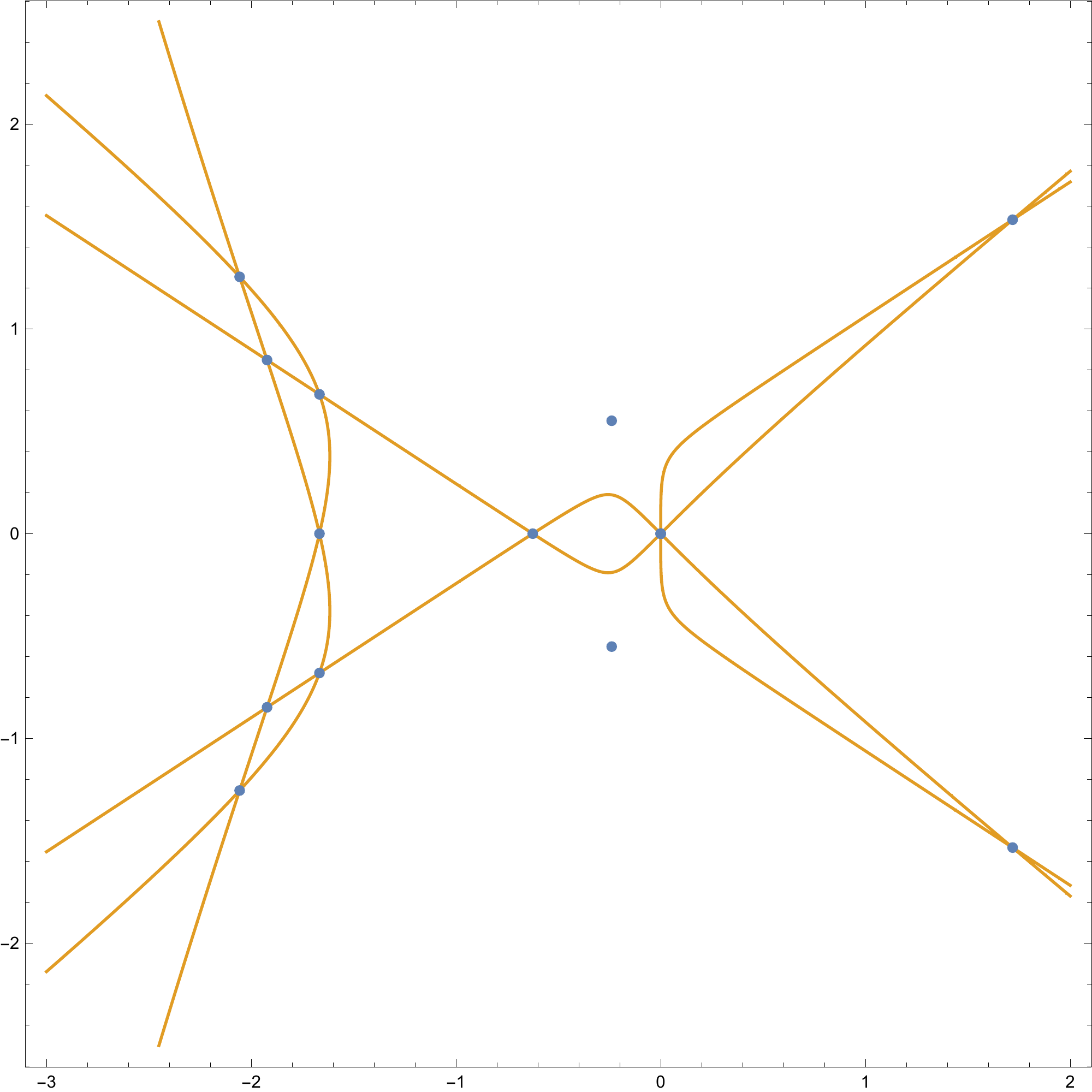}
        \caption{The curve of normals in the $(u,w)$  chart, with marked singular points.}\label{fig:nodalcubicright}
     \end{subfigure}
\caption{The nodal cubic $5(x^2-y^2)(x-1)+(x^2+y^2)=0$, its evolute, and its curve of normals.}   \label{fig:nodalcubic}
\end{figure}

Using Maple we determine that $N_\Ga$ has 13 singular points, all of which are real. 10 of these singular points can bee seen in Fig.~\ref{fig:nodalcubic}.
Furthermore, $N_\Ga$ has a triple points  at $(0:1:0)$ (which does not contribute to the count of diameters) and two double points at $(- \frac{5}{8}:1:0)$ and $(-\frac{5}{3}:1:0)$. All singular points can be seen in the different affine chart shown in  Fig.~\ref{fig:nodalcubicright}, and we see 10 crunodes and two acnodes.
Hence this nodal cubic has 5 vertices and 10 diameters.

\smallskip
\noindent
{\bf V.} Our next  example is a generic cubic in the Weierstrass form given by the equation 
\[y^2+x(x-2)(x+1)=0.\]
 The line at infinity is an inflectional tangent, at the point $(0:1:0)$. Hence the curve is not in general position. 
Its numerical characters are $d=3$, $\iota=9$, $d^\vee=6$. 

According to \cite[Thm.~8]{JoPe2}, the degree of its curve of normals is $\deg N_\Gamma=d+d^\vee - \iota_E$, with $\iota_E=3-1=2$, hence $\deg N_\Gamma=7$, which checks with the computation of the equation for $N_\Gamma$. The degree of its evolute is $\deg E_\Gamma=3d+\iota-3\iota_E=12$, which also checks with the computation of the equation for the evolute. According to Proposition~\ref{pr:7},
\[\kappa(E)=6d-3d^\vee+3\iota-5\iota_E=18-18+27-10=17,\]
see column V in Table~\ref{tb:inv}.  The evolute has an inflection point at $(0:1:0)^\perp =(1:0:0)$, with tangent the line at infinity that intersects the evolute with multiplicity 4. 

Indeed, the evolute has 17 cusps, of which 9 are real (we only see 7 of these in Fig.~\ref{fig:WeierNCleft}),
and $\binom{11}{2}-1-17=37$  nodes, of which 3 are real crunodes, see Fig.~\ref{fig:WeierNCleft}.

\begin{figure}[H]
     \centering
     \begin{subfigure}[t]{0.32\textwidth}
         \centering
      \includegraphics[width=0.95\columnwidth]{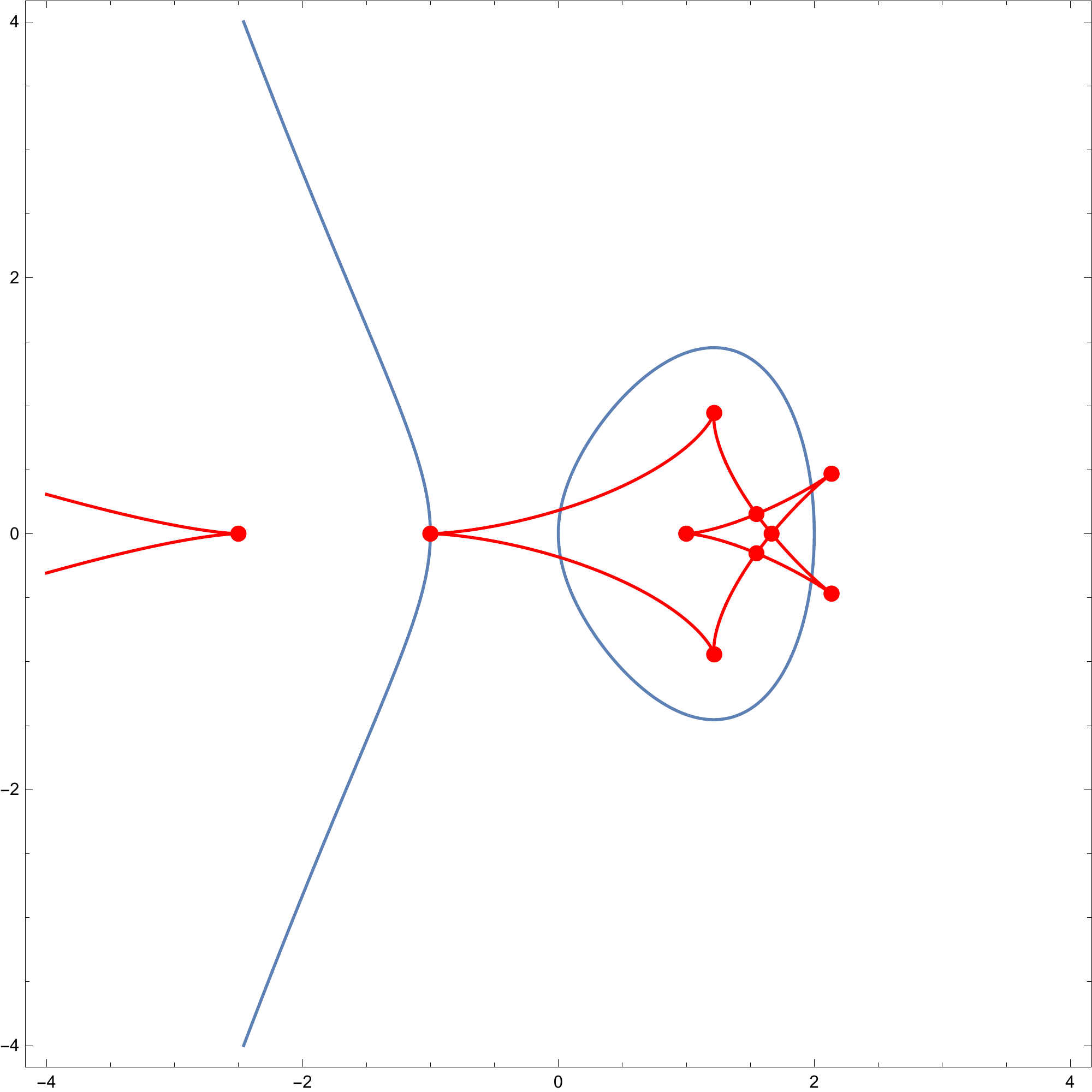} 
       \caption{The Weierstrass cubic  in blue and its evolute  in red.}\label{fig:WeierNCleft}
     \end{subfigure}
     \hfill
     \begin{subfigure}[t]{0.32\textwidth}
         \centering
        \includegraphics[width=0.95\columnwidth]{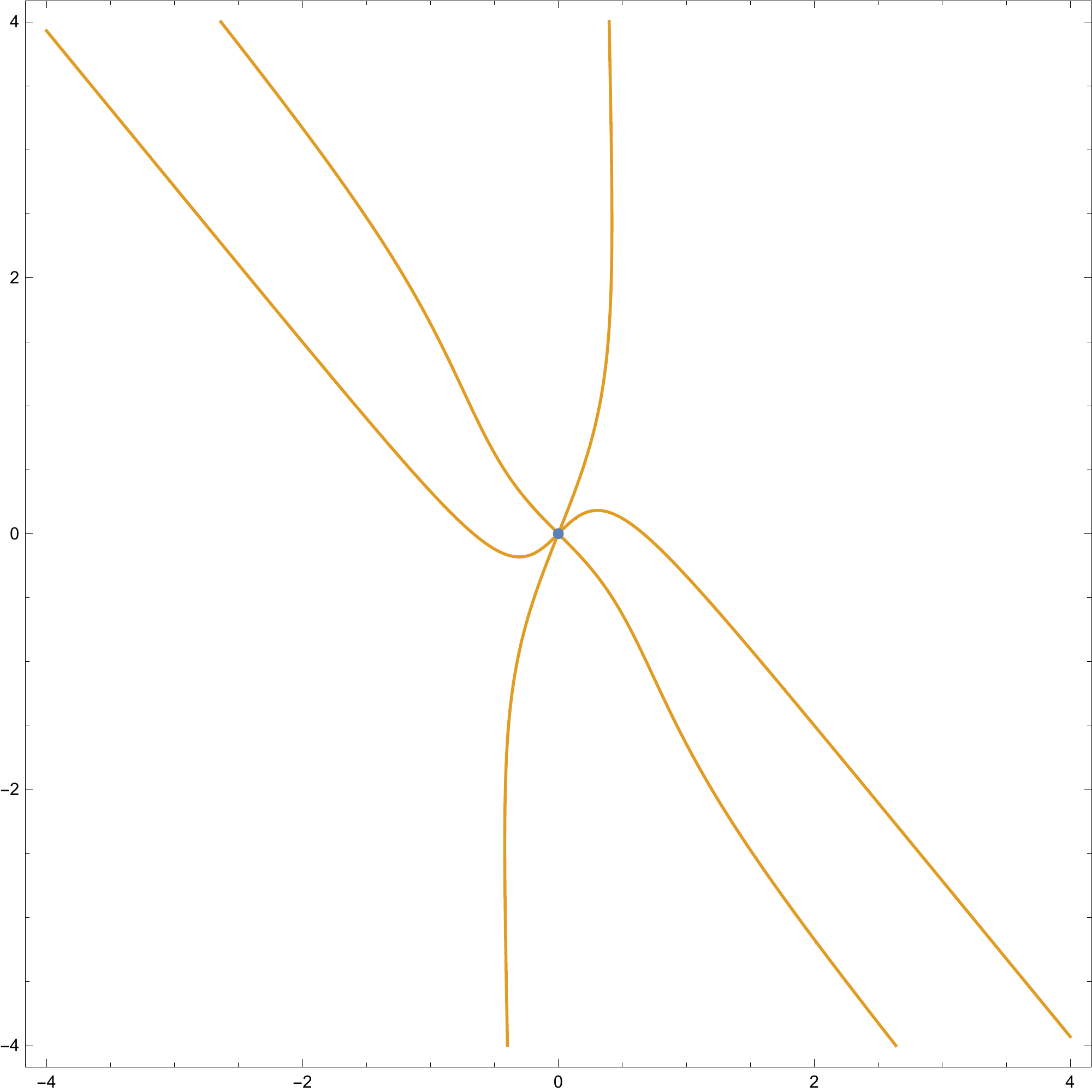}
        \caption{The curve of normals  in the $(u,v)$ chart.}\label{fig:WeierNCmiddle}   
     \end{subfigure} \hfill
     \begin{subfigure}[t]{0.32\textwidth}
         \centering
        \includegraphics[width=0.95\columnwidth]{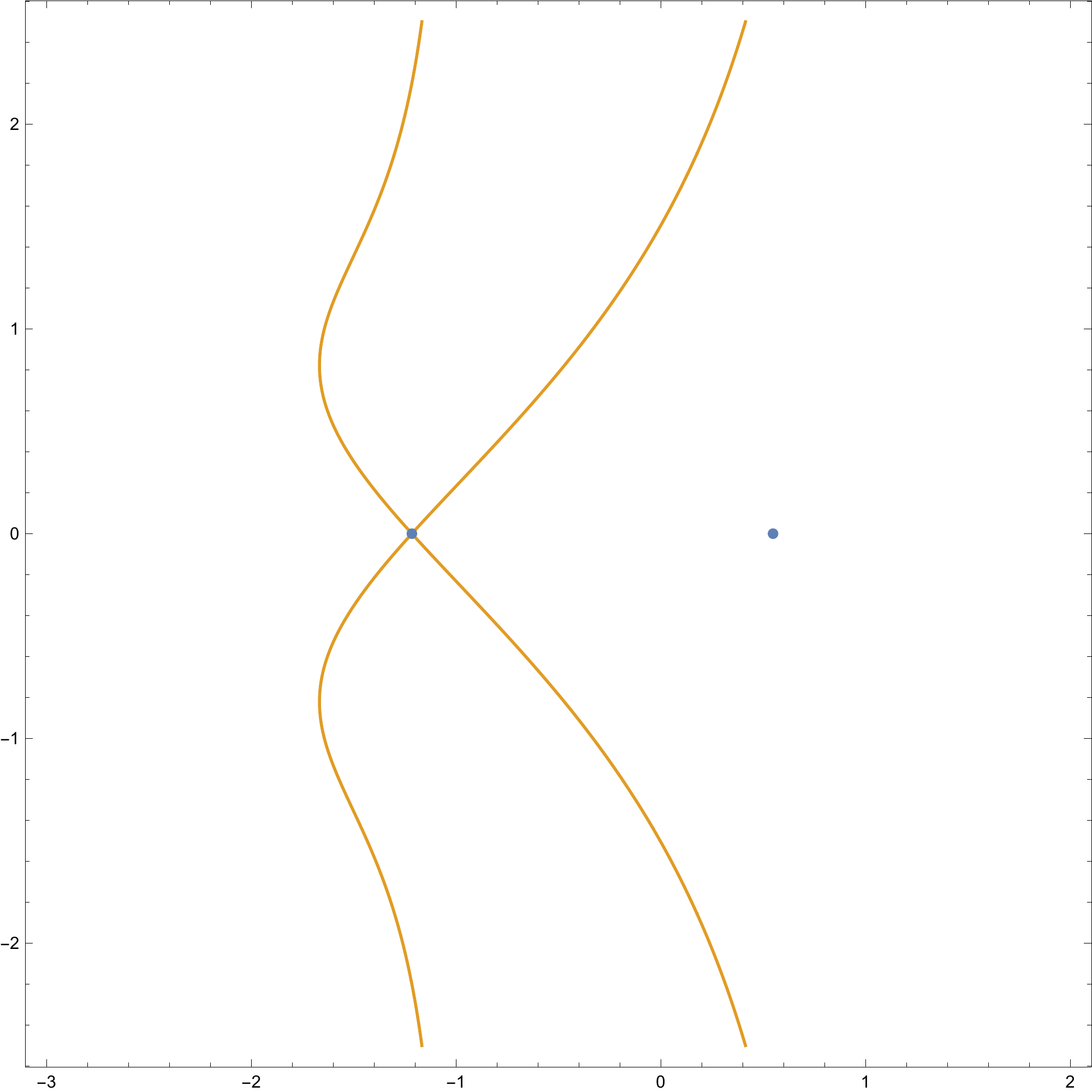}
        \caption{The curve of normals in the $(v,w)$ chart.}\label{fig:WeierNCright}
     \end{subfigure}
\caption{The Weierstrass cubic $y^2+x(x-2)(x+1)=0$, its evolute, and its curve of normals in different charts.}   \label{fig:WeierNC}
\end{figure}
  We see from Fig.~\ref{fig:WeierNCleft}  that $N_\Ga$ has a triple point at $u=0,v=0$. This comes from the fact that the line $y=0$ is perpendicular to $\Ga$ at the three distinct points $(-1: 0:1)$, $(0:0:1)$, $(2:0:1)$, and gives 3 diameters. Additionally, $N_\Ga$ has an $E_6$-singularity (corresponding to the inflection point of $E_\Ga$), with $\delta$-invariant 3. 
The remaining singular points of $N_\Ga$ are
$\binom{\deg N_\Ga -1}{2}-g-\binom{3}{2}-3=15-1-3-3=8$ nodes, of which 2 are real: 1 is an acnode and 1 is a crunode.

The Weierstrass cubic has 9 vertices and $4$ diameters.
\medskip

\noindent
{\bf VI.}  Our final example among  cubics is a nonsingular  cubic given by the equation
\[(x^2 - y^2) (x - 1)  + 1/64=0\]
 with three real branches transversal to the line at infinity, see Fig.~\ref{fig:GenCubicNC}. This is a curve in general position, but the three intersection points with the line at infinity are  inflection points. By Proposition~\ref{prop:1}, its curve of normals has degree $\deg N_\Gamma =3+6=9$ and its evolute has degree $\deg E_\Gamma=3d+\iota=18$.
 
Using Maple, we find that the evolute $E_\Ga$ has 21 ordinary complex cusps and 105 complex nodes, see column VI of Table~\ref{tb:inv}. The three cusps on the line at infinity are $E_6$-singularities, for the same reason as for the nodal cubic IV.
We find  18 real singular points in the $(x,y)$-plane, of which 9 are real cusps and 9 are crunodes 
(not all of them can be seen in Fig.~\ref{fig:GenCubicleft}). 

Using Maple, we find that the curve of normals  $N_\Ga$ has 2 real triple points, one at $(0:0:1)$ and one at $(0:1:0)$. The first contributes 3 diameters of the curve, whereas the second corresponds to the line $z=0$ and does not contribute to the diameters. Furthermore, we find  21 complex nodes, of which 
three are on the line $w=0$, namely at the real points $(\alpha:1:0)$, where $\alpha$ is any of the three real roots of $\phi(t)=t^3 + 128t^2 + 256t + 128$.  Of the remaining 18 nodes, 10 are real. The real affine situation in the $(u,v)$-chart is visible in Fig.~\ref{fig:GenCubicmiddle}, where we see two acnodes among the 10 visible real affine nodes. 
In Fig.~\ref{fig:GenCubicright}  we can see the triple point $(0:1:0)$ and two of the points $(\alpha:1:0)$ -- the third of these, an acnode, is further out and not shown to allow for a detailed vision around the origin. We thus have that 3 of the  13 real nodes  are acnodes. 

The cubic has 9 vertices and 13 diameters.

\begin{figure}[H]
     \centering
     \begin{subfigure}[t]{0.32\textwidth}
         \centering
      \includegraphics[width=0.95\columnwidth]{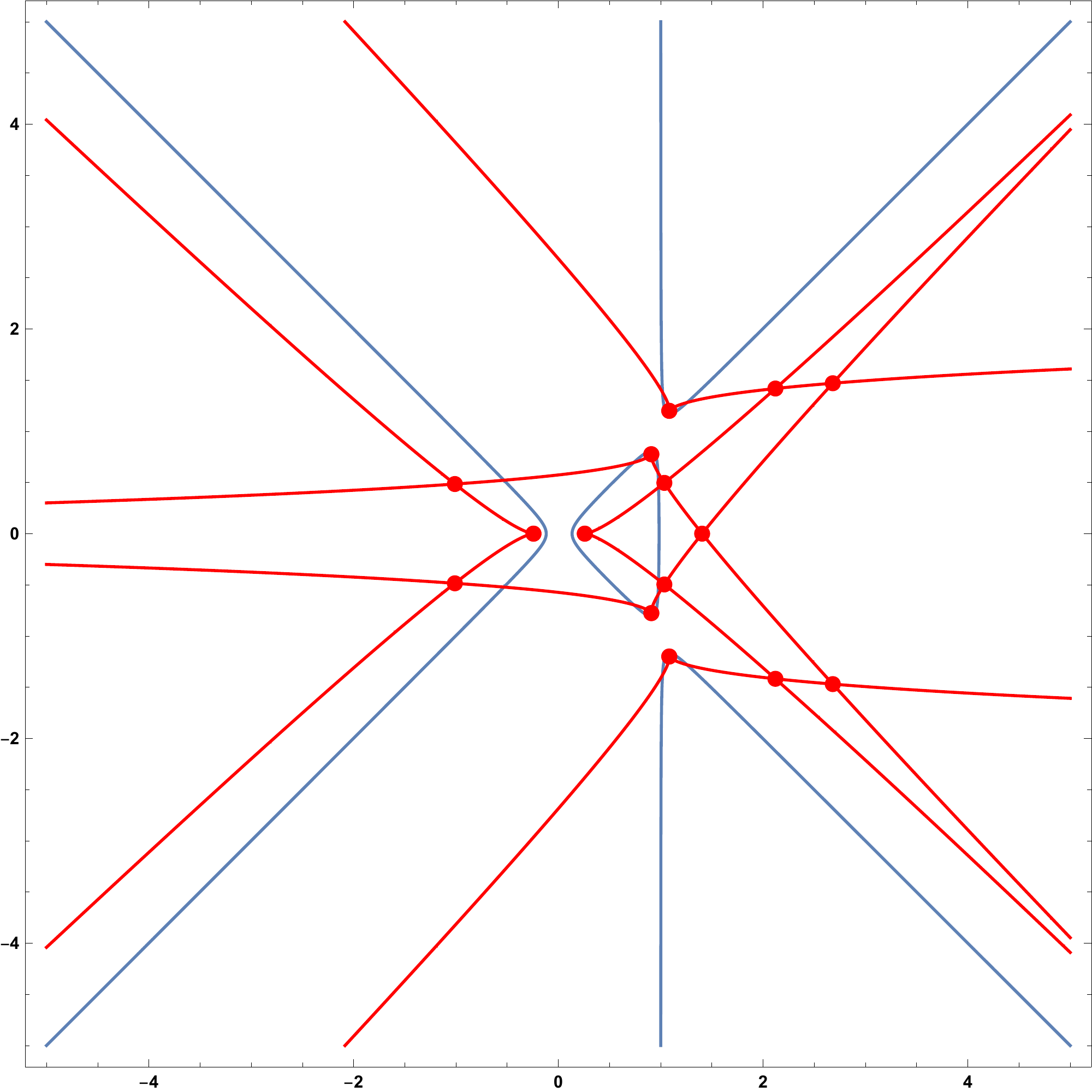} 
       \caption{The cubic  in blue and its evolute  in red.}\label{fig:GenCubicleft}
     \end{subfigure}
     \hfill
     \begin{subfigure}[t]{0.32\textwidth}
         \centering
        \includegraphics[width=0.95\columnwidth]{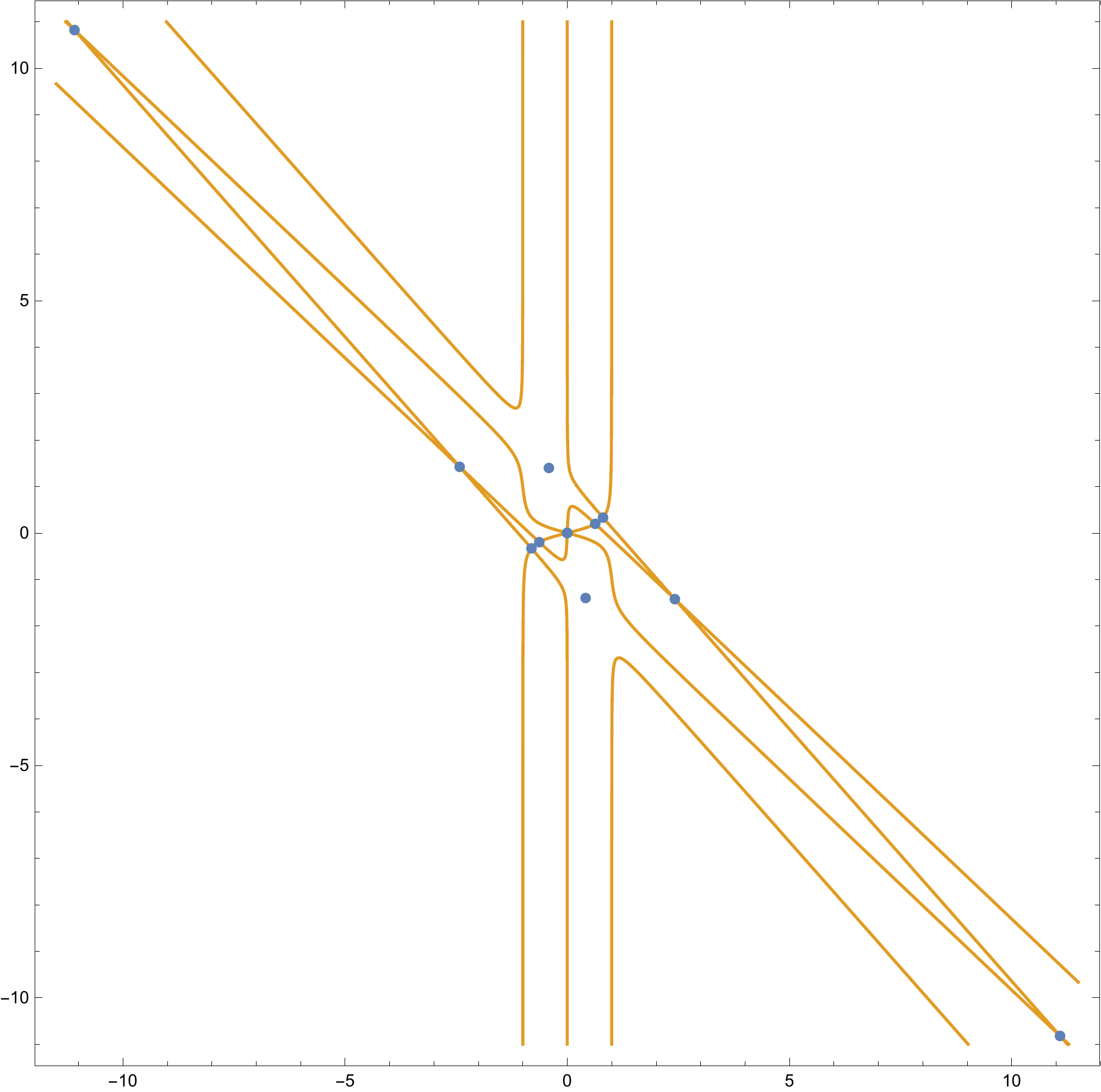}
        \caption{The curve of normals in the $(u,v)$ chart, with marked singular points.}\label{fig:GenCubicmiddle}   
     \end{subfigure} \hfill
     \begin{subfigure}[t]{0.32\textwidth}
         \centering
        \includegraphics[width=0.95\columnwidth]{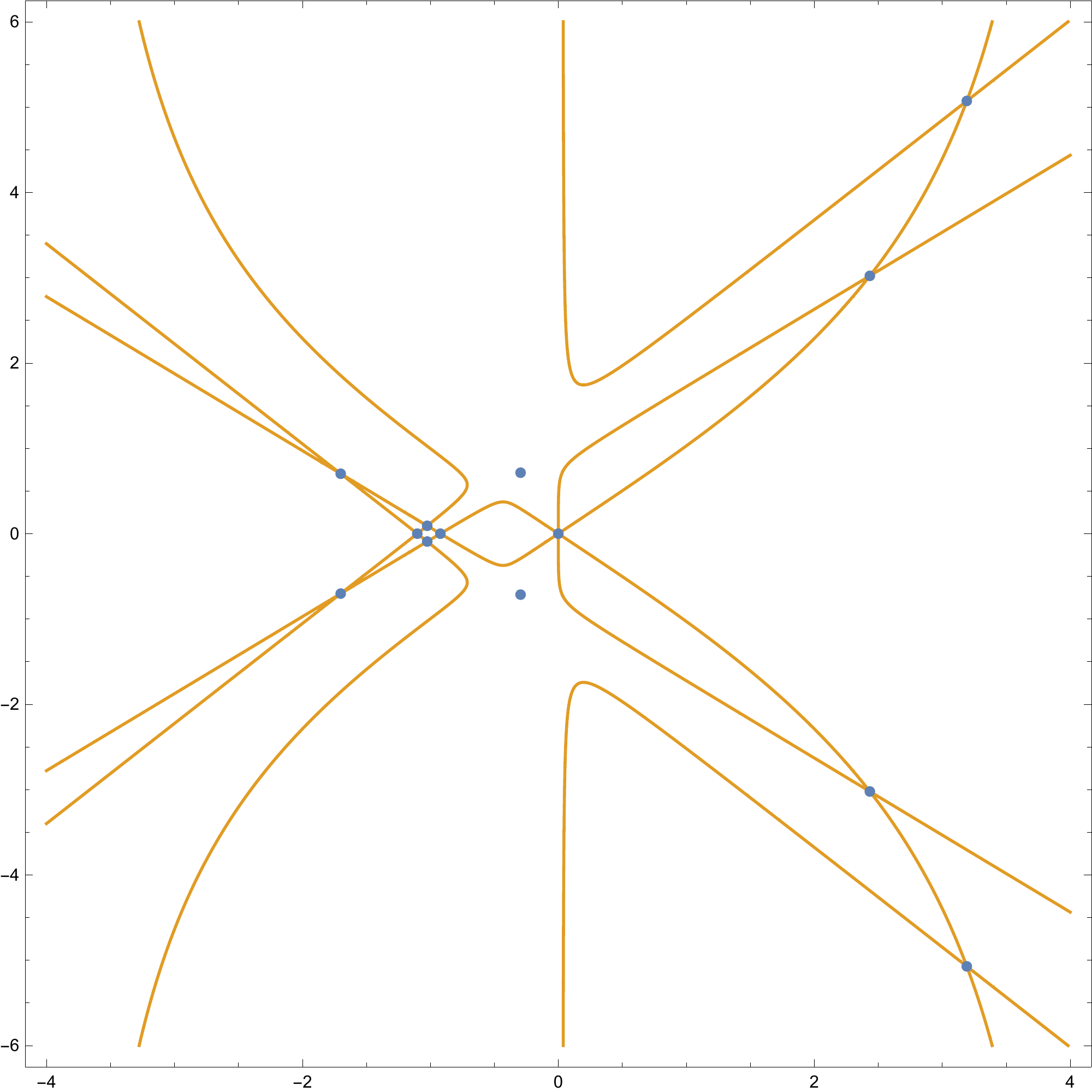}
        \caption{The curve of normals in the $(u,w)$ chart, with marked singular points.}\label{fig:GenCubicright}
     \end{subfigure}
\caption{A nonsingular cubic with three real branches transversally intersecting the line at infinity, its evolute, and its curve of normals in two charts.}   \label{fig:GenCubicNC}
\end{figure}

\subsection{Rational curves of higher degree}\label{subs:higher} \hfill\\

\noindent
{\bf VII.}  The ampersand curve is a rational curve given by the equation 
\[\textstyle (-1 + x) (-3 + 2 x) (-x^2 + \frac{1}{2}y^2) - 4 (-2 x + x^2 + \frac{1}{2}y^2)^2=0.\]   
It intersects the line at infinity transversally in four (non-circular) points, hence is in general position.
Its numerical characters are
$d=4, d^\vee =6, \delta=3, \kappa=0, \iota=6, \tau =4$. 
We see 2 real inflection points on the curve, so by Klein's formula $4+2\tau'+\iota'=6+2\delta'+\kappa'$, we get
$4+2\tau'+2= 6+0+0$ (since all nodes are crunodes, $\delta'=0$),
hence $\tau'=0$ (i.e., there are no conjugate double tangents).

The degree of its curve of normals is
$\deg N_\Gamma=4+6=10$ and that of its evolute is $\deg E_\Gamma=12 +6= 18$.
We get
$ \kappa(N)=0$, $\delta_N=\delta(N)-\binom{6}2=36-6=30$ and
$\kappa(E)=24$, $\delta_E=\delta(E)-\kappa(E)=136-24=112$, see column VII of Table~\ref{tb:inv}. 

The evolute has 24 ordinary cusps, of these 6 are real and in the $(x,y)$-plane. It has 112 nodes, of these 4 are crunodes and 4 are acnodes (except for 2 acnodes, these can all be seen in Fig.~\ref{fig:ambersandleft}). 

The curve of normals has a quadruple point at $(0:1:0)$.  Using Maple we find that $N_\Ga$ has additional 30 singular points, all nodes. Of these, 14 are real -- 11 of these points can bee seen in Fig.~\ref{fig:ambersandright}, and
the nodes at $(-2:1:0)$ and $(-159/209 \pm \sqrt\frac{201}{209}: 1:0)$ can be seen in Fig.~\ref{fig:ambersvw}. Hence we can see all the real singular points 
by considering the 
two affine charts shown in  Fig.~\ref{fig:ambersand}. The point $(0:0:1)$ is an acnode. Note that the ``acnode''  $(0:1:0)$ we see in 
Fig.~\ref{fig:ambersvw} is the ordinary quadruple point corresponding to the line $z=0$ -- this point has four complex branches.
Thus we see that $N_\Ga$ has 13 crunodes.

The ampersand curve has 6 vertices and 13 diameters.

\begin{figure}[H]
     \centering
     \begin{subfigure}[t]{0.32\textwidth}
         \centering
      \includegraphics[width=0.95\columnwidth]{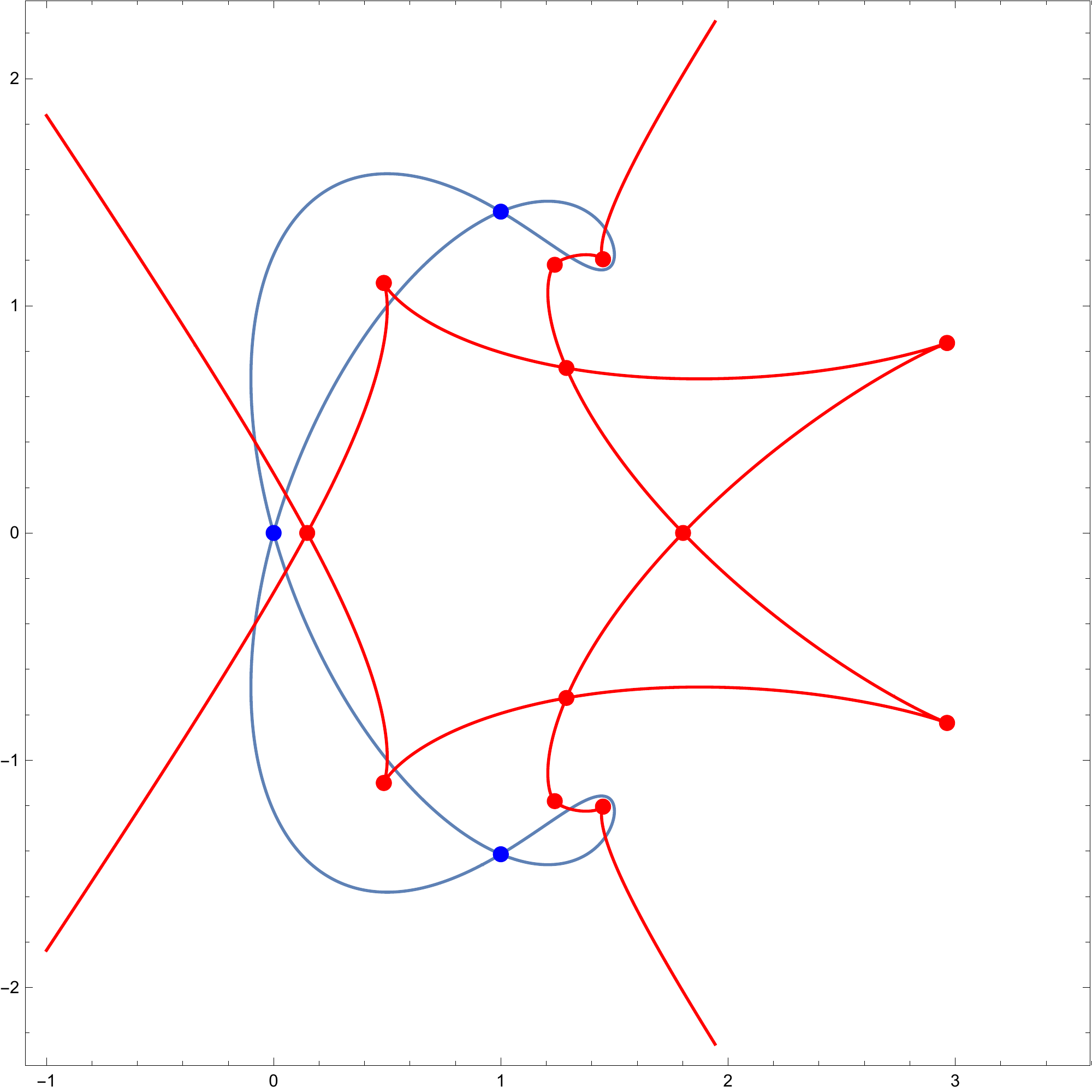} 
       \caption{The ampersand curve in blue with its evolute in red.}\label{fig:ambersandleft}
     \end{subfigure}
     \hfill
     \begin{subfigure}[t]{0.32\textwidth}
         \centering
        \includegraphics[width=0.95\columnwidth]{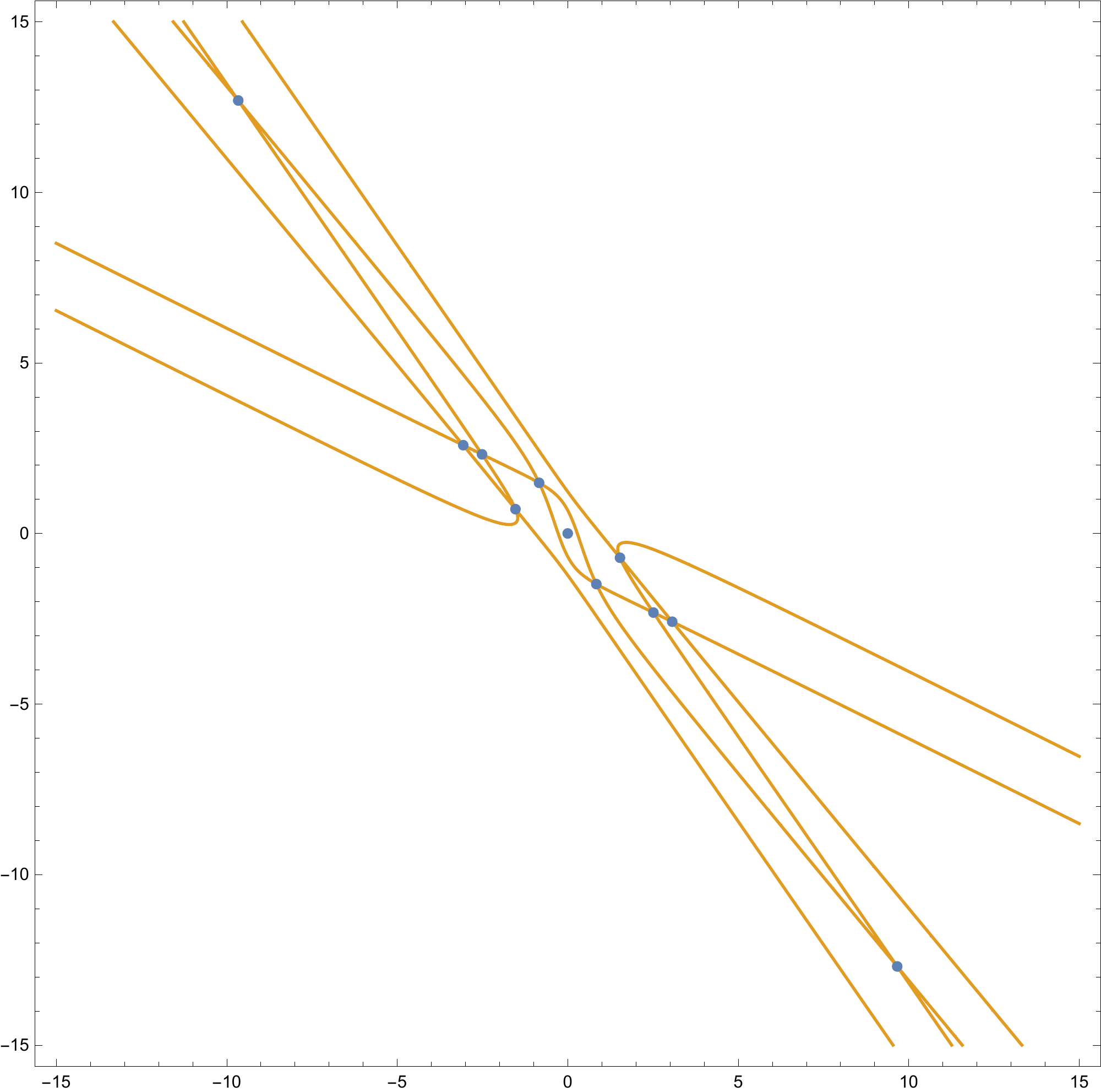}
        \caption{The curve of normals in the $(u,v)$ chart.}\label{fig:ambersandright}
     \end{subfigure} \hfill
      \begin{subfigure}[t]{0.32\textwidth}
         \centering
        \includegraphics[width=0.95\columnwidth]{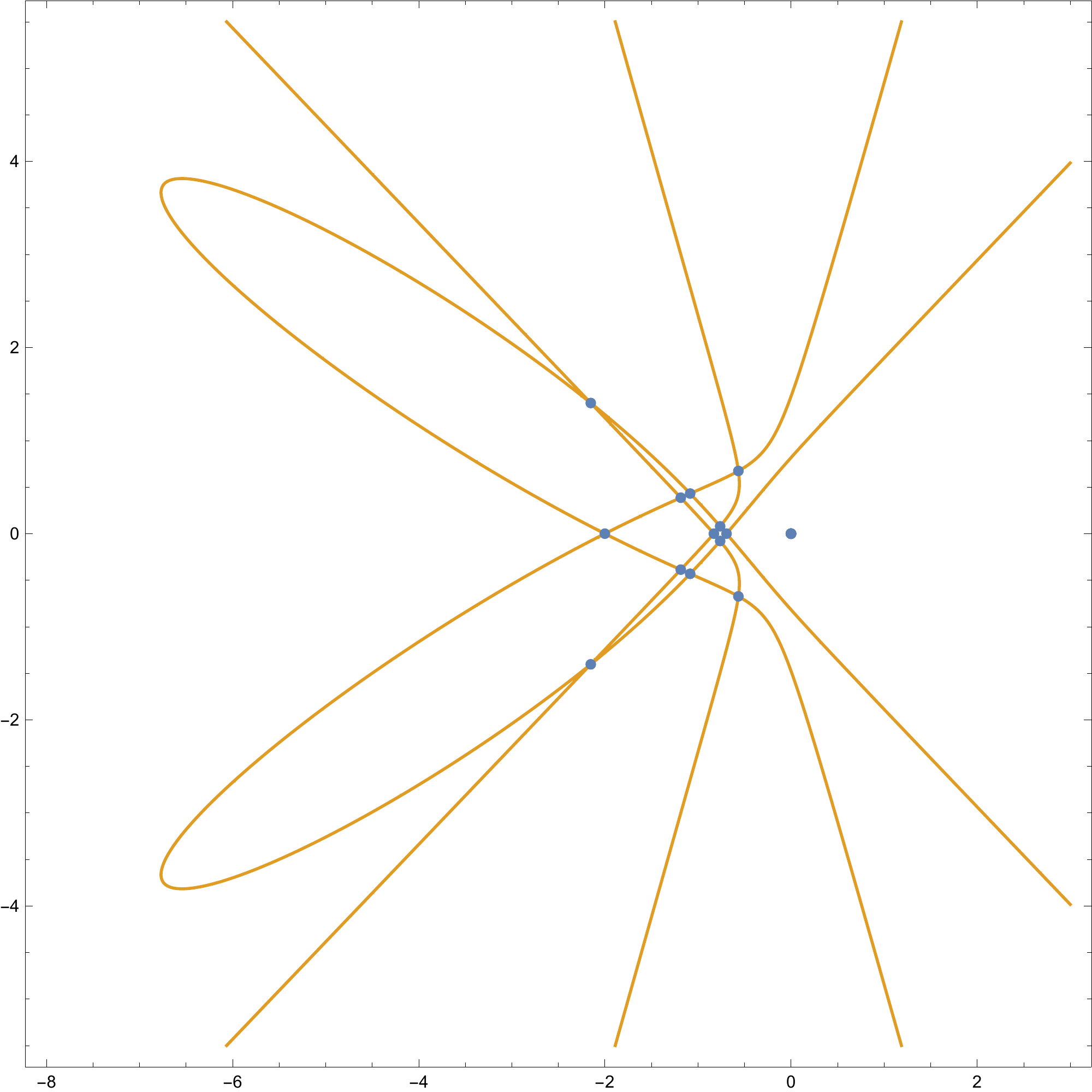}
        \caption{The curve of normals in the $(u,w)$ chart, with marked singular points.}\label{fig:ambersvw}
     \end{subfigure}\hfill
  \caption{The ampersand curve, its evolute, and its curve of normals.}   \label{fig:ambersand}
\end{figure}
\medskip

\smallskip

\noindent
{\bf VIII.} The cross curve is a rational curve given by the equation
\[x^2y^2-4x^2-y^2=0.\]

The curve is not in general position: it intersects the line at infinity in two points that are crunodes of the curve, with tangents transversal to the line at infinity, and with each branch of the crunodes having an inflection point at the crunode. The curve has an acnode at the origin $(0:0:1)$.
Its numerical characters are
$d=4$, $d^\vee=6$, $\delta=3$, $\kappa =0$, $\iota=6$, $\tau=4$. 
Klein's formula gives $4+2\tau' +\iota' =6+2\delta'+\kappa'$, hence $4+2\tau'=6+2$ (since one node is an acnode, $\delta'=1$), hence $\tau'=2$. 

We get
$\deg N_\Ga=d+d^\vee=10$, $\kappa(N)=0$, $\delta_N=\delta(N)-\binom{4}2=36-6=30$, and
$\deg E_\Ga=3d+\iota=18$.

The evolute has 16 ordinary cusps, of which 4 are real, and two real singularities on the line at infinity, each with multiplicity 6, $\delta$-invariant 18, and 2 branches. (Each of these singularities is the coming-together of two $E_6$-singularities.) The evolute has $\kappa(E)=24$, and $\delta_E=\delta(E)-16-2\cdot 18=84$ nodes, of these two are real and are acnodes. 

The curve of normals has a quadruple point at $(0:1:0)$ (which does not contribute to the number of diameters) with four real branches and $\delta$-invariant 12 -- it can be seen in Fig. \ref{fig:crossvw}. Furthermore, using Maple, we find that the curve of normals  has 24 double points in the affine $(u,v)$-plane, namely $(\pm \sqrt{2}:0:1)$, $(\pm \sqrt{-2}:0:1)$, and  $(\pm \sqrt{\alpha}: \pm\sqrt{\gamma(\alpha)}:1)$, where $\alpha$ is any of the five roots of $\phi(t)= t^5 - 6t^4 + 44t^3 - 56t^2 + 24t - 4$  and $\gamma(\alpha)=\frac{1}{3}(84+6\alpha^4-33\alpha^3+252\alpha^2-219\alpha$.  Thus in total we have 6 real affine nodes in the $(u,v)$-chart,  4 of which are acnodes, as can be seen in Fig. \ref{fig:crossright}.

The cross curve has 4 vertices and 2 diameters.

\begin{figure}[H]
     \centering
     \begin{subfigure}[t]{0.32\textwidth}
         \centering
      \includegraphics[width=0.95\columnwidth]{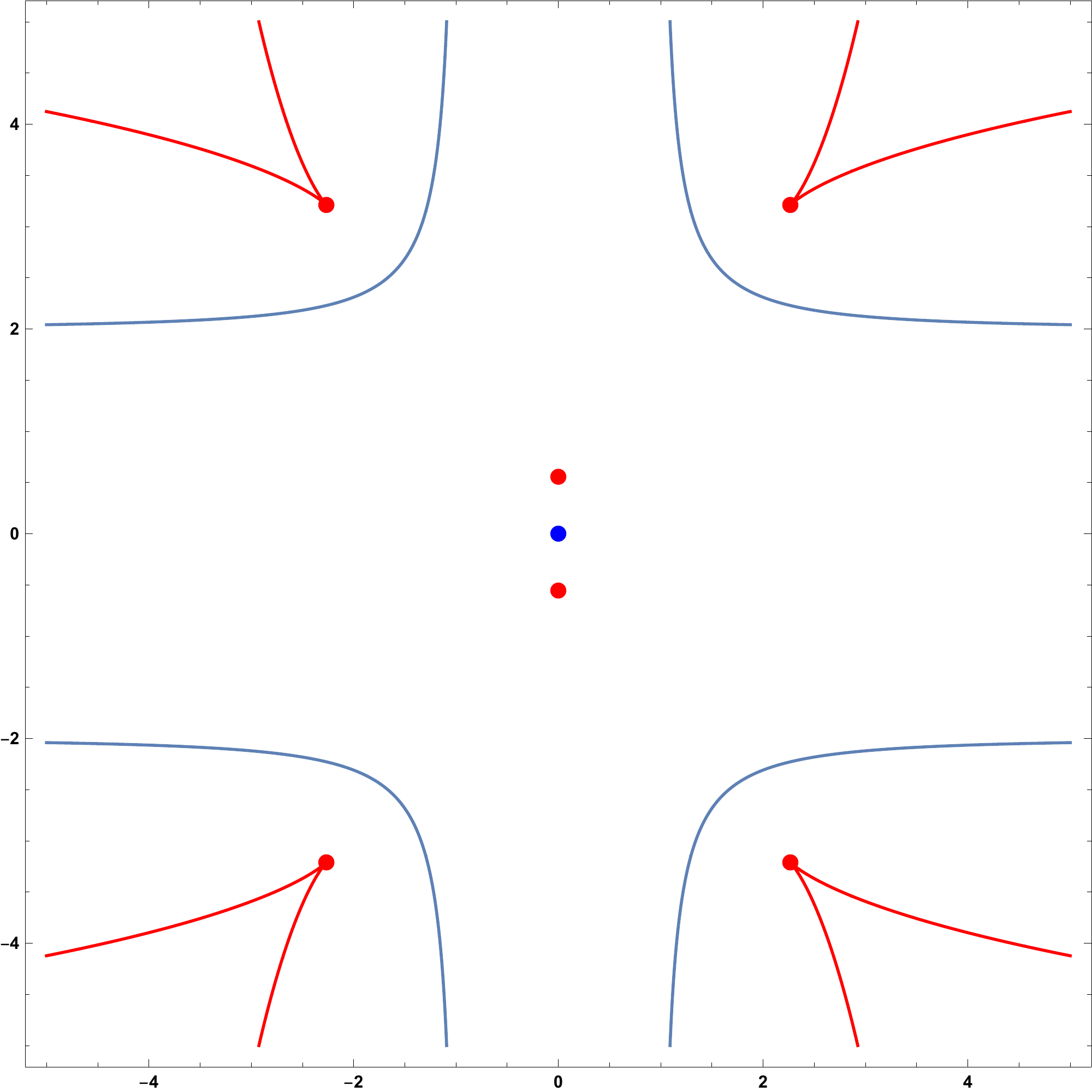} 
       \caption{The cross curve in blue with its evolute in red with marked singular points.}\label{fig:crossleft}
     \end{subfigure}
     \hfill
     \begin{subfigure}[t]{0.32\textwidth}
         \centering
        \includegraphics[width=0.95\columnwidth]{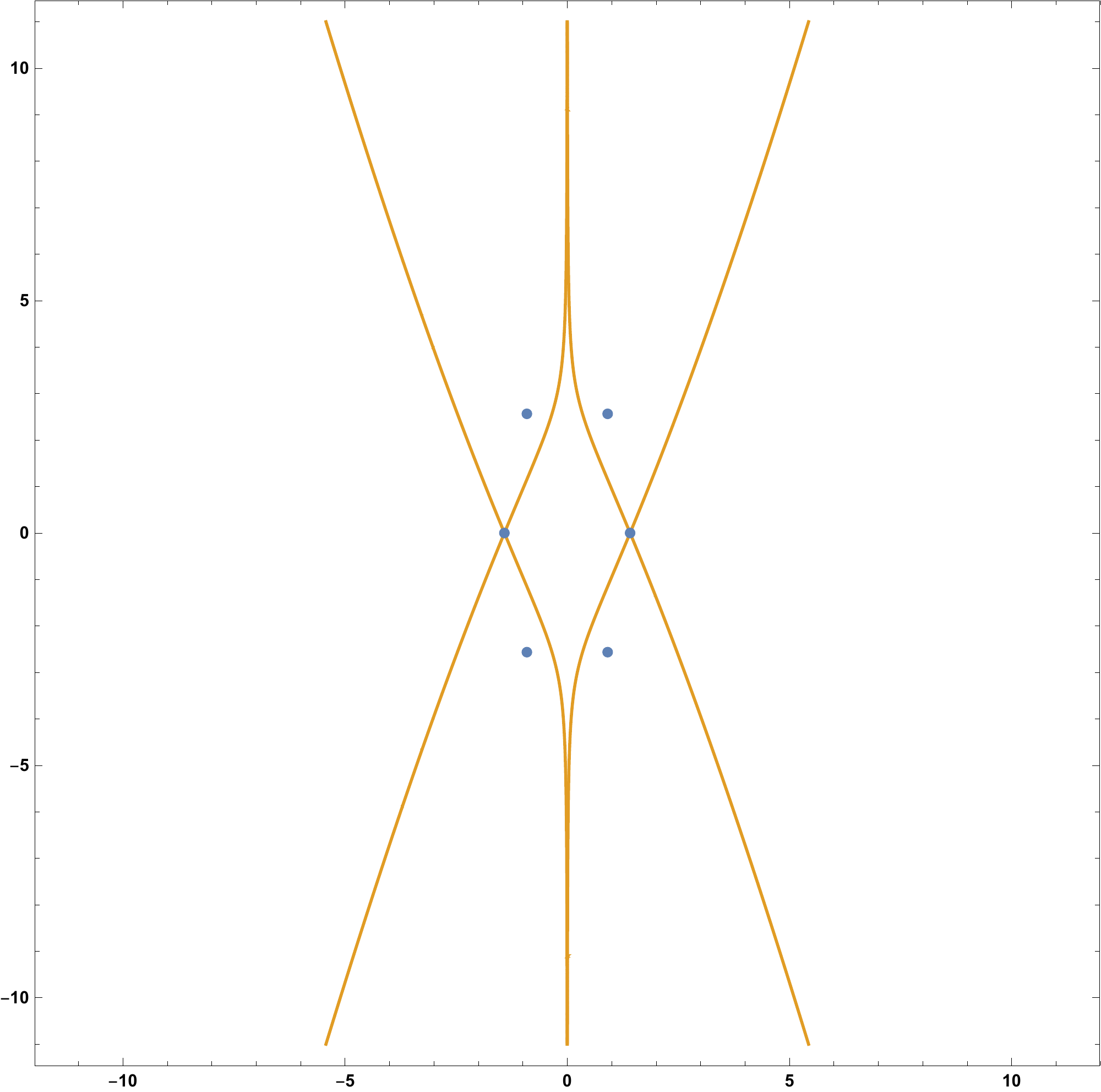}
        \caption{The curve of normals in $(u,v)$ charts with marked singular points.}\label{fig:crossright}
     \end{subfigure} \hfill
       \begin{subfigure}[t]{0.32\textwidth}
         \centering
        \includegraphics[width=0.95\columnwidth]{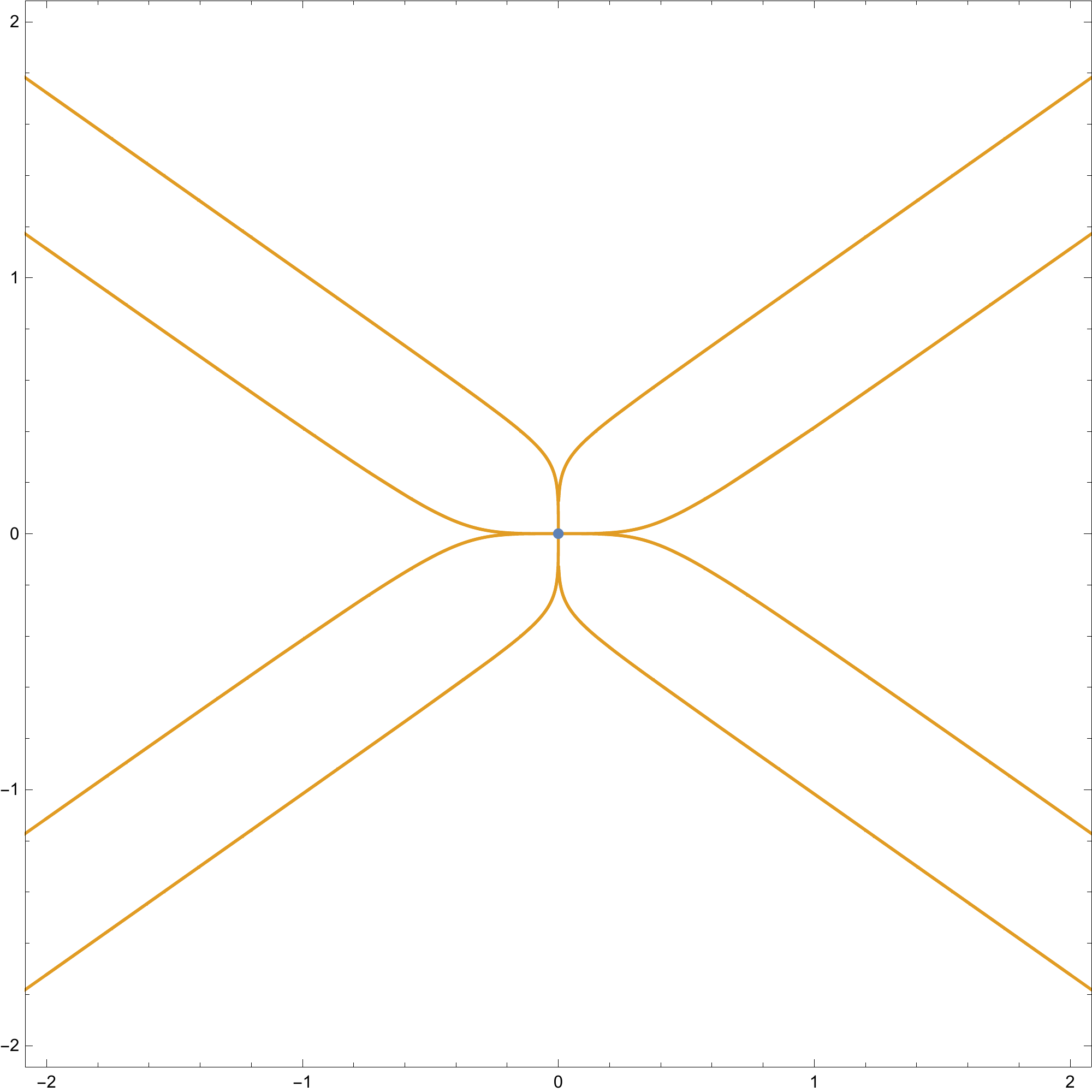}
        \caption{The curve of normals in the $(u,w)$ chart with marked singular points.}\label{fig:crossvw}
     \end{subfigure}\hfill
\caption{The cross curve, its evolute, and its curve of normals.}\label{fig:cross}
\end{figure}
\medskip

\smallskip
\noindent
{\bf IX.} The bean curve  is a rational curve given by the equation
\[x^4+x^2y^2+y^4-x(x^2+y^2)=0.\]
The curve intersects the line at infinity in four distinct (non-circular) points, so it is in general position. It has an
ordinary triple point at the origin, with one real branch. The point $(1:0:1)$ is an inflection point with tangent line $x=1$. The intersection number of the curve with its tangent at this point is 4. The corresponding point on its dual curve is an $E_6$-singularity (a cusp of multiplicity 3), hence $\iota' =2$.  The numerical characters of the curve are
$d=4$, $d^\vee=6$, $\delta=3$, $\kappa =0$, $\iota=6$, $\tau=4$. Klein--Schuh's formula gives $4-6=(3-1) -\iota'-2\tau'=2-2-2\tau'$, hence $\tau'=1$.

 Its curve of normals
has 
$\deg N_\Ga=10$, $\kappa(N)=0$, and $\delta(N)=36$ nodes, and its evolute has
$\deg E_\Ga=18$, $\kappa(E)=24$, $\delta(E)=136$, see column IX  of Table~\ref{tb:inv}. 

The evolute has 124 nodes and 24 cusps. There are 5 real cusps in the $(x,y)$-plane and one at $(0:1:0)$. This last one corresponds to the higher order inflection point $(1:0:1)$ of the curve, which gives a critical point of the curvature radius function and hence this cusp is a vertex of $\Ga$. We see 3 crunodes and 5 acnodes in Fig. \ref{fig:beanleft}.

The curve of normals has an ordinary quadruple point at $(0:1:0)$, with all branches complex. In addition, $N_\Ga$ has a crunode visible in Fig.~\ref{fig:beanright},  and singularities at  $(-\frac{3}{2}:1:0)$, $(1 - \sqrt{5}: 1:0)$, $(1 + \sqrt{5}:1:0)$, which can be seen in the $(u,w)$-chart shown in Fig.~\ref{fig:beanvw}, which gives two acnodes and an additional crunode. 

The bean curve has 6 vertices and 2 diameters.

\begin{figure}[H]
     \centering
     \begin{subfigure}[t]{0.32\textwidth}
         \centering
      \includegraphics[width=0.95\columnwidth]{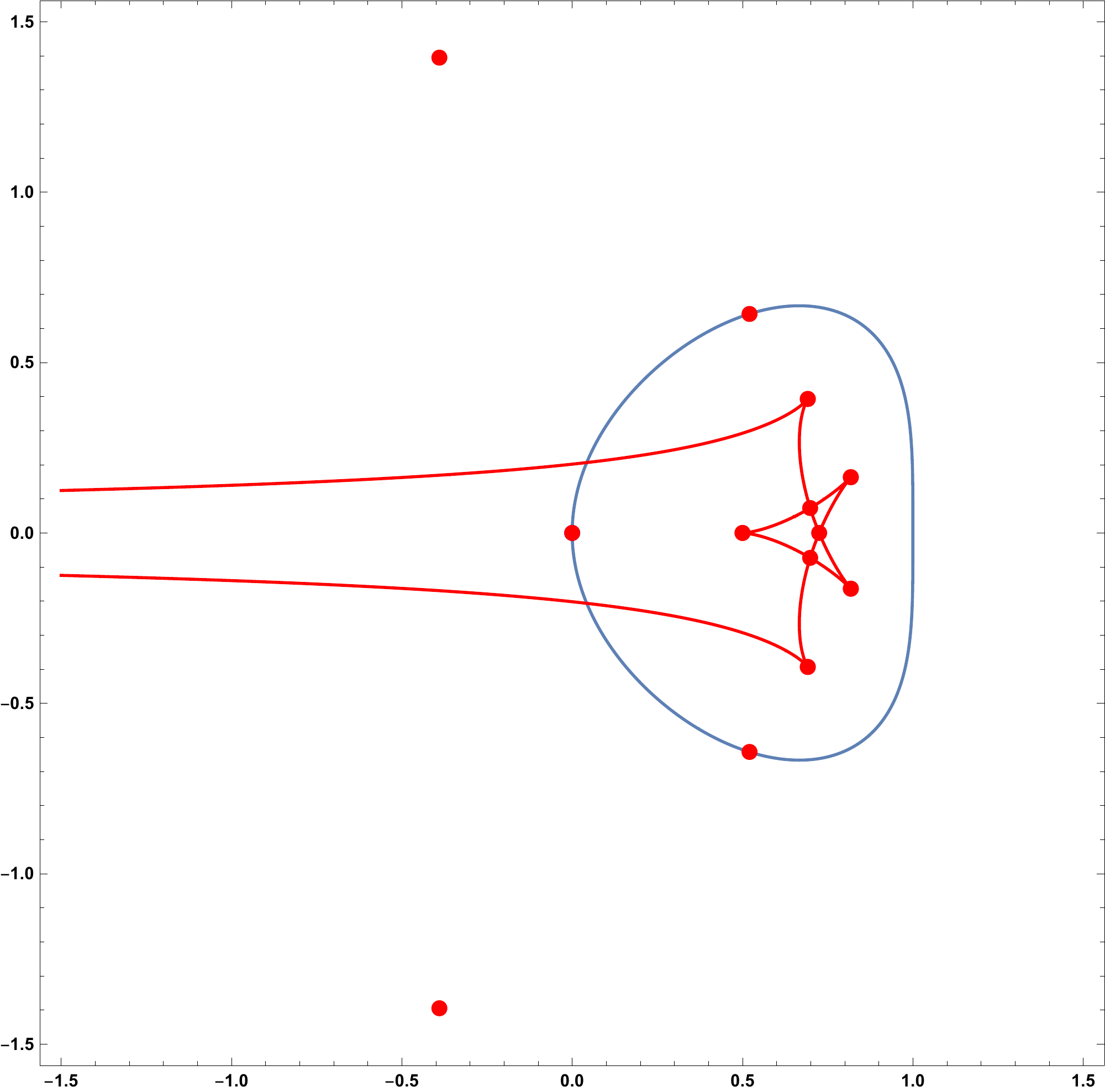} 
       \caption{The bean curve in blue with its evolute in red.}\label{fig:beanleft}
     \end{subfigure}
     \hfill
     \begin{subfigure}[t]{0.32\textwidth}
         \centering
        \includegraphics[width=0.95\columnwidth]{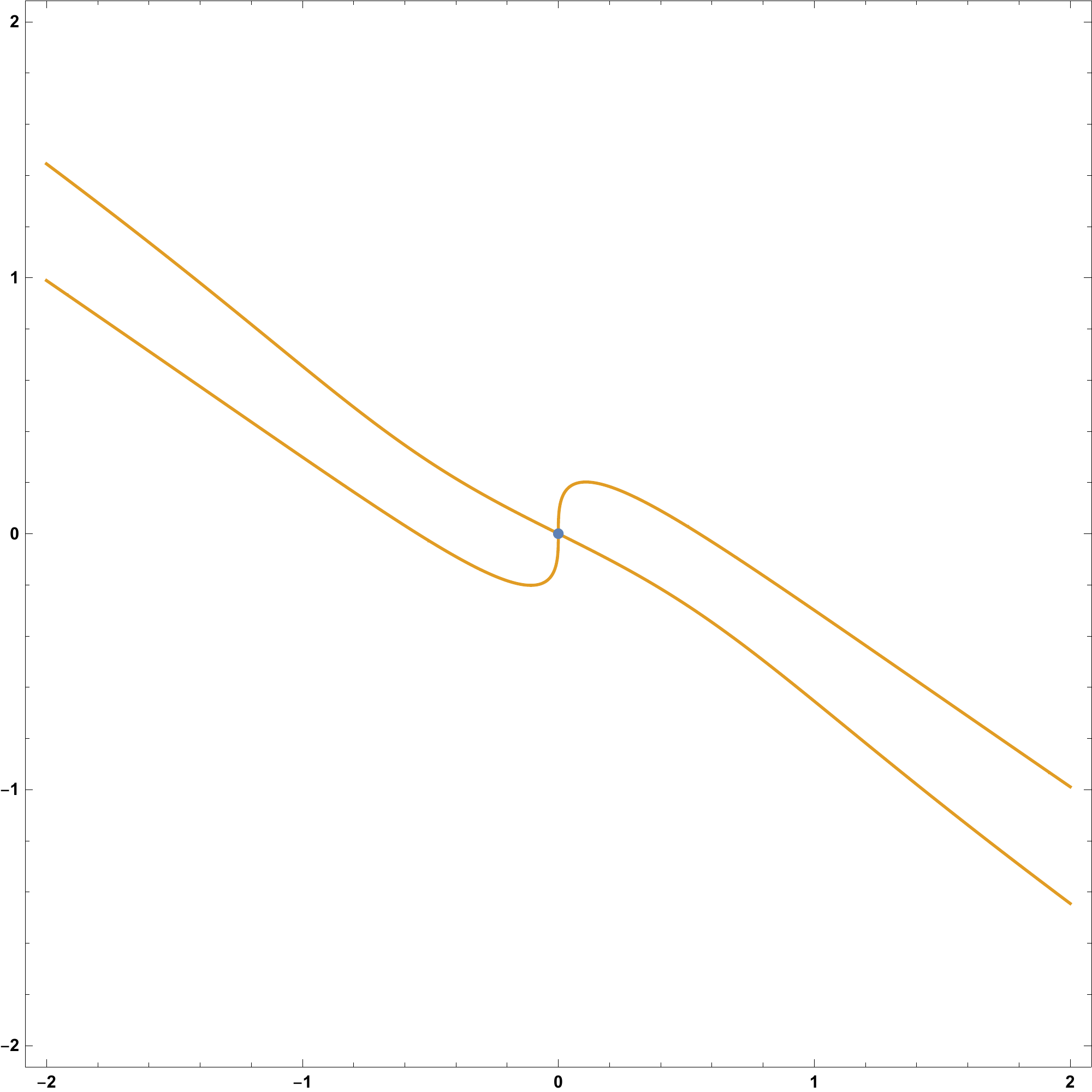}
        \caption{The curve of normals in the $(u,v)$ chart.}\label{fig:beanright}
     \end{subfigure} \hfill
      \begin{subfigure}[t]{0.32\textwidth}
         \centering
            \includegraphics[width=0.95\columnwidth]{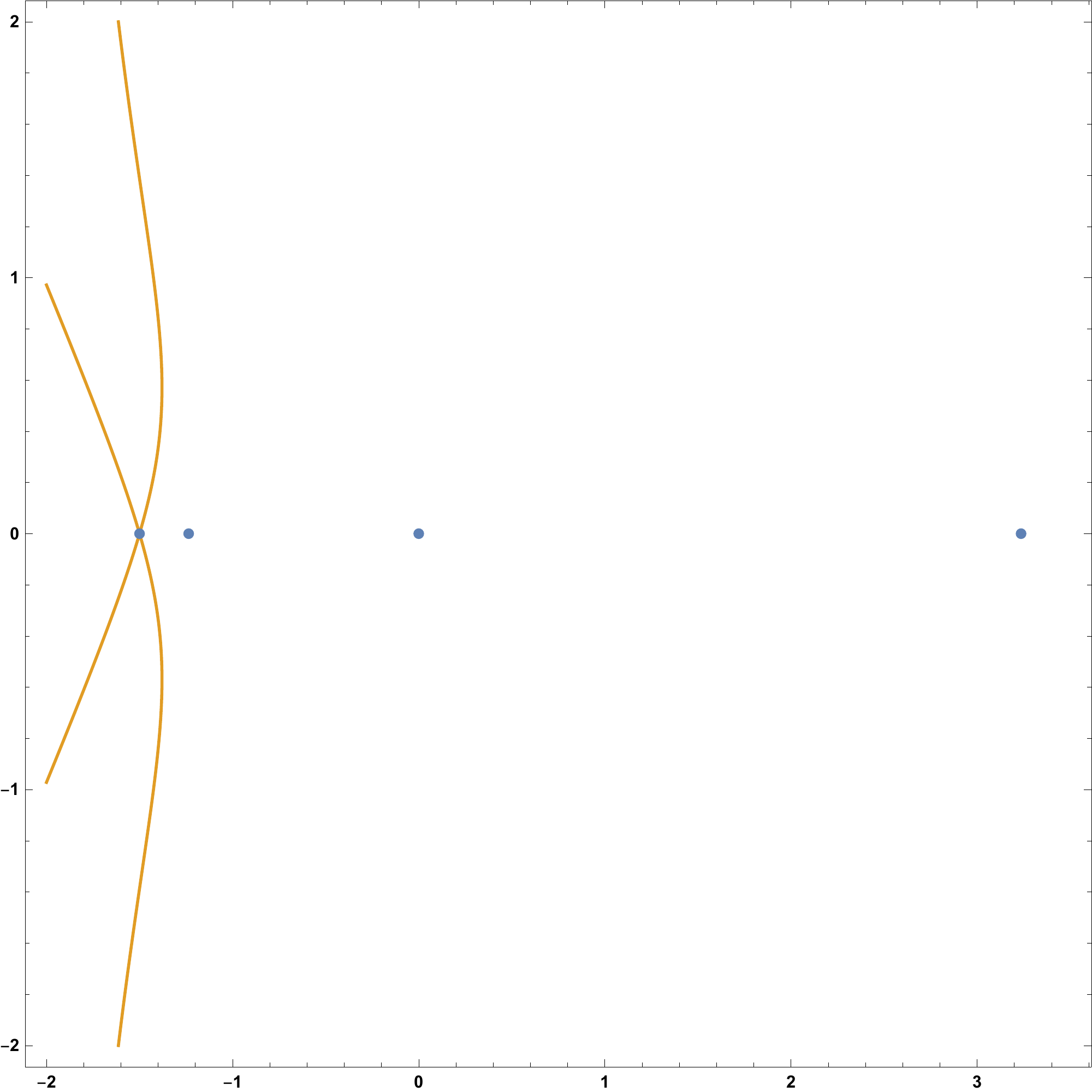}
        \caption{The curve of normals in the $(v,w)$ chart with marked singular points.}\label{fig:beanvw}
     \end{subfigure}\hfill

\caption{The bean curve, its evolute, and its curve of normals.}
\label{fig:bean}
\end{figure}

\smallskip
\noindent
{\bf X.} The trifolium is a rational curve given by the equation  
\[(x^2+y^2)^2-x^3+3xy^2=0.\]

 The trifolium is a \emph{2-circular} curve: the highest degree homogeneous part of its equation is divisible by $(x^2+y^2)^2$. The curve touches the line at infinity at the two circular points. It has one singular point, an ordinary triple point at the origin $(0:0:1)$. 
 Its numerical characters are
$d=4$, $d^\vee=6$, $\delta=3$, $\kappa =0$, $\iota=6$, $\tau=4$.  
Klein--Schuh's formula gives $4-6=-\iota'-2\tau'=-2\tau'$, since there are no real inflection points. Hence $\Ga$ has one pair of conjugate tangents.

The evolute has totally 24 nodes and 12 cusps of which no nodes are real, i.e., it has neither acnodes nor crunodes. Of totally 12 cusps $6$ are real and $6$ are non-real. All real cusps can be seen in Figure~\ref{fig:trifolium} (a). 

For the curve of normals we get from Equation~(\ref{degnorm})
$\deg N_\Ga=4+6-2(2-1+1)=6$, and for the evolute we get from Equation~(\ref{degev})
$\deg E_\Ga=2\deg N_\Ga +d^\vee-2d+\kappa-\iota(E)=12+6-8+0-0=10$, see column X of   Table~\ref{tb:inv}. The curve of normals has 7 double points in the affine $(u,v)$-plane, all of which are crunodes, as can be seen in Fig~\ref{fig:trifoliumright}. Furthermore, it has  three real nodes at the line at infinity $w=0$, two crunodes at   $(\frac{1}{3} + \frac{\sqrt{33}}{3}:1: 0)$ and $(\frac{1}{3} - \frac{\sqrt{33}}{3}:1: 0)$, and an acnode at $(0:1:0)$.
Its evolute is a hypocycloid, with 6 real cusps in the affine plane, and is again a 2-circular curve. 

The trifolium has 6 vertices and 9 diameters.
\begin{figure}[H]

     \centering
     \begin{subfigure}[t]{0.49\textwidth}
         \centering
      \includegraphics[width=0.62\columnwidth]{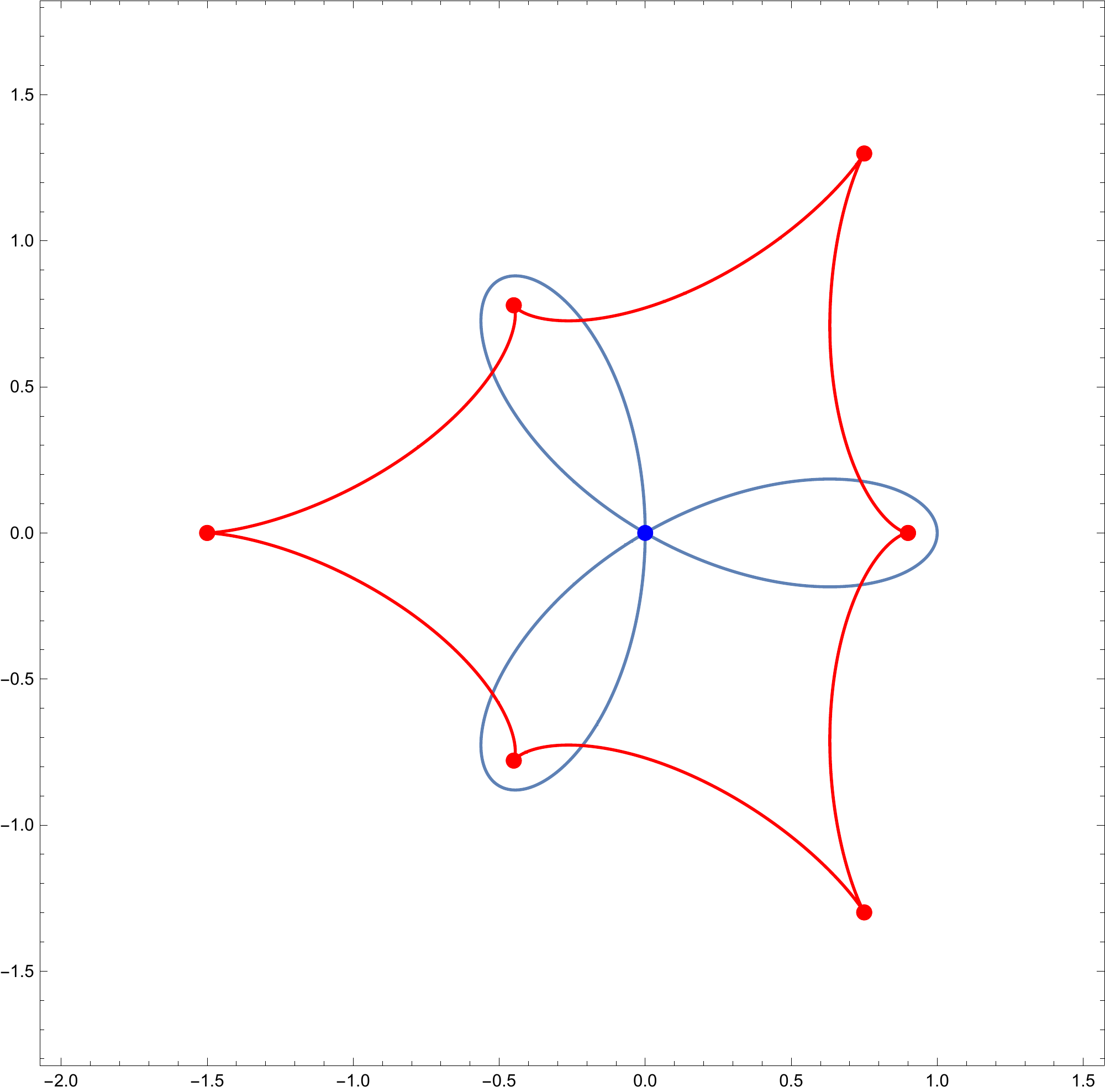} 
       \caption{The trifolium in blue with its evolute in red.}\label{fig:trifoliumleft}
     \end{subfigure}
     \hfill
     \begin{subfigure}[t]{0.49\textwidth}
         \centering
        \includegraphics[width=0.62\columnwidth]{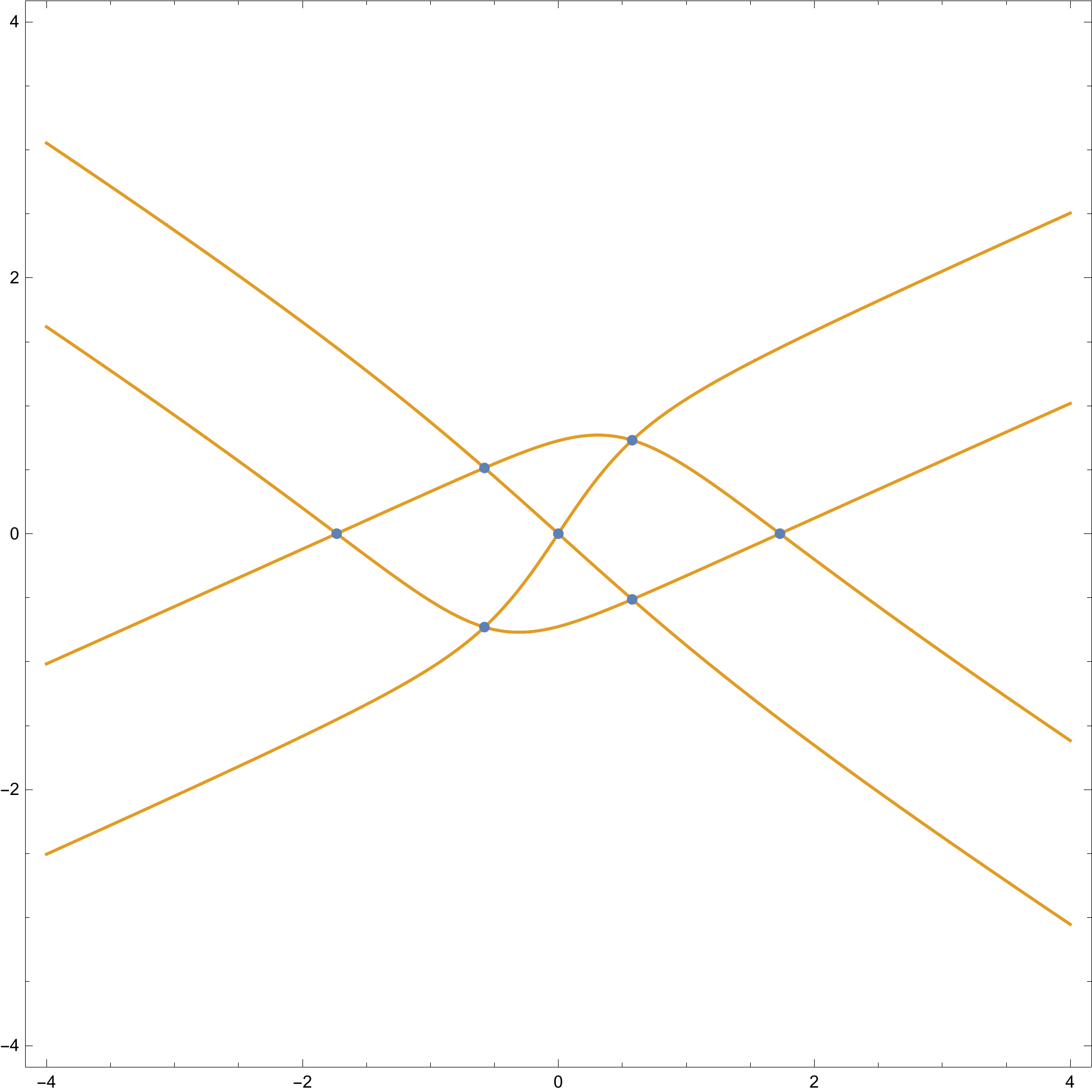}
        \caption{The curve of normals in the $(u,v)$ chart.}\label{fig:trifoliumright}
     \end{subfigure} \hfill

\caption{The trifolium, its evolute, and its curve of normals.}
\label{fig:trifolium}
\end{figure}

\smallskip
  \noindent
{\bf XI.} The quadrifolium is a rational curve given by the equation 
\[(x^2 + y^2)^3 - 4x^2y^2.\]

The quadrifolium is a 3-circular curve. It has cusps (with tangent the line at infinity) at the two circular points. Its quadruple point at the origin has $\delta$-invariant 8.
Its numerical characters are
$d=6$, $d^\vee=8$, $\delta=10$, $\kappa =2$, $\iota=8$, $\tau=21$. 

The degree of the curve of normals is  
$\deg N_\Ga=6+8-2(3-2+2)=8$, by Equation~(\ref{degnorm}). It has 17 nodes in the affine $(u,v)$-plane, all real and visible in Fig.~\ref{fig:quadfoliumright}, and four real nodes $(0:1:0)$, $(1: 0:0)$, $(\frac{3}{4}\sqrt{6}: 1: 0)$, and $(-\frac{3}{4}\sqrt{6}: 1: 0)$ on the line $w=0$. All of these nodes are crunodes, except for  $(0:1:0)$ which is an acnode. 

The evolute is a hypocycloid of degree $\deg E_\Ga=2\deg N_\Ga +d^\vee-2d+\kappa-\iota_E=16+8-12+2-0=14$, by (\ref{degev}).  It has  8 real cusps in the affine plane and no other real singularities.

The quadrifolium has 8 vertices and 20 diameters.
\begin{figure}[H]

     \centering
     \begin{subfigure}[t]{0.49\textwidth}
         \centering
      \includegraphics[width=0.62\columnwidth]{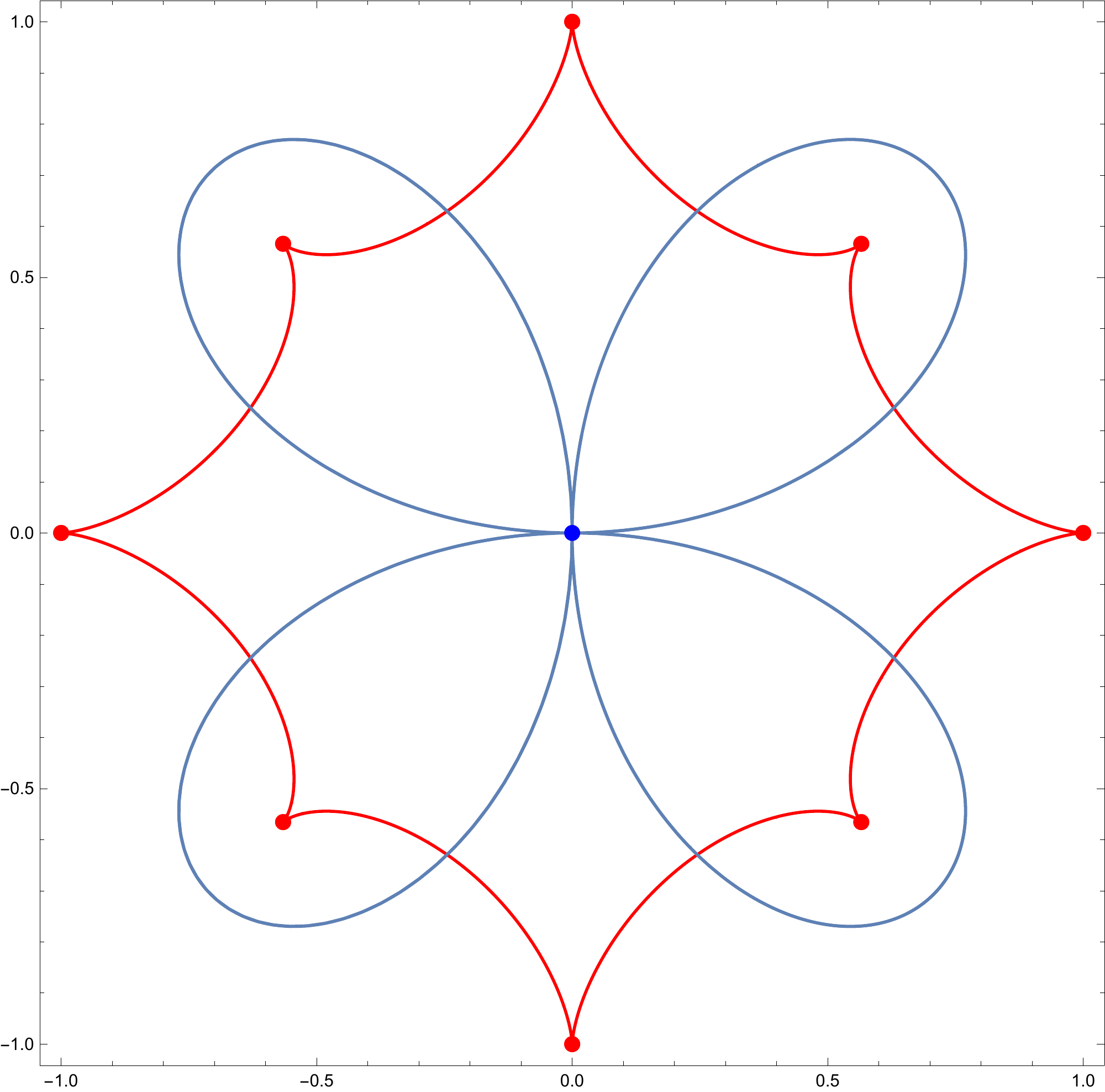} 
       \caption{The quadrifolium in blue with its evolute in red.}\label{fig:quadfoliumleft}
     \end{subfigure}
     \hfill
     \begin{subfigure}[t]{0.49\textwidth}
         \centering
        \includegraphics[width=0.62\columnwidth]{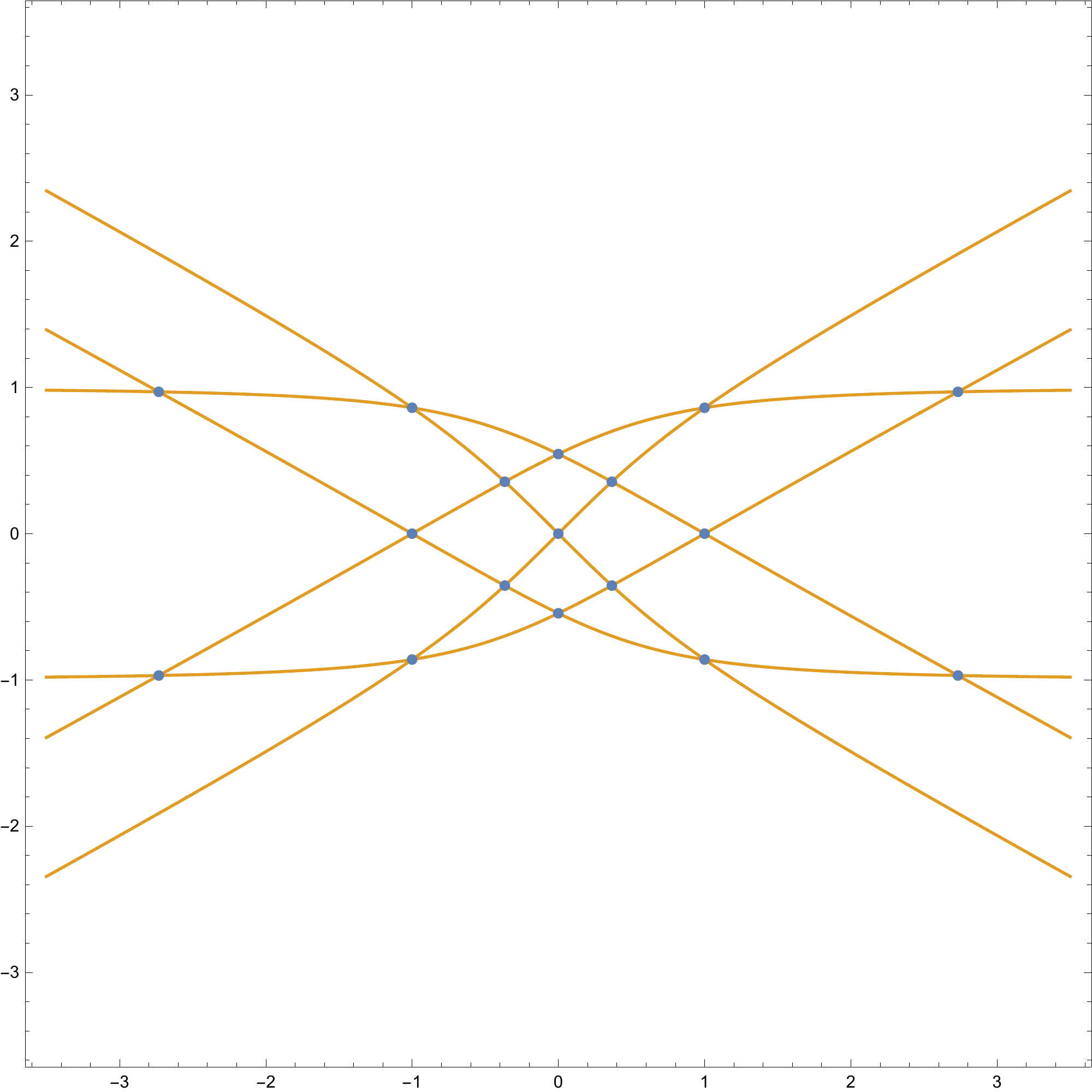}
        \caption{The curve of normals in the $(u,v)$ chart.}\label{fig:quadfoliumright}
     \end{subfigure} \hfil
\caption{The quadrifolium, its evolute, and its curve of normals.}
\label{fig:quadfolium}
\end{figure}
\medskip

\smallskip
\noindent
{\bf XII.} The D\"urer folium is a rational curve given by the equation 
\[(x^2 + y^2) (2(x^2 + y^2) - 1)^2 - x^2=0\]
is a 3-circular curve. It has $E_6$-singularities at the circular points, and two crunodes and one $A_3$-singularity in the affine plane.
Its numerical characters are
$d=6$, $d^\vee=6$, $\delta=10$, $\kappa =4$, $\iota=4$, $\tau=10$, see row XII of  Table~\ref{tb:inv}.

Its evolute is a 3-circular epicycloid of degree $\deg E_\Ga=10$, with 4 real cusps in the affine plane, and no other real singularities.

Its curve of normals has degree
$\deg N_\Ga=6+6-2(3-3+3)=6$    
 and 10 double points, all real. Of these, 7 are in the affine $(u,v)$-plane  and are crunodes, as  can be seen in Fig.~\ref{fig:duererright}. Additionally the curve of normals has three crunodes at the line at infinity $w=0$. These are $(1:0:0)$, $(\frac{3}{2}\sqrt{6}: 1: 0)$ and $(-\frac{3}{2}\sqrt{6}: 1: 0)$. 
 
 The D\"urer folium has 4 vertices and 10 diameters.
\begin{figure}[H]

     \centering
     \begin{subfigure}[t]{0.49\textwidth}
         \centering
      \includegraphics[width=0.62\columnwidth]{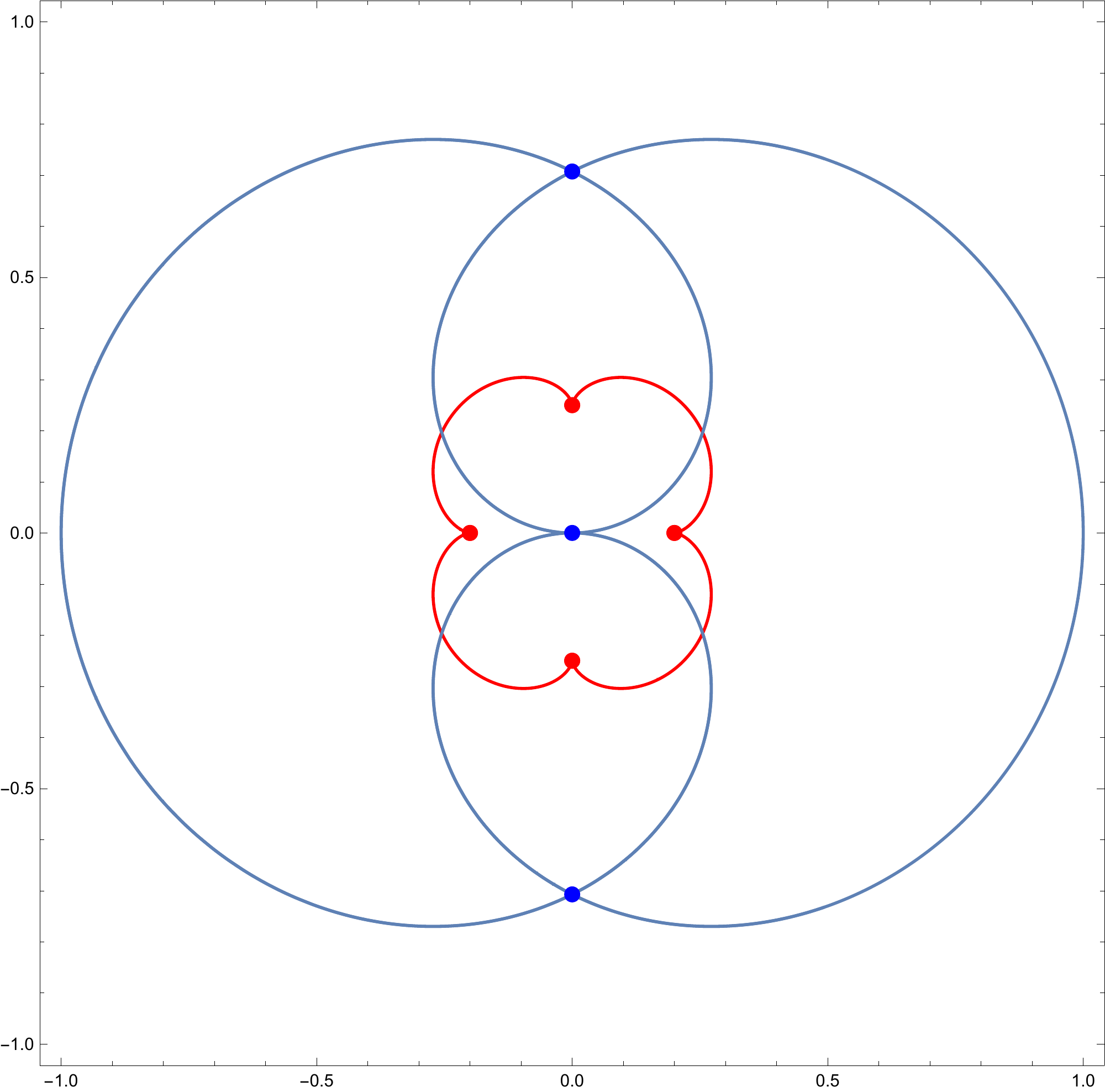} 
       \caption{The D\"urer folium in blue with its evolute in red.}\label{fig:duererleft}
     \end{subfigure}
     \hfill
     \begin{subfigure}[t]{0.49\textwidth}
         \centering
        \includegraphics[width=0.62\columnwidth]{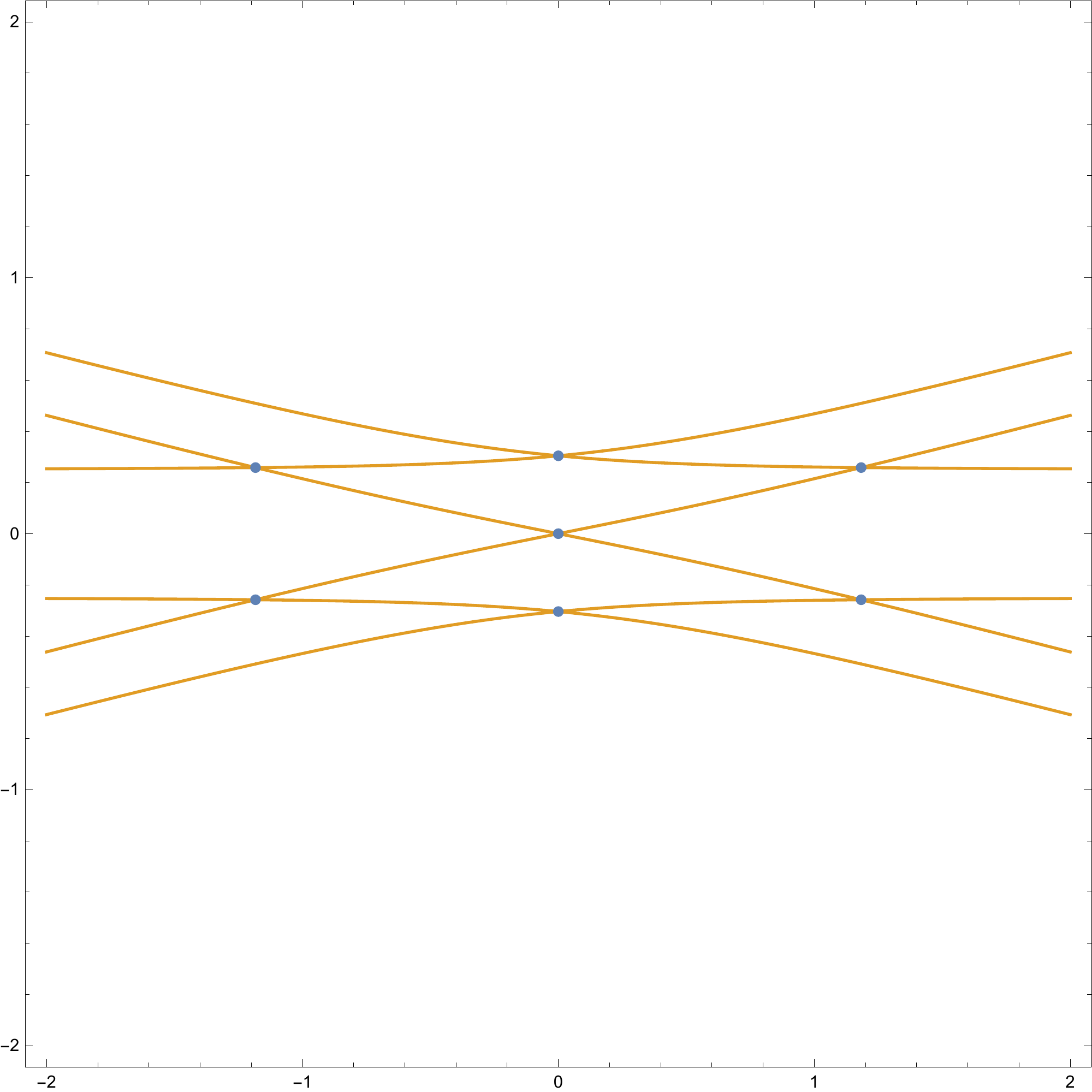}
        \caption{The curve of normals in the $(u,v)$ chart.}\label{fig:duererright}
     \end{subfigure} \hfil
\caption{The D\"urer folium, its evolute, and its curve of normals.}
\label{fig:duerer}
\end{figure}

\smallskip
\noindent
{\bf XIII.} The nephroid is a rational curve given by the equation 
\[4(x^2 + y^2 - 1)^3 - 27y^2=0.\]
It is a 3-circular curve, with an $E_6$ singularity at each circular point. 
It has two real ordinary cusps, two complex nodes, and no inflection points. Its numerical characters are
$d=6$, $d^\vee=4$, $\delta=10$, $\kappa =6$, $\iota=0$, $\tau=3$. 
Klein--Schuh's formula checks:
$6-4=2(2-1)-0=2$.

The nephroid is an epicycloid, hence its evolute $E_\Ga$ is again a nephroid (turned by 90 degrees), hence also 3-circular. It has equation 
\[(4(x^2+y^2)-1)^3-27x^2=0.\]

$\deg E_\Ga=6$, $\kappa(E)=6$, $\iota(E)=0$, and $\delta_E =10-2\cdot 1-2\cdot 3=2$.

Since the curve of normals is the dual of the evolute, we have
$\deg N_\Ga=d^\vee =4$, $\kappa(N)=\iota(E)=0$, and $\delta_N=\delta(N)=3$, see column XIII of   Table~\ref{tb:inv}. 
The curve of normals has three crunodes (we see two of them in Fig.~\ref{fig:neph}, the third correspond to the line $x=0$).  

The nephroid  has 2 vertices and 3 diameters.
     
     \begin{figure}[H]
  \centering
     \begin{subfigure}[t]{0.49\textwidth}
         \centering
      \includegraphics[width=0.62\columnwidth]{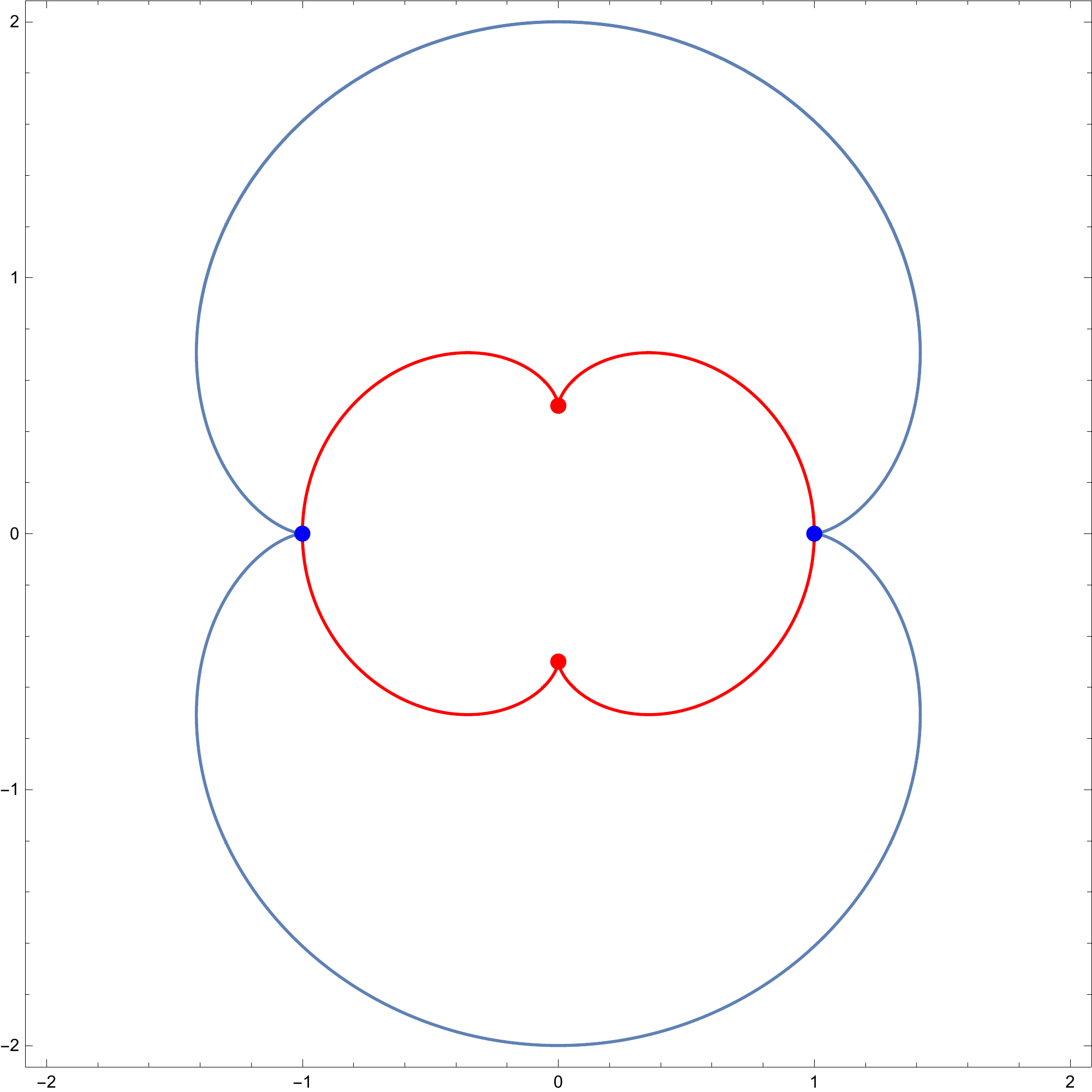} 
       \caption{The nephroid in blue with its evolute in red.}\label{fig:nephleft}
     \end{subfigure}
     \hfill
     \begin{subfigure}[t]{0.49\textwidth}
         \centering
        \includegraphics[width=0.62\columnwidth]{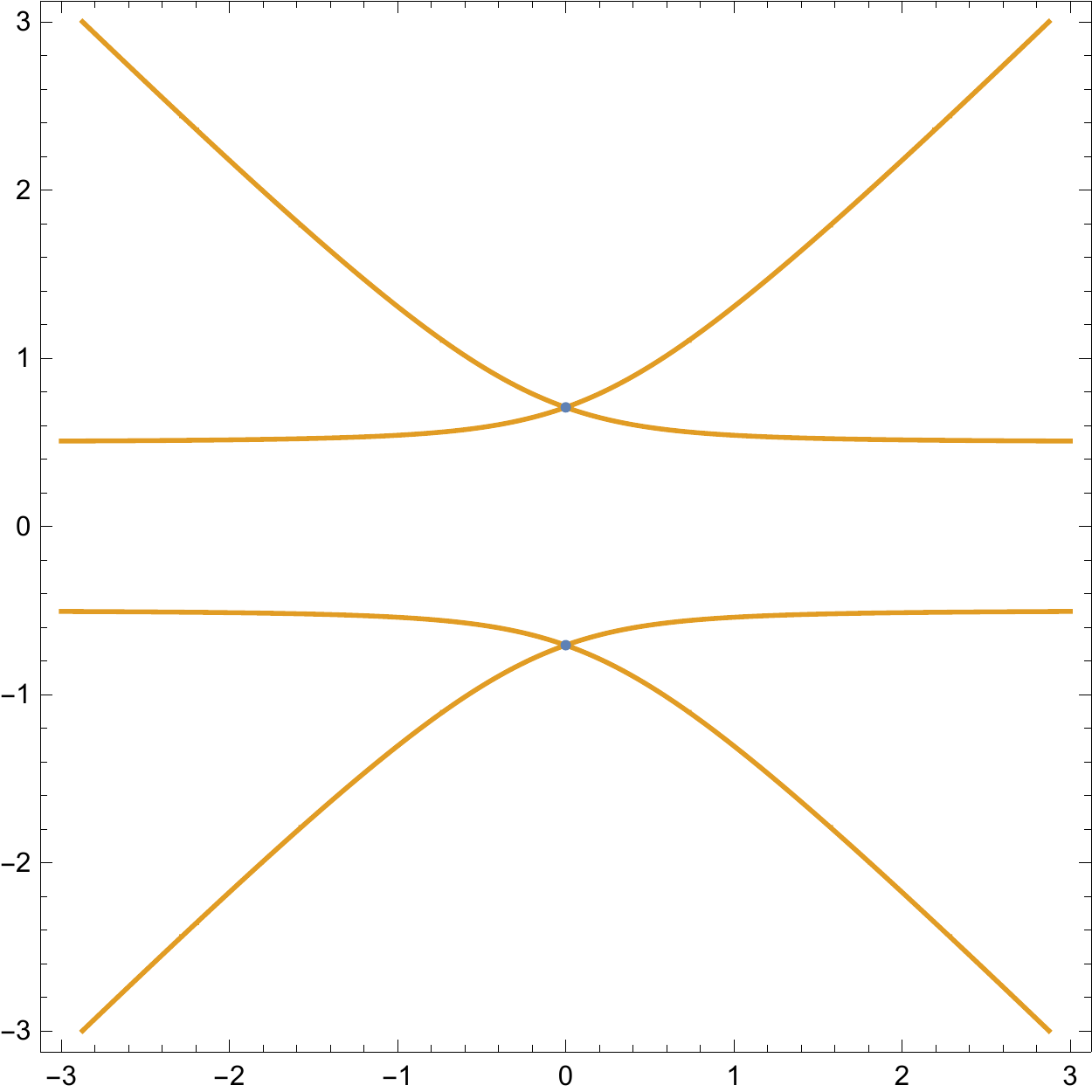}
        \caption{The curve of normals in the $(u,v)$ chart.}\label{fig:nephright}
     \end{subfigure} 

\caption{The nephroid, its evolute, and its curve of normals.}
\label{fig:neph}
\end{figure}

\smallskip
 \noindent
{\bf XIV.} Cayley's sextic is a rational curve given by the equation 
\[4(x^2 + y^2 - x)^3 - 27(x^2 + y^2)^2=0.\]
It is 3-circular, with $E_6$ singularities at the two circular points. 
It has two other singular points: one crunode and one $E_6$ singularity at the origin. Its numerical characters are
$d=6$, $d^\vee=4$, $\delta=10$, $\kappa =6$, $\iota=0$, $\tau=3$, see column XIV of Table~\ref{tb:inv}.

Its evolute $E_\Ga$ is a nephroid, hence $\deg E_\Ga =6$ and $\kappa(E)=6$. Its dual, the curve of normals of $\Ga$, has degree 4 and 3 crunodes (we see one in the affine $(u,v)$-plane).

Cayley's sextic has 2 vertices and 3 diameters.
     
 \begin{figure}[H]
  \centering
     \begin{subfigure}[t]{0.49\textwidth}
         \centering
      \includegraphics[width=0.62\columnwidth]{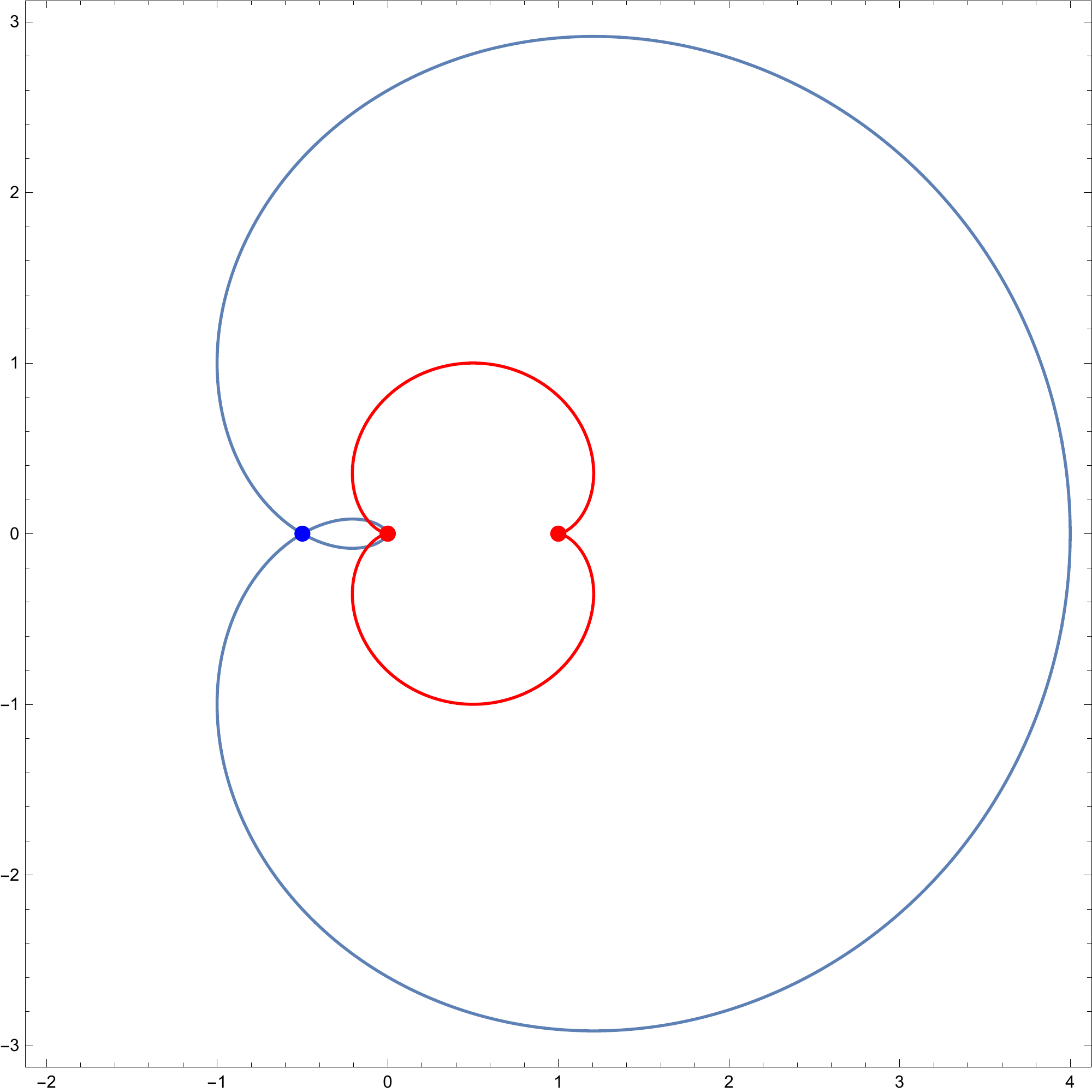} 
       \caption{Cayley's sextic in blue with its evolute in red.}\label{fig:caleyleft}
     \end{subfigure}
     \hfill
     \begin{subfigure}[t]{0.49\textwidth}
         \centering
        \includegraphics[width=0.62\columnwidth]{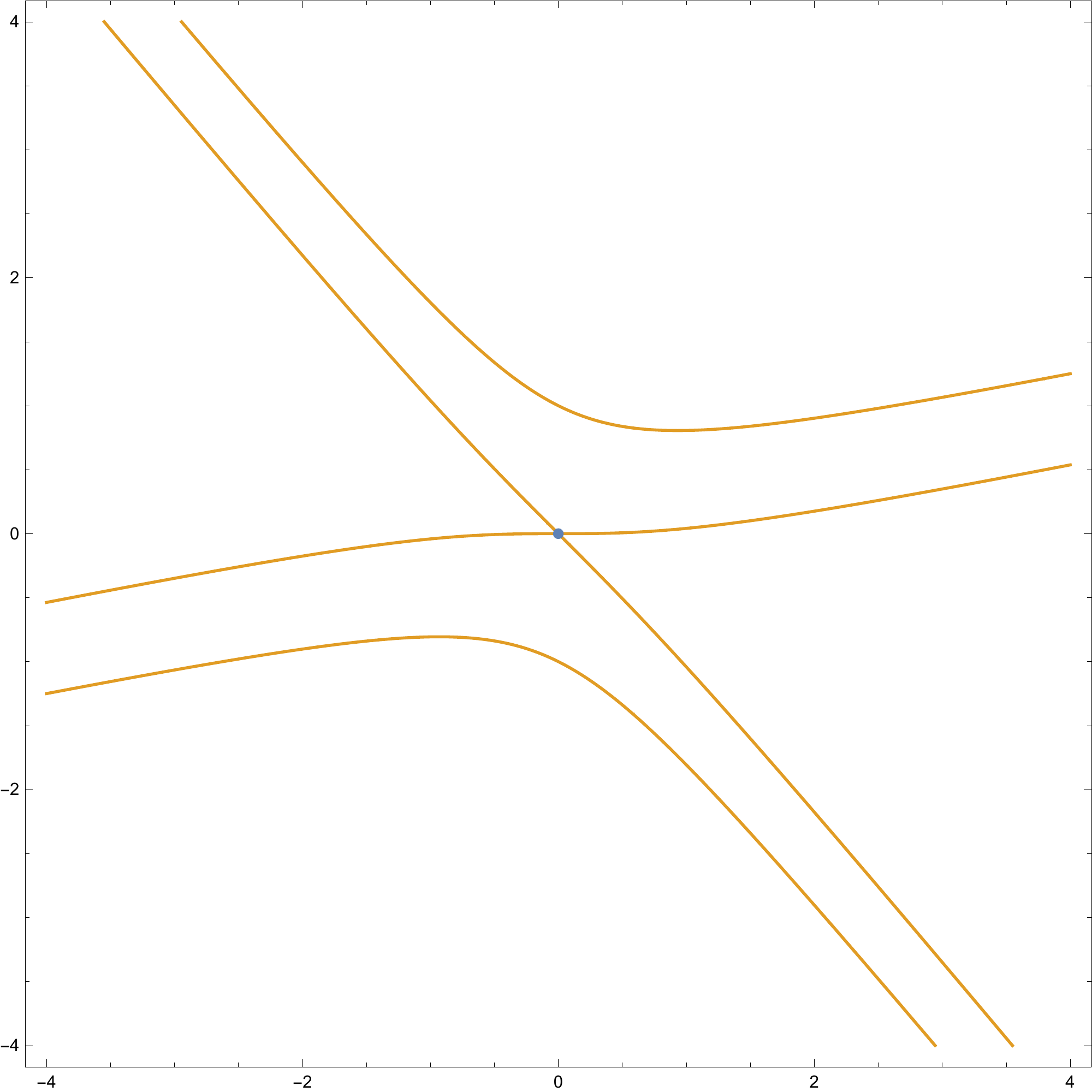}
        \caption{The curve of normals in the $(u,v)$ chart.}\label{fig:caleyright}
     \end{subfigure} \hfil
\caption{Cayley's sextic, its evolute, and its curve of normals.}
\label{fig:caley}
\end{figure}
\medskip

\smallskip
 \noindent
{\bf XV.} The ranunculoid is an epicycloid  given by  the equation 

\noindent{\small 
\begin{dmath*}
0=-52521875 + 93312 x^5 + x^{12} - 1286250 y^2 - 933120 x^3 y^2 - 32025 y^4 + 466560 x y^4 - 812 y^6 - 21 y^8 - 42 y^{10} + y^{12} + 
 6 x^{10} (-7 + y^2) + 3 x^8 (-7 - 70 y^2 + 5 y^4) + 4 x^6 (-203 - 21 y^2 - 105 y^4 + 5 y^6) + 3 x^4 (-10675 - 812 y^2 - 42 y^4 - 140 y^6 + 5 y^8) + 6 x^2 (-214375 - 10675 y^2 - 406 y^4 - 14 y^6 - 35 y^8 + y^{10})\\
 =(x^2+y^2)^6+\text{terms of lower degree}.
 \end{dmath*}}
 
 This is a rational curve which is 6-circular, with cuspidal singularities with multiplicity 6 and $\delta$-invariant 15 at the two circular points, and additional five real cusps.   Its numerical characters are
$d=12$, $d^\vee=7$, $\delta=55$, $\kappa =15$, $\iota=0$, $\tau=15$. 

Its evolute $E_\Ga$ is again a ranunculoid, hence 6-circular, and has $\deg E_\Gamma=12$.
 For the curve of normals, we get
$\deg N_\Ga=7$, $\kappa(N)=0$, $\delta_N=\delta(N)=15$, see column XV of   Table~\ref{tb:inv}. All of the nodes are real and we can see 12 of these in Fig. \ref{fig:ranright}. Additionally, there are 
three nodes at the line at infinity at the points $(\alpha:1:0)$, where $\alpha$ runs through the roots of $\phi(t)=125t^3 - 100t^2 - 20t + 8.$ Changing again to a different affine chart, all of the nodes are visible in Fig. \ref{fig:ranuw} and we find that $N_\Ga$ has 15 crunodes. 

The ranunculoid has 5 vertices and 15 diameters.

\begin{figure}[H]

  \centering
     \begin{subfigure}[t]{0.32\textwidth}
         \centering
      \includegraphics[width=0.95\columnwidth]{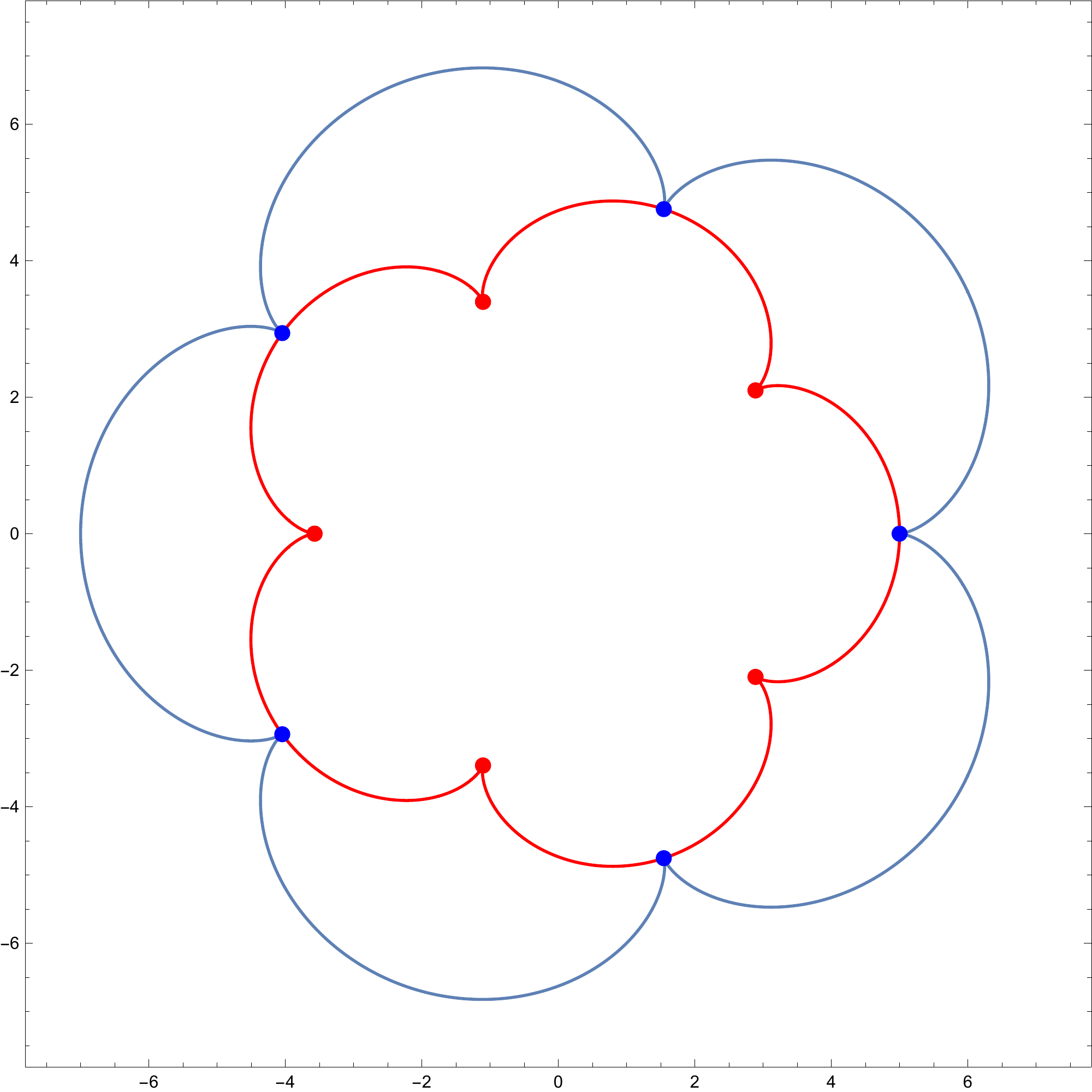} 
       \caption{The ranunculoid in blue with its evolute in red.}\label{fig:ranleft}
     \end{subfigure}
     \hfill
     \begin{subfigure}[t]{0.32\textwidth}
         \centering
        \includegraphics[width=0.95\columnwidth]{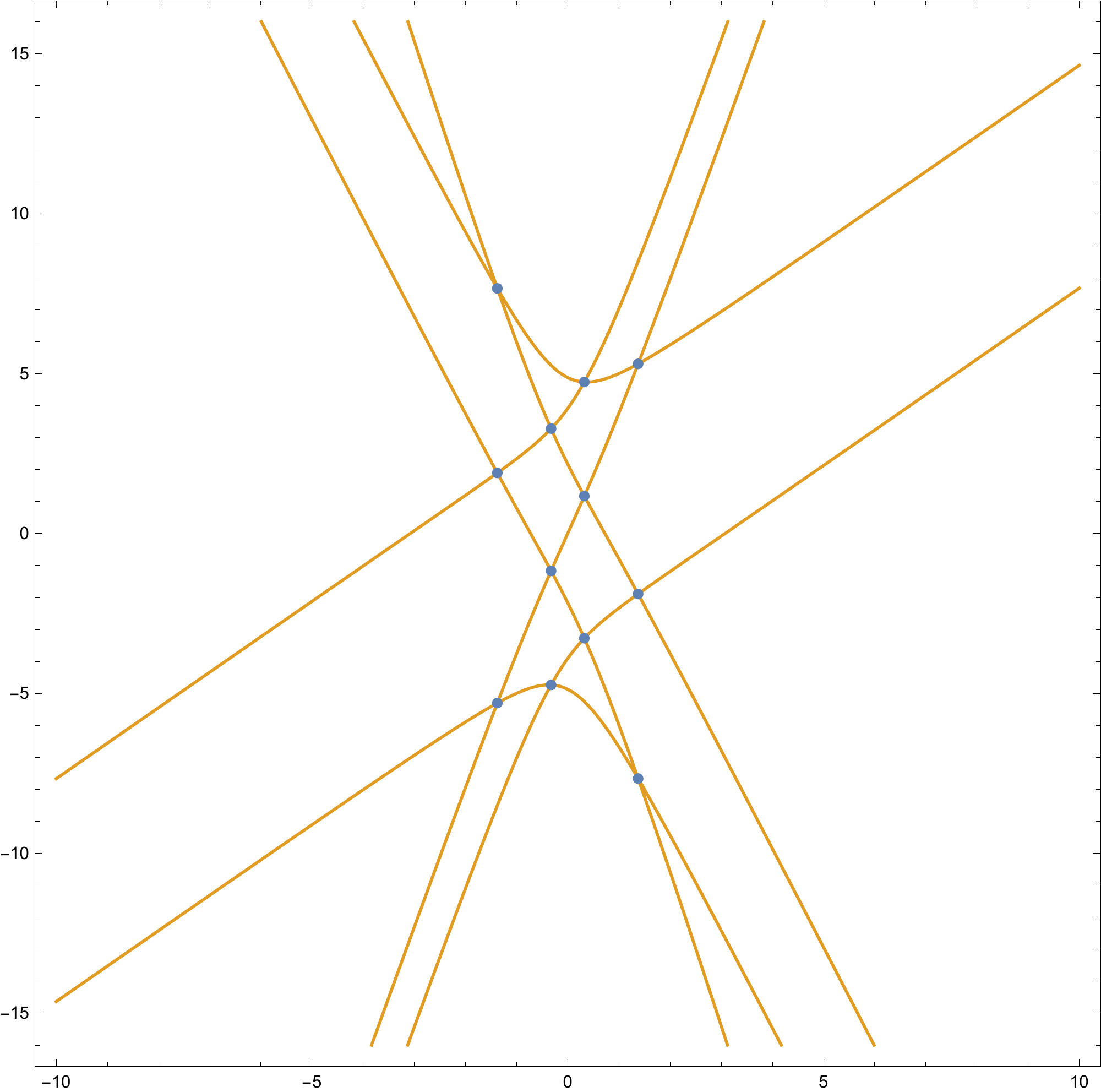}
        \caption{The curve of normals in the $(u,v)$ chart.}\label{fig:ranright}
     \end{subfigure} \hfil
       \begin{subfigure}[t]{0.32\textwidth}
         \centering
            \includegraphics[width=0.95\columnwidth]{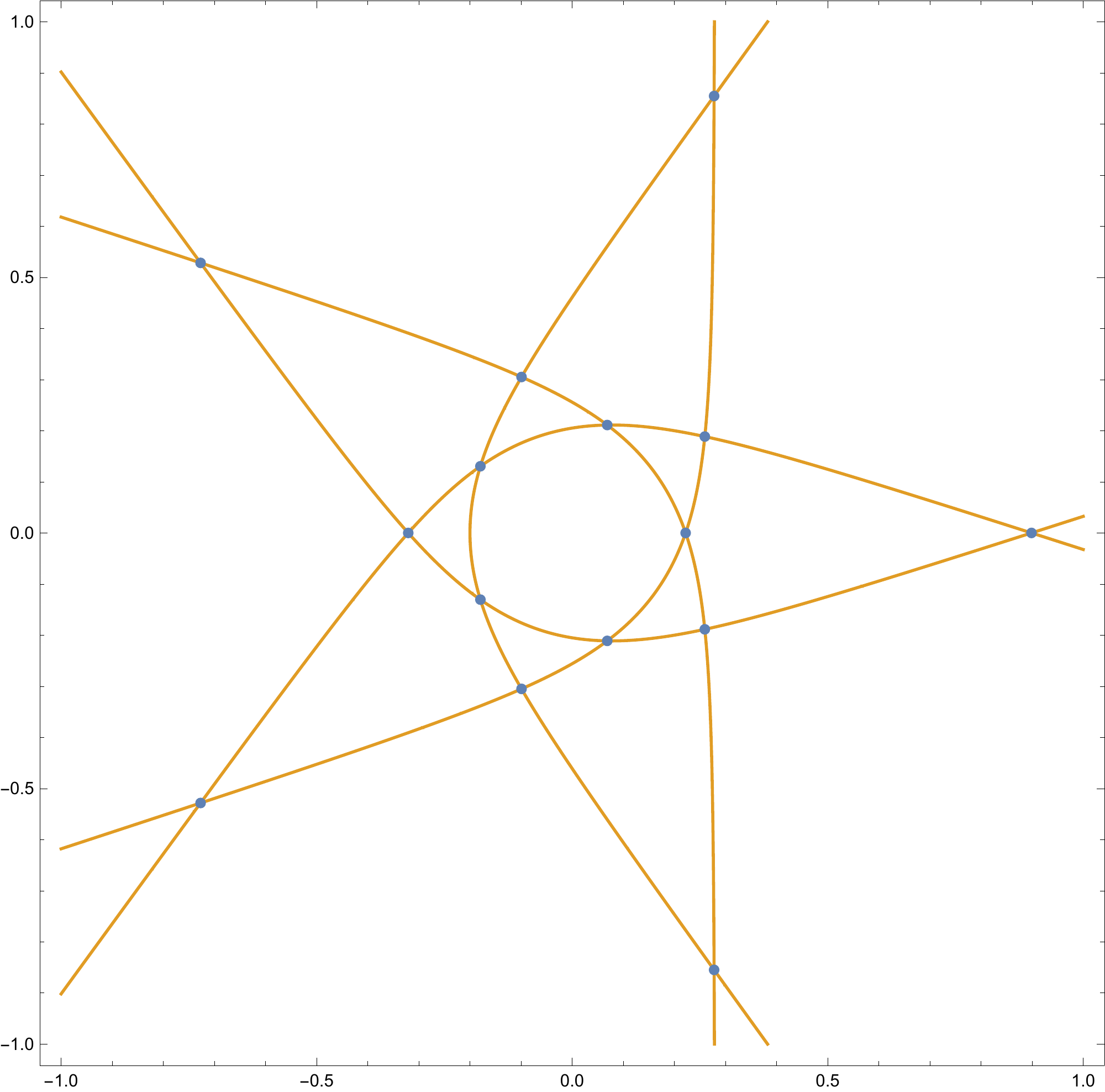}
        \caption{The curve of normals in the $(u,w)$ chart.}\label{fig:ranuw}
     \end{subfigure}

\caption{The ranunculoid, its evolute, and its curve of normals.}
\label{fig:ran}
\end{figure}

\subsection{Tables of invariants of the complexifications of the above examples}\label{subs:tables}  \hfill\\

Let $\Gamma\subset \mathbb R^2$ be a real, plane curve. Let $C_0\subset \mathbb P^2_{\mathbb C}$ denote  projective closure of its complexification and $\nu:C\to C_0$ the normalization map. We are interested in studying the evolute $E_C$  and the curve of normals $N_C$ of $\Gamma$ and of $C$. In Table~\ref{tb:inv} we give the numerical characters (cf. Appendix A) of the curve $C$, its evolute, and its curve of normals for each of the fifteen examples in \S~8.1. We set $d:=\deg C$, $d^\vee :=\deg C^\vee$, $\delta:=\delta$-invariant of $C_0$, $\kappa:=$ the degree of the ramification locus of $\nu$, $\tau:= \delta$-invariant of $C^\vee$, $\iota:=$ the ramification divisor of $C\to C^\vee$, $d(E):=\deg E_C$,  $d(N):=\deg N_C$, $\delta(E):=$ the $\delta$-invariant of $E_C$, $\kappa(E):=$ the degree of the ramification divisor of $C\to E_C$, $\tau(E):=$ the $\delta$-invariant of $E_C$, $\iota(E):=$ the degree of the ramification divisor of $C\to E_C^\vee =N_C$. Since we have $N_C=E_C^\vee$, by duality, we have $\kappa(E)=\iota(N)$, $\iota(E)=\kappa(N)$, $\tau(E)=\delta(N)$, and $\tau(N)=\delta(E)$.


\begin{center}
\begin{table}[h]
\tablestyle[sansbold]
\begin{tabular}{||c c c c c c c c c c c c c c c c||} 
\hline
& I & II & III & IV & V & VI & VII & VIII & IX & X & XI & XII & XIII & XIV & XV\\ [0.5ex] 
 \toprule\hline
 $d$ & 2 & 2 &3 & 3 & 3 & 3 & 4 & 4 & 4 & 4 & 6 & 6 & 6 & 6 & 12 \\ 
 \hline
$d^\vee$ & 2 & 2 & 3 & 4 & 6 & 6 & 6 & 6 & 6 & 6 & 8 & 6 & 4 & 4 & 7 \\ 
\hline
$\delta$ & 0 & 0 & 1 & 1 & 0 & 0 & 3 & 3 & 3 & 3 & 10 & 10 & 10 & 10 & 55  \\ 
\hline
$\kappa$ & 0 & 0 & 1 & 0 & 0 & 0 & 0 & 0 & 0 & 0 & 2 & 4 & 6 & 6 & 15 \\ 
\hline
$\tau$ & 0 & 0 & 1 & 3 & 9 & 9 & 10 & 10 & 10 & 10 & 10 & 10 & 3 & 3 & 15  \\ 
\hline
$\iota$ & 0 & 0 & 1 & 3 & 9 & 9 & 6 & 6 & 6 & 6 & 8 & 4 & 0 & 0 & 0 \\ 
\hline
$d(E)$ & 6 & 6 & 4 & 12 & 12 & 18 & 18 & 18 & 18 & 10 & 14 & 10 & 6 & 6 & 12 \\ 
\hline
$d(N)$ & 4 & 4 & 4 & 7 & 7 & 9 & 10 & 10 & 10 & 6 & 8 & 6 & 4 & 4 & 7 \\ 
\hline
$\delta(E)$ & 10 & 10 & 3 & 55 & 54 & 135 & 136 & 136 & 136 & 36 & 78 & 36 & 10 & 10 & 55 \\ 
\hline
$\kappa(E)$ & 6 & 6 & 2 & 15 & 17 & 27 & 24 & 24 & 24 & 12 & 18 & 12 & 6 & 6 & 15 \\ 
\hline
$\tau(E)$ & 3 & 3 & 3 & 15 & 14 & 27 & 36 & 36 & 36 & 10 & 21 & 10 & 3 & 3 & 15 \\ 
\hline
$\iota(E)$ & 0 & 0 & 2 & 0 & 2 & 0 & 0 & 0 & 0 & 0 & 0 & 0 & 0 & 0 & 0 \\ 
\hline

\end{tabular}

\caption{Table of invariants case-by-case.} \label{tb:inv}
\end{table}
\end{center}

In Table~\ref{ta:singularities} we list the actual complex singularities of the evolute and the curve of normals, namely the number of singularities of types $A_1$, $A_2$, $D_4$, $E_6$, and other. The last row indicates whether the curve of normals has a point of multiplicity $d$ corresponding to the line at infinity with respect to $\Ga$.

\begin{center}
\begin{table}[h]
\tablestyle[sansbold]
\begin{tabular}{||c c c c c c c c c c c c c c c c||} 
 \hline
& I & II & III & IV & V & VI & VII & VIII & IX & X & XI & XII & XIII & XIV & XV\\ [0.5ex] 
 \toprule\hline
  $A_1(E)$ & 4 & 4 & 0 & 39 & 37 & 105 & 112 & 84 & 112 & 24 & 60 & 18 & 2 & 2 & 20 \\
 \hline
  $A_2(E)$ & 6 & 6 & 0 & 13 & 17 & 21 & 24 & 16 & 24 & 12 & 18 & 12 & 2 & 2 & 5 \\
 \hline
   $E_6(E)$ & 0 & 0 & 1 & 1 & 0 & 3 & 0 & 0 & 0 & 0 & 0 & 2 & 2 & 2 & 0\\
   \hline
    other & 0 & 0 & 0 & 0 & 0 & 0 & 0 & $2$ & 0 & 0 & 0 & 0 & 0 & 0 & $2$\\
 \hline
   $A_1(N)$ & 2 & 2 & 1 & 12 & 8 & 21 & 30 & 30 & 30 & 10 & 21 & 10 & 3 & 3 & 15 \\
 \hline
  $A_2(N)$ & 0 & 0 & 2 & 0 & 0 & 0 & 0& 0 & 0 & 0 & 0 & 0 & 0 & 0 & 0\\
 \hline
 $D_4(N)$ & 0 & 0 & 0 & 0 & 0 & 1 & 0 & 0 & 0 & 0 & 0 & 0 & 0 & 0 & 0 \\
 \hline
  $E_6(N)$ & 0 & 0 & 0 & 0 & 1 & 0 & 0 & 0 & 0 & 0 & 0 & 0 & 0 & 0 & 0 \\
 \hline
   $d$-uple$(N)$ & 1 & 1 & 0 & 1 & 1 & 1 & 1 & 1 & 1 & 0 & 0 & 0 & 0 & 0 & 0 \\
 \hline
 
 \end{tabular}
 \medskip
 
 \caption{Table of actual (complex) singularities of the evolute and the curve of normals. The two ``other'' singularities for VIII are singularities with multiplicity 6, $\delta$-invariant 18, and 2 branches: the merging of two $E_6$-singularities, and the two ``other'' singularities for XV are singularities with multiplicity 6, $\delta$-invariant 15, and 1 branch.}\label{ta:singularities}
 \end{table}
\end{center}
\medskip 

In Table~\ref{ta:kleinschuh} we list the terms appearing in the Klein--Schuh formula (Thm.~\ref{th:genKlein}) for the evolute and its dual curve, the curve of normals, and the number of vertices and diameters of $\Ga$.
\begin{center}
\begin{table}[h]
\tablestyle[sansbold]
\begin{tabular}{||c c c c c c c c c c c c c c c c||} 
 \hline
& I & II & III & IV & V & VI & VII & VIII & IX & X & XI & XII & XIII & XIV & XV\\ [0.5ex] 
  \toprule\hline
$d(E)$ & 6 & 6 & 4 & 12 & 12 & 18 & 18 & 18 & 18 & 10 & 14 & 10 & 6 & 6 & 12\\
\hline
$d(N)$ & 4 & 4 & 4 & 7 & 7 & 9 & 10  & 10 & 10 & 6 & 8 & 6 & 4 & 4  & 7\\
\hline
{\small$\sum (m_p(E)-r_p(E))$} & 4 & 4 & 2 & 9 & 9 & 15 & 14 & 16 & 16 & 6 & 8 & 4 & 2 & 2  & 5\\
\hline
{\small $\sum (m_q(N)-r_q(N))$} & 2 & 2 & 2 & 4 & 4 & 6 & 6 & 8 & 8 & 2 & 2 & 0 & 0 & 0 & 0\\
\hline
$V(\Ga)$ & 4 & 2 & 0 & 5 & 9 & 9 & 6 & 4 & 5 & 6 & 8 & 4 & 2 & 2 & 5 \\
\hline
$D(\Ga)$ & 2 & 1 & 0 & 10 & 4 & 13 & 14 & 2 & 2 & 9 & 20 & 10 & 3 & 3 & 15 \\
\hline
\end{tabular}
\medskip

\caption{The first four rows show the terms in the Klein--Schuh formula, applied to $E_\Ga$ and $N_\Ga$. The last two rows show the number of vertices $V(\Ga)$ and the number of diameters $D(\Ga)$.}\label{ta:kleinschuh}
\end{table}
\end{center}

\section {Some further problems}
Below we present very small sample of natural questions related to the topic of the paper.  

\smallskip
\noindent 
{\bf 1.} The basic technical tool which we use to approach Problems 1--4 formulated in \S~\ref{sec:init} is to consider small deformations of real line arrangements.  Line arrangements themselves do not have evolutes in the conventional sense, but probably such evolutes can be defined as certain  limiting piecewise linear geometric objects. In general, our approach has a strong resemblance with  methods of tropical geometry. One wonders if it is possible to develop a rigorous tropical geometry in the context of evolutes and their duals. 

\smallskip
\noindent 
{\bf 2.} Analog of the problem about the maximal number of ovals. 
\begin{problem} For a given degree $d$, what is the maximal number of connected components in the complement of $\bR^2\setminus N_\Ga$, $\bR^2\setminus E_\Ga$ and $\bR^2\setminus (E_\Ga \cup \Ga)$? (One can ask the same question for the complements in $\bR P^2$.)
\end{problem}

This question is an  analog  of the problem about the maximal number of ovals for a plane curve of a given degree answered by Harnack's theorem, but since our curves are always singular it is better to ask about the maximal number of components of the complement. 

\smallskip
\noindent 
{\bf 3.} 
 Notice that if we fix a real polynomial defining $\Ga$ then normals to $\Ga$ split into two complementary classes -- gradient-like and antigradient-like. Namely, each normal is  proportional to the gradient with either a positive  or a negative constant. If we change the sign of the defining polynomial then we interchange gradient-like and anitgradient-like normals.

\begin{problem} Given $d$, what is the maximal number of normals of one type which can pass through a given point?
\end{problem} 

By Proposition~\ref{lm:Shustin}  the total number of real normals through a point can reach $d^2$, but it is not clear at the moment how many of them can belong to one class.

\smallskip
\noindent 
{\bf 4.} Are the leading coefficients in our lower bounds  for the  Problems 1--4  sharp?

\appendix
\section{On dual curves, envelopes, and evolutes over $\bC$.} \label{sec:Ragni} 
In this appendix we consider curves in the complex projective plane. We define dual and reciprocal curves, envelopes, evolutes, and curves of normals, and provide modern proofs for some of the formulas found in \cites{Sa, Hi, Co}.

 \subsection{Dual curves}\label{dual}
 Let $V$ be a complex vector space of dimension 3.
 Let $C_0\subset \mathbb P(V)\cong \mathbb P^2_\bC$ be a plane curve, with normalization map $\nu: C\to C_0$.
 Let $V_{C} \to \mathcal P$ denote the Nash quotient extending $V_{C_{0,\rm sm}}\to \mathcal P_{C_0}^1(1)|_{C_{0,\rm sm}}$ and let $\mathcal K$ denote the kernel.   The \emph{dual curve} $C^\vee \subset \mathbb P(V)^\vee:=\mathbb P(V^\vee)$ is the image of the morphism $\gamma: C\to \mathbb P(V)^\vee$ given by the quotient  $V_C^\vee \cong \bigwedge^2 V_C \to \bigwedge^2 \mathcal P$, or, equivalently, by the quotient $V_C^\vee \to \mathcal K^\vee$. 
 
 Let $d$ denote the degree of $C_0$ and $d^\vee$ that of its dual curve. Let $\kappa$ (resp. $\iota$) denote the degree of the ramification locus of the map $\nu:C\to C_0$ (resp. $\gamma:C\to C^\vee$). It was shown in \cite[Ex.\,(3.5), p.\,265]{Pi1} that
 \begin{equation}\label{cuspinfl}
 3(d^\vee-d)=\iota-\kappa.
 \end{equation}
 We also have the Pl\"ucker formulas \cite[Ex.\,(3.9), p.\,267]{Pi1}
  \begin{eqnarray*}
  d^\vee & = & d(d-1)-e \\
  d &=&d^\vee(d^\vee -1)-e^\vee\\
  d^\vee&=&2d+2g-2-\kappa\\
  d &=&2d^\vee+2g-2-\iota,
  \end{eqnarray*}
where $e$ (resp. $e^\vee$) denotes the sum of the multiplicities of the Jacobian ideal at the singular points of $C_0$ (resp. $C^\vee$), and $g$ denotes the geometric genus of $C_0$ (i.e., the genus of $C$). Note that the multiplicity of the Jacobian ideal at a point of $C_0$ is equal to the degree of the \emph{conductor} plus the degree of the ramification locus at that point.
Since the degree of the conductor is equal to $2\delta$, where $\delta$ is the sum of the $\delta$-invariants of the singular points of $C_0$, the genus of $C$ is equal to $g=\binom{d-1}2 - \delta$.
In the case that  $C_0$ is a \emph{Pl\"ucker curve}, with $\delta_n$ nodes, $\kappa$ cusps, $\iota$ inflection points, and $\tau_t$ double tangents, we have $e=2(\delta_n +\kappa)+\kappa=2\delta_n+3\kappa$ and similarly $e^\vee=2\tau_t+3\iota$.
\medskip

The \emph{reciprocal} curve of $C_0\subset \mathbb P(V)$ was classically defined (cf. \cite[Art.~80, p.~63]{Sa}) as a curve in the same plane $\mathbb P(V)$, defined via a nondegenerate quadric $Q\subset \mathbb P(V)$. Let $p\in C_0$ be a nonsingular point, let $T_p$ denote the tangent line, and let $p^*$ be the polar point of $T_p$ with respect to $Q$. Define $C^*\subset \mathbb P(V)$ as the closure of the set of points $p^*$. Then $C^*$ and $C^\vee$ are isomorphic under the isomorphism $V\cong V^\vee$ given by $Q$.
\bigskip

 Consider a (local or global) parameterization of the curve $C_0$,
 \[r(t):=(r_0(t):r_1(t):r_2(t)):V_C\to \mathcal O_C(1).\]
Then the dual curve has a parameterization
 \[s(t): V_C^\vee \cong \bigwedge^2 V_C \to \bigwedge^2 \mathcal P\cong \gamma^* \mathcal O_{C^\vee}(1),\]
 where  
 \[s(t):=\bigwedge^2 \binom{r(t)}{r'(t)}=(r_1r'_2-r'_1r_2:-r_0r'_2+r'_0r_2:r_0r'_1-r'_0r_1).\]
  We get
  \[s'(t)=\bigwedge^2 \binom{r'(t)}{r'(t)}+\bigwedge^2 \binom{r(t)}{r''(t)}=(r_1r''_2-r''_1r_2:-r_0r''_2+r''_0r_2:r_0r''_1-r''_0r_1),\]
  since the first summand is $0$.
 Therefore the dual of the dual curve is given by
 \[\bigwedge^2\binom{s(t)}{s'(t)}=(r_0(r'_1r''_2-r''_1r'_2)-r_1(r'_0r''_2-r''_0r'_2)+r_2(r'_0r''_1-r''_0r'_1))(r_0:r_1:r_2).\]
 Note that the multiplier is the determinant of the map $V_C\to \mathcal P^2_C(1)$, i.e., the determinant of the matrix
 \[W(r(t)):=\left( \begin{array}{ccc}
 r_0(t) & r_1(t) & r_2(t)\\
 r'_0(t) & r'_1(t) & r'_2(t)\\
 r''_0(t) & r''_1(t) & r''_2(t)\\
 \end{array} \right).
 \]
 This comes as no surprise, since it follows from the well known corresponding identity for the cross product of vectors in 3-dimensional space:
 \[(r\wedge r')\wedge (r\wedge r'')=(r\wedge r\wedge r'')r'+(r\wedge r' \wedge r'')r=(r\wedge r' \wedge r'')r.\]
 In particular this reproves the classical fact that the dual of the dual curve is the curve itself, provided that the determinant of $W(r(t))$ is not identically 0.
 \medskip

\subsection{Envelopes}

Let  $\mathbb P(V)\cong \mathbb P^2_\bC$ be the projective plane and $C$ a smooth curve. Suppose $Z\subset C\times \mathbb P(V)$ is a flat family of plane curves parameterized by $C$. Let $\pi:Z\to C$ and $\psi:Z\to \mathbb P(V)$ denote the projection maps. The \emph{envelope} of $Z$ is the branch locus $D\subset \mathbb P(V)$ of the finite morphism $\psi$. Note that the degree of $\psi$ is equal to the number of curves in the family passing through a given, general point. For a classical definition, see e.g. \cite[Art.~84, p.~67]{Sa}.

\medskip
 
 Let $C_0\subset \mathbb P(V)$ be a plane curve of degree $d$ and geometric genus $g$. Assume given an exact sequence  of locally free sheaves on the normalization $C$ of $C_0$, with $\mathcal F$ of rank 2, $\mathcal L$ of rank 1,
 \[0\to \mathcal L \to V_{C} \to \mathcal F \to 0,\]
 and consider the induced family of lines $Z:=\mathbb P(\mathcal F)\subset C\times \mathbb P(V)$ parameterized by $C$. Set $\mathcal O_Z(1):=\mathcal O_{\mathbb P(\mathcal F)}(1)=\psi^*\mathcal O_{\mathbb P(V)}(1)$.
    
  \begin{proposition}\label{degenv}
  The rational equivalence class of the envelope $D$ of $Z$ is given by
  \[[D]=\psi_*\bigl((\pi^*(c_1(\Omega^1_{X})+c_1(\mathcal F))+c_1(\mathcal O_Z(1)))\cap [Z]\bigr),\]
  and its degree by 
  \[\deg D=2g-2+\deg c_1(\mathcal F)+\deg \psi.\]
  \end{proposition}
  
  \begin{proof}
  The ramification locus of $\psi$ is the subscheme of  $Z$ defined by the $0$th Fitting ideal $F^0(\Omega^1_\psi)$. The exact sequence
  \[0\to \psi^*\mathcal P^1_{\mathbb P(V)}(1)\cong V_Z \to \mathcal P^1_Z(1) \to \Omega^1_\psi \otimes \mathcal O_Z(1)\to 0,\]
  gives that the class of the ramification locus is equal to 
  \[ c_1(\mathcal P^1_Z(1))\cap [Z]=(c_1(\Omega^1_Z)+3c_1(\mathcal O_Z(1))\cap  [Z].\]
  From the standard exact sequences
\[0 \to \pi^*\Omega^1_{C} \to \Omega^1_Z \to \Omega^1_{Z/C} \to 0 \]
and
\[0 \to \Omega^1_{Z/C} \to \pi^*\mathcal F \otimes \mathcal O_Z(-1) \to \mathcal O_Z \to 0\]
we deduce the formula for the class of $D$. The degree is equal to
\[\deg D=\pi_* \bigl((\pi^*(c_1(\Omega^1_{X})+c_1(\mathcal F))+c_1(\mathcal O_Z(1)) )\psi^*c_1(\mathcal O_{\mathbb P(V)}(1)))\cap [Z]\bigr),\]
hence to
\[\deg D=\bigl((c_1(\Omega^1_C)+c_1(\mathcal F))\pi_*c_1(\mathcal O_Z(1)) +\pi_* c_1(\mathcal O_Z(1))^2\bigr)\cap [C],\]
from which the formula in the proposition follows, since the degree of the Segre class $s_1(\mathcal F)=\pi_* c_1(\mathcal O_Z(1))^2$ is the degree of the finite map $\psi$.
  \end{proof}
  
  The surjection $V_C\to \mathcal F$ induces a surjection $V_{C}^\vee  \cong \bigwedge^2 V_{C}\to \bigwedge^2 \mathcal F$, hence a morphism $C\to \mathbb P(V^\vee)=\mathbb P(V)^\vee$. Let $C'$ denote the image of $C$.
     
  \begin{proposition}\label{env}
  The dual curve of the envelope $D\subset \mathbb P(V)$ of $Z$ is equal to the the curve  $C'$.
    \end{proposition}
  
  \begin{proof}
Let the map $V_C\to \mathcal F$ be given (locally) by $\binom{v(t)}{w(t)}$. Then the curve $C'$ is parameterized (locally) by $\bigwedge^2 \binom{v(t)}{w(t)}=v(t)\wedge w(t)$. Its dual curve is given by
\[\bigwedge^2 \binom{\bigwedge^2 \binom{v(t)}{w(t)}}{\bigwedge^2 \binom{v(t)}{w(t)}'}=(v\wedge w)\wedge(v'\wedge w)+(v\wedge w)\wedge(v\wedge w').\]
As in Subsection \ref{dual}, this gives
\[(v\wedge v'\wedge w)w-(w\wedge v'\wedge w)v+(v\wedge v\wedge w')w-(w\wedge v \wedge w')v,\]
which is equal to
\[ (v\wedge w \wedge w')v +(v\wedge v'\wedge w)w,\]
since the two middle terms are 0.

A local parameterization of $Z$ is given by $(t,u)\mapsto v(t)+uw(t)$. The map $V_Z\to \mathcal P^1_Z(1)$ is then given locally by the matrix
\[
\left( \begin{array}{c}
v(t)+uw(t)\\
v'(t)+uw'(t)\\
w(t)
\end{array}\right) \cong
\left( \begin{array}{c}
v(t)\\
v'(t)+uw'(t)\\
w(t)
\end{array}\right)
\]
The ramification locus of $q$ is the set of $(t,u)$ satisfying
\[\bigwedge^3 \left( \begin{array}{c}
v(t)\\
v'(t)\\
w(t)
\end{array}\right)
-u\bigwedge^3\left( \begin{array}{c}
v(t)\\
w(t)\\
w'(t)	
\end{array}\right) =0.
\]
With this value of $u$, we see that the ramification curve is parameterized (locally) by 
\[(v\wedge w \wedge w')v+(v\wedge v'\wedge w)w.\]
This shows that the dual of the curve $C'$ is equal to $D$ generically. In addition to $C^{'\vee}$, $D$ can contain a finite number of normals to $C$ (see the argument in Example \ref{tgenv}), but since the dual of a line is a point, the dual of 
$D$ is equal to the curve $C'$. 
\end{proof}
  
  \begin{example}\label{tgenv}{\rm
 Let  $C\to C_0\subset \mathbb P(V)$ and 
 \[ 0\to \mathcal K \to V_C\to \mathcal P \to 0\]
 be as in Subsection \ref{dual}.
  The degree of the branch curve $D$ of $\psi: \mathbb P(\mathcal P) \to \mathbb P(V)$ is equal to $2g-2+\deg \mathcal P+\deg \psi=2g-2+2\deg C^\vee$. Now $\deg C=\deg C^{\vee \vee} =2\deg C^\vee +2g-2-\iota$, where $\iota$ is the (weighted) number of cusps of $C^\vee$, hence equal to the (weighted) number of inflection points of $C$. This corresponds to the well known fact the the branch locus of $\psi$ is equal to the curve $C$ union its inflectional tangents.
Locally we can see this as follows. The map $\psi:\mathbb P(\mathcal P)\to \mathbb P(V)$ is given by $(t,u)\mapsto r(t)+ur'(t)$. Its differential is given by
\[d\psi=\binom{r'(t)+ur''(t)}{r'(t)}.\]
So the rank of $d\psi$ is $<2$ when $u=0$ (corresponding to the points on the curve $C$), and also on points $(t_0,u)$, for $t_0$ such that $r''(t_0)=0$ (these are the points on an inflectional tangent). Since the bundle $\mathcal P$ is twisted, the cusps will not appear here.  }
  \end{example}
  
    \begin{example}\label{curveenv} {\rm
  Let $C\to C_0\subset \mathbb P(V)$ be a plane curve. Let $Q\subset \mathbb P(V)$ be a nondegenerate quadric.
   Set $\mathcal F =: \Coker (\mathcal O_C(-1)\to V^\vee_C)$. Identify $V\cong V^\vee$ via the quadric $Q$. Then the surjection $V_C\to \mathcal F$ gives a family of lines $\mathbb P(\mathcal F)\subset C\times \mathbb P(V)$. As observed by Salmon \cite[Art.~84, p.~67]{Sa}, the reciprocal curve $C^*$ is the envelope of this family of lines.}
  \end{example}

  \begin{example} {\rm
  A \emph{caustic} is the envelope of rays either reflected or refracted by a curve.}
  \end{example}

\subsection{Evolutes}

Fix a line $L_\infty\subset \mathbb P(V)$ and a binary quadratic form giving a degree 2 subscheme $Q_\infty\subset L_\infty$. 
  Let $C_0\subset \mathbb P(V)$ be a plane curve, with normalization map $\nu: C\to C_0$. Let $\mathcal E$ denote the \emph{Euclidean normal bundle} of $C_0$ with respect to $L_\infty$ and $Q_\infty$, defined on $C$ \cite[Section 4]{Pi}. Recall that we have $\mathcal E\subseteq \mathcal K^\vee \oplus \mathcal O_C(1)$, where $\mathcal K = \Ker (V_C\to \mathcal P)$ is as in Subsection \ref{dual}. If $L_\infty =\mathbb P(V')$ and ${V'}^\vee \cong V'$ is the isomorphism given by $Q_\infty$, then the map $V_C\to \mathcal K^\vee \oplus \mathcal O_C(1)$ defined by $V_C\to V'_C\cong {V'}^{\vee}_{C} \to {\mathcal K^{\vee}}$ and $V_C \to \mathcal O_C(1)$ factors via the surjection $V_C\to \mathcal E$.
  
  Consider the exact sequence 
  \[0\to \mathcal Q \to V_C \to \mathcal E \to 0.\]
  Define the  \emph{curve of normals} $N_C\subset \mathbb P(V)^\vee$ to be the image of the morphism $C\to \mathbb P(V)^\vee$ given by the quotient $V_C^\vee \cong \bigwedge^2V_C \to \bigwedge^2 \mathcal E$.
The \emph{evolute} $E_C\subset \mathbb P(V)\cong \mathbb P(V)^{\vee\vee}$ of $C$ is the dual curve  of $N_C$ (cf. \cite[Art.~99, p.~82]{Sa}).
By Proposition \ref{env}, 
the envelope of the family of normals $\mathbb P(\mathcal E) \subset C\times \mathbb P(V)$ is equal to $E_C$ union its inflectional tangents (cf. Example \ref{tgenv}).

 \begin{remark}\label{recipdual}{\rm
 Assume the quadric $Q_\infty \subset L_\infty$ is the restriction of a nondegenerate quadric $Q\subset \mathbb P(V)$. Let $p\in C$ be a nonsingular point. By definition, the normal $N_p$ is the line between the point $p$ and the polar of the intersection $T_p\cap L_\infty$ with respect to $Q_\infty$. But this is the same as the line between the point $p$ and the polar point $p^*$ of $T_p$ with respect to $Q$. It follows that the normal to $C$ at $p$ is the same as the normal to the reciprocal curve $C^*$ at $p^*$ (cf. \cite[Art.~109, p.~93]{Sa}.). We conclude that the evolute of the reciprocal curve $C^*$ is the same as the evolute of $C$ (cf. \cite[Thm.~5.2, p.~123]{DHOST}).}
 \end{remark}

 The degree of $N_C$ is equal to the degree of $c_1(\mathcal E)$. By Equation~(\ref{cuspinfl}) the degree of $E_C$ satisfies
  \begin{equation}\label{2}
  3(\deg E_C-\deg N_C)=\iota(N) -\kappa(N),
  \end{equation}
  where $\kappa(N)$ is the (weighted) number of cusps and $\iota(N)$ the (weighted) number of inflections of $N_C$. Note that $\iota(N)=\kappa(E)$ and $\kappa(N)=\iota(E)$.
 By Proposition \ref{degenv}, the degree of the envelope of $\mathbb P(\mathcal E)$ is equal to
 \[2g-2+\deg c_1(\mathcal E)+\deg \psi=2\deg N_C+2g-2=\deg E_C+\iota(E),\]
 where $\psi:\mathbb P(\mathcal E)\to \mathbb P(V)$ and $\iota(E)$ is the (weighted) number of inflectional pooints on $E_C$. This is the same as we get from one of the Pl\"ucker formulas for the class of the curve $N_C$:
 \[\deg E_C=2\deg N_C+2g-2-\kappa_N=2\deg N_C+2g-2-\iota(E).\]
 Since $N_C$ and $E_C$ have the same geometric genus $g$ as $C$, the same formula applied to the curve $C_0$ gives
 \[2g-2=d-2d^\vee+\iota,\]
 hence we get
 \[\deg E_C=2\deg N_C+d^\vee-2d+\kappa-\iota(E).\]

 \medskip
 
 Assume $C_0$ is in generic position with respect to $L_\infty$ and $Q_\infty$, i.e., $C_0\cap L_\infty$ is equal to $d$ distinct points, none contained in $Q_\infty$. Then, by \cite[Prop. 8]{Pi}, $\mathcal E = \mathcal K^\vee\oplus \mathcal O_C(1)$, hence
 \[\deg N_C=(c_1(\mathcal P)+c_1( \mathcal O_C(1)))\cap [C]=d^\vee +d,\]
 where $d:=\deg C_0$ and $d^\vee:=\deg C^\vee$. 
  
 \begin{remark}{\rm
  For each $p\in C_0\cap L_\infty$, we have $N_p=L_\infty$. Hence $L_\infty\in N_C$ is an ordinary $d$-uple point.
  A way to understand the degree of the curve of normals is the following: Let $q\in Q_\infty \subset L_\infty$. The degree of $N_C$ is the number of normals to $C$ passing through $q$. Since $q^\perp=q$, the normal from a point $p\in C_0$, $p\notin L_\infty$, through $q$ is the same as the tangent at that point. To the number $d^\vee$ of such normals, we must add $d$ times $L_\infty$, so we do get $d^\vee+d$ for the degree of $N_C$ (cf. \cite[Art.~111, p.~95]{Sa}. }
  \end{remark}
  
  \medskip

 When $C_0$ is not necessarily in generic position with respect to $L_\infty$, Salmon gives the formula $\deg N_C=d^\vee+d-f-g$, where $f$ is the number of times the curve passes through a circular point and $g$ is the number of times the curve is tangent to $L_\infty$  \cite[Art.~111, p.~94--95]{Sa}. This formula is made precise in \cite[Thm.\,8, p.\,3]{JoPe2}. Indeed,
if $p\in C\cap Q_\infty$, but $L_\infty\ne T_p$, then $N_p=T_p\ne L_\infty$, so the contribution to $\deg N_C$ from $C\cap L_\infty$ is $d-1$ instead of $d$. If $L_\infty=T_p$ for some $p\in C$, then the number of other tangents to $C$ through $p$ is $d^\vee -1$ instead of $d^\vee$. Both these phenomena produce an inflectional tangent of the evolute $E_C$, as explained in \cite[Art.~108, 112, p.~96]{Sa}, so that Salmon's $f+g$ is equal to $\iota_E$.  In both cases, we have $\mathcal E \subsetneq \mathcal K^\vee\oplus \mathcal O_C(1)$, and $\iota_E=\deg c_1(\mathcal K^\vee)+\deg c_1(\mathcal O_C(1))-\deg c_1(\mathcal E)$. Therefore, under the assumption that the intersections of $C$ with $L_\infty$ are not worse than described above,  we get the following formulas for the degrees of $N_C$ and $E_C$.

\begin{proposition}\label{pr:7}
Let $C_0$ be a plane curve of degree $d$ and class $d^\vee$. Let $N_C$ denote its curve of normals and $E_C$ its evolute. Let $\iota(E)$ denote the (weighted) number of inflections on $E_C$. Then we have, as in \cite[Art.\,113, p.\,96]{Sa},
\[\deg N_C=d+d^\vee -\iota_E \text{ and } \deg E_C= 3d+\iota -3\iota(E). \]
The number of cusps of $E_C$ is equal to, as in \cite[Art.\,113, p.\,97]{Sa},
\[ \kappa(E)=6d-3d^\vee +3\iota-5\iota(E).\]
\end{proposition}

\begin{proof}
The last formula follows from Equation~(\ref{2}).
\end{proof}

In the case that the curve is singular at a circular point, or touches the line at infinity at a circular point, these formulas must be modified. If $\iota_E$ denotes the weighted number of inflection points on $E_C$, the formula
\begin{equation}\label{degev}
\deg E_C=2\deg N_C+d^\vee-2d+\kappa-\iota(E)
\end{equation}
still holds. For the degree of the curve of normals, it was shown in \cite[Thm.\,8, p.\,3]{JoPe2} that we have
\begin{equation}\label{degnorm}
\deg N_C=d+d^\vee- \sum_{p\in C_0\cap L_\infty} (i_p(C_0,L_\infty)-m_p(C_0)) - m_{q_1}(C_0)-m_{q_2}(C_0),
\end{equation}
where $i_p$ denotes intersection multiplicity, $m_p$ denotes multiplicity, and  $Q_\infty=\{q_1, q_2\}$.

\begin{corollary}
Assume $C_0$ is in general position with respect to $L_\infty$ and $Q_\infty$. Then the number of cusps of $E_C$ that lie in the affine plane $\mathbb P(V)\setminus L_\infty$ is equal to
\[\kappa(E)-d=5d-3d^\vee+3\iota.\]
\end{corollary}

\begin{proof}
By assumption, $\iota(E)=0$.
 If $p\in C_0\cap L_\infty$, then the evolute has a cusp at $p^\perp \in L_\infty$, and $L_\infty$ is the tangent to the evolute at this point (see \cite[p. 162]{Hi}). 
 \end{proof}

 In the situation of the corollary, if $q\in L_\infty$, then the tangents to $E_C$ passing through $q$ are the $d^\vee$ normals to $C_0$ corresponding to the tangents to $C_0$ through $q^\perp$, plus the line $L_\infty$ counted one time for each cusp of $E_C$ on $L_\infty$. Hence the class of $E_C$ is $d^\vee+d$, which is, as it must be, equal to degree of its dual curve $N_C$.

  \begin{corollary}\label{diam}
  Assume $N_C$ has no cusps (i.e., $\kappa(N)=\iota(E)=0$) and that all singular points of $N_C$ other than the $d$-multiple point are nodes. Then the number  $\delta_N$ of nodes of $N_C$, i.e., the number of (complex) \emph{diameters}  of $C_0$, is equal to  
   \[\textstyle\delta_N=  \frac{1}2({d^\vee}^2 +2d^\vee d-4d^\vee -\kappa),\]
in agreement with \cite[Ex. 6, p.~163]{Hi}.

If in addition $C_0=C$ is nonsingular, we get
\[\textstyle \delta_N=\binom{d}2(d^2+d-4),\]
which agrees with the formula for (half of) the ``bottleneck degree'' $\rm{BND}(C)$ in \cite[Cor.~2.11 (1)]{DiEkWe}.
  \end{corollary}

\begin{proof}
The delta invariant of $N_C$ is equal to 
  \[
  \textstyle\binom{d^\vee +d-\iota(E)-1}2 -g=\binom{d^\vee +d-\iota(E)-1}2-\frac{1}2 (d^\vee-2(d-1)+\kappa).\]
  When $\iota(E)=0$, this gives
  \[\textstyle \delta_N=\frac{1}2({d^\vee}^2 +2d^\vee d-4d^\vee +d(d-1)-\kappa)-\binom{d}{2}
  = \frac{1}2({d^\vee}^2 +2d^\vee d-4d^\vee -\kappa),\]
  since the contribution to the delta invariant from the $d$-uple point is $\binom{d}2$. When $X_0$ is nonsingular, we have $\kappa=0$ and $d^\vee=d(d-1)$.
\end{proof}

 Note that when $C_0$ has cusps ($\kappa\neq0$), then the formula for $\delta_N$ is not expressable as a polynomial in only the degree and class of $C_0$, but that if $\kappa=0$ (e.g., if $C_0$ is a general projection of a smooth curve in $\mathbb P^n$, $n\ge 3$, then it is.
\medskip

 \begin{corollary}
Assume $C_0$ is in general position with respect to $L_\infty$ and $Q_\infty$ and
has only $\delta$ ordinary nodes and $\kappa$ cusps as singularities. Then
  \begin{eqnarray*}
  \deg N_C & = & d^2-2\delta-3\kappa\\
  \deg E_C & = & 3d(d-1)-6\delta -8\kappa\\
  \kappa(E) & = & 3(d(2d-3)-4\delta -5\kappa).
 \end{eqnarray*}
 
 If in addition $C_0$ is nonsingular, then the number of (complex) nodes $\delta_E$ of $E_C$ is
 \[\textstyle \delta_E  =  \frac{d}2(3d-5)(3d^2-d-6).\]
 \end{corollary}
 
 \begin{proof}
   The $\delta$-invariant of $E_C$ is equal to
  \[\textstyle \binom{\deg E_C-1}2 -g=\binom{3d+\iota-\iota(E)-1}2-\frac{1}2(d^\vee-2d+2+\kappa).\]
  When $X_0$ is nonsingular and $\iota(E)=0$, this gives
  \[\textstyle \binom{3d(d-1)-1}2-\binom{d-1}2=\frac{d}2(9d^3-18d^2-d+12).\]
  Assuming that all the cusps of $E_C$ are ordinary and that the remaining singular points are ordinary nodes, we get for the number of nodes of $E_C$,
  \[\textstyle \delta_E=\frac{d}2(9d^3-18d^2-d+12)-3d(2d-3)=\frac{d}2(3d-5)(3d^2-d-6).\]
 \end{proof}
 
     
  \begin{example}\label{nodalcub}{\rm 
  Let $C_0\subset \mathbb P(V)$ be a nodal cubic (in general position with respect to $L_\infty$ and $Q_\infty$). Then $\deg C=d=3$, $\delta =1$, $\iota=3$, and $\deg C^\vee =d^\vee=4$. Its curve of normals $N_C$ has degree $d+d^\vee=d^2-2\delta=7$, and its evolute $E_C=N_C^\vee$ has degree $3d+\iota =9+3=12$. The number of diameters of $X_0$ is
  \[\textstyle\delta_N=\frac{1}2 (16+24-16)=12.\]
The number of cusps of $E_C$ is
\[\kappa(E)=  6d-3d^\vee+3\iota=18-12+9=15.\]
The (complex) \emph{vertices} of $C$ are the cusps of $E_C$ that do not lie on $L_\infty$. Hence, the number of vertices is $\kappa_E-d=\kappa_E-3=12$.
Since $\delta_E+\kappa_E=\frac{1}2(11\cdot 10)=55$, we get $\delta_E=40$. }
  \end{example}

  \begin{example}\label{tricusp} {\rm 
  Let $C_0\subset \mathbb P(V)$ be a tri-cuspidal quartic (in general position with respect to $L_\infty$ and $Q_\infty$). Then $\deg C=d=4$, $\delta =0$, $\kappa=3$, $\tau=1$, $\iota=0$, and $\deg C^\vee =d^\vee=4\cdot 3-3\cdot 3=3$.  (This curve is the dual curve of a nodal cubic.) 
  Its curve of normals $N_C$ has degree $d+d^\vee=7$, and its evolute $E_C=N_C^\vee$ has degree $3d+\iota =12+0=12$. The number of (complex) diameters of $C_0$ is
  \[\textstyle\delta_N=\frac{1}2 (9+24-12-3)=9.\]
The number of cusps of $E_C$ is
\[\kappa(E)=  6d-3d^\vee+3\iota=24-9+0=15.\]
The (complex) \emph{vertices} of $C$ are the cusps of $E_C$ that do not lie on $L_\infty$. Hence, the number of (complex) vertices is $\kappa(E)-d=15-4=11$.
Since $\delta_E+\kappa(E)=\frac{1}2(11\cdot 10)=55$, we get $\delta_E=40$.}
  \end{example}
  
  \begin{remark}{\rm 
  We observed in Remark \ref{recipdual} that the evolute of a curve is equal to the evolute of its reciprocal curve. Since the reciprocal curve is just a realization in the original plane of the dual curve, this explains why the evolutes of the 
curves of Examples \ref{nodalcub} and \ref{tricusp} are  ``equal''.}
  \end{remark}

  \begin{example}\label{rational}{\rm
  Let $C_0\subset \mathbb P(V)$ be a nodal, rational curve of degree $d$, in general position with respect to $L_\infty$ and $Q_\infty$. Then $C_0$ has $\delta=\binom{d-1}2$ nodes, $\kappa =0$ cusps, $\iota=3(d-2)$ inflection points, and $\tau=(d-2)(2d-3)$ double tangents.
  
The degree of the curve of normals is $\deg N_C=3d-2$. By genericity, the curve of normals has no cusps: $\kappa(N)=0$. The curve of normals has one ordinary $d$-uple point and $\delta_N=2(d-1)(2d-3)$ nodes, which are the (complex) \emph{diameters} of $C_0$. The degree of the evolute is $\deg E_C=6(d-1)$. The number of cusps of the evolute is $\kappa(E)=3(3d-4)$, of which $d$ lie on $L_\infty$. Hence the number of (complex) vertices of $C_0$ is $\kappa(E)-d=4(2d-3)$. The number of nodes of the evolute is $\delta_E=2(3d-4)(3d-5)$.}
  \end{example}
  
  \begin{example}\label{param}{\rm 
  Choose coordinates $x,y,z$ in $\mathbb P(V)$ such that $L_\infty: z=0$. Let 
  $(x(t),y(t))$ be a local parameterization of the affine curve $C_0\cap (\mathbb P(V)\setminus L_\infty)\subset \mathbb A^2_\bC$. Assume $Q_\infty:x^2+y^2=0$. Then the normal vectors are $(y'(t),-x'(t))$, and hence the curve of normals (in the dual plane) is given by (in line coordinates)
  \[ V^\vee_C\cong \bigwedge^2 V_C\to \bigwedge^2
  \left( \begin{array}{ccc}
 x(t) & y(t) & 1\\
 y'(t) & -x'(t) & 0
 \end{array} \right)=\bigl(x'(t):y'(t):-x(t)x'(t)-y(t)y'(t)\bigr).   \]
 
 The evolute is obtained by taking the dual of the curve of normals:
 \[\bigwedge^2
  \left( \begin{array}{ccc}
 x'(t)&y'(t)&-x(t)x'(t)-y(t)y'(t)\\
 x''(t) & y''(t) & -x(t)x''(t)-x'(t)^2-y(t)y''(t)-y'(t)^2
 \end{array} \right),\]
 which gives
 \[\bigl(x(x'y''-x''y')-y'({x'}^2+{y'}^2): y(x'y''-x''y')+x'({x'}^2+{y'}^2):x'y''-x''y'\bigr).\]}
  \end{example}
  

  
  
  
  
  
  
  

\noindent
\paragraph{\em Acknowledgements.} The third author is sincerely grateful to the mathematics department of the UiT The Arctic University of Norway for the hospitality.  He also wants to thank Professor E.~Shustin of Tel Aviv University for discussions and Professor R.~Fr\"oberg of Stockholm University for his indispensable help with Macaulay2. Finally, financial support by the Troms\o~Research Foundation within the Pure Mathematics in Norway project is greatly acknowledged.

 \newgeometry{left=7mm,right=5mm,bottom=10mm}
  \section{Formulas of evolutes and  curves of normals }\label{subsec:formulae}

\noindent We collect here the polynomials describing the evolutes and  curves  of normals associated to the examples presented in Section~\ref{sec:zoo}.

\medskip
\noindent
{\bf I.}  For the  rotated rotated ellipse given by $(x+y)^2+4(y-x)^2-1=0$ we have that its evolute satisfies the equation

  \noindent
     {\tiny
\begin{dmath*}
0=8000x^6  + 28800x^5 y + 58560x^4 y^2  + 71424x^3 y^3  + 58560x^2 y^4  + 28800xy^5+ 8000y^6  + 4752x^4  - 25920x^3 y - 68256x^2 y^2  - 25920xy^3  + 4752y^4  + 4860x^2  + 5832xy + 4860y^2 - 729,
\end{dmath*} }   
and its curve of normals is given by   {\tiny$$9u^4  - 80u^2 v^2  + 96uv^2  - 18u^2  - 80v^2  + 9=0.$$ }
\medskip

\noindent
{\bf II.} For the hyperbola given by $(x+y)^2-4(y-x)^2-1=0$ we find that the evolute satisfies the equation 
  \noindent

      {\tiny
\begin{dmath*}  
0=-15625 + 22500 x^2 - 54000 x^4 + 1728 x^6 + 75000 x y - 72000 x^3 y + 17280 x^5 y + 22500 y^2 - 55200 x^2 y^2 + 62784 x^4 y^2 -  72000 x y^3 + 98560 x^3 y^3 - 54000 y^4 + 62784 x^2 y^4 + 17280 x y^5 + 1728 y^6,
 \end{dmath*}
 }
and its curve of normals is given by 
{\tiny
\[ 25 u^4 - 48 u^2 v^2 + 160 uv^2 - 50 u^2 - 48 v^2 + 25 =0.\]
}
\medskip
\noindent
{\bf III.} For the cissoid given by $(x^2+y^2)x-4y^2=0$ its evolute satisfies the equation
{\tiny
\begin{dmath*}
27y^4+4608y^2+32768x=0,
\end{dmath*}
}
\noindent and its curve of normals
{\tiny
\begin{dmath*}
36u^3v+uv^3+12u^2v^2+64u^2+96uv+128=0.
\end{dmath*}
}

\medskip
\noindent
{\bf IV.}
For the nodal cubic $(x^2-y^2)(x-1)+(x^2+y^2)/5=0$ we find that the evolute satisfies the equation 

{\tiny

\begin{dmath*}
0=31640625\,x^{8}y^{4}+393359375\,x^{6}y^{6}-1369921875 \,x^{4}y^{8}+1433203125\,x^{2}y^{10}-488281250\,y^{12}-412031250\,x^{7}y^{4} -6307500000\,x^{5}y^{6}+13851093750\,x^{3}y^{8}-7131562500\,x\,y^{10}-
  270000000\,x^{8}y^{2}-1235953125\,x^{6}y^{4}+52577250000\,x^{4}y^{6}-44355328125\,x^{2}y^{8}+17990718750\,y^{10}-345600000\,x^{9}-734400000\,x^{7}y^{2}+64522912500\,x^{5}y^{4}-256637837500\,x^{3}y^{6}+82223775000\,x
     \,y^{8}+5391360000\,x^{8}+61315920000\,x^{6}y^{2}-543557615625\,x^{4}y^{4}+569477064375\,x^{2}y^{6}-143503008750\,y^{8}-35043840000\,x^{7}-583922304000\,x^{5}y^{2}+1833671022750\,x^{3}y^{4}-584714902500\,x\,y^{6}+
     116110540800\,x^{6}+2488494614400\,x^{4}y^{2}-2639387549475\,x^{2}y^{4}+487482536250\,y^{6}-194973143040\,x^{5}-5465670635520\,x^{3}y^{2}+1342636676640\,x\,y^{4}+137371852800\,x^{4}+6074689654656\,x^{2}y^{2}- 207256169664\,y^{4}-33191424000\,x^{3}-3188439313920\,x\,y^{2}+     627317913600\,y^{2}
 \end{dmath*}}
 
 and its curve of normals is given by 
 
 {\tiny

\begin{dmath*}
0=576 u^7 + 2640 u^6v + 4225 u^5v^2 - 2208 u^5 + 2750 u^4v^3 - 7800 u^4v + 625 u^3v^4  - 8325 u^3v^2 + 2016 u^3- 2750 u^2v^3 + 4920 u^2v - 625 uv^4 + 3000 uv^2 + 720 v
 \end{dmath*}}
\medskip
\noindent
{\bf V.}  For the  generic cubic in the Weierstra\ss form $y^2+x(x-2)(x+1)=0$ we have that its evolute satisfies the equation 

 {\tiny
\begin{dmath*}
0=800000 - 2880000 x^2 - 64000 x^3 + 3916800 x^4 + 210432 x^5 -   2470400 x^6 - 247296 x^7 + 710400 x^8 + 119296 x^9 - 76800 x^{10} - 
  18432 x^{11} - 19680000 y^2 + 94944000 x y^2 - 69350400 x^2 y^2 + 
  2580480 x^3 y^2 + 31350240 x^4 y^2 - 21081600 x^5 y^2 - 
  4652544 x^6 y^2 + 4719552 x^7 y^2 + 152352 x^8 y^2 - 
  15552 x^9 y^2 - 135854400 y^4 + 53402880 x y^4 - 32003904 x^2 y^4 + 
  180044544 x^3 y^4 - 169396176 x^4 y^4 + 56111376 x^5 y^4 - 
  12544560 x^6 y^4 + 879120 x^7 y^4 + 6561 x^8 y^4 - 54925856 y^6 + 
  174791424 x y^6 - 102761952 x^2 y^6 + 26959328 x^3 y^6 + 
  6351816 x^4 y^6 - 978840 x^5 y^6 - 25779744 y^8 + 31683600 x y^8 - 
  32542272 x^2 y^8 + 3354912 x^3 y^8 - 275562 x^4 y^8 - 
  1113480 y^{10} + 3283416 x y^{10} + 524880 x^2 y^{10} - 964467 y^{12}
\end{dmath*}}

 and its curve of normals is given by 
 
 {\tiny
\begin{dmath*}
0=-20 u^3 + 40 u^5 + 12 u^7 + 8 u^2 v + 112 u^4 v - 24 u^6 v + 20 uv^2 + 4 u^3 v^2 - 24 u^5 v^2 - 8 v^3 - 12 u^2 v^3 +  36 u^4 v^3 + 27 u^3 v^4.
\end{dmath*}}
\medskip
\noindent
{\bf VI.}
For the general cubic $(x^2 - y^2) (x - 1)  + 1/64=0$ we find that the evolute satisfies the equation 

{\tiny

\begin{dmath*}
0=7421703487488\,{x}^{14}{y}^{4}+64494878184177664\,{x}^{12}{y}^{6}-
386768883110903808\,{x}^{10}{y}^{8}+966636472192991232\,{x}^{8}{y}^{10
}-1289014380968542208\,{x}^{6}{y}^{12}+967334112320815104\,{x}^{4}{y}^
{14}-387347775982927872\,{x}^{2}{y}^{16}+64658155660902400\,{y}^{18}+
666304046432256\,{x}^{13}{y}^{4}-778441038325874688\,{x}^{11}{y}^{6}+
3866799875957981184\,{x}^{9}{y}^{8}-7699915113687416832\,{x}^{7}{y}^{
10}+7673361907876626432\,{x}^{5}{y}^{12}-3825970611172147200\,{x}^{3}{
y}^{14}+763498675304398848\,x{y}^{16}-26388279066624\,{x}^{14}{y}^{2}-
205152377068584960\,{x}^{12}{y}^{4}+5083847097719980032\,{x}^{10}{y}^{
6}-19140906446254768128\,{x}^{8}{y}^{8}+29942806645177319424\,{x}^{6}{
y}^{10}-22738968088198250496\,{x}^{4}{y}^{12}+8025005517648691200\,{x}
^{2}{y}^{14}-966605960745320448\,{y}^{16}-2199023255552\,{x}^{15}-
21596538673299456\,{x}^{13}{y}^{2}+2536865915631304704\,{x}^{11}{y}^{4
}-22461106519013326848\,{x}^{9}{y}^{6}+61385662332521152512\,{x}^{7}{y
}^{8}-73317615121905745920\,{x}^{5}{y}^{10}+39792563429809586176\,{x}^
{3}{y}^{12}-7905764100091674624\,x{y}^{14}-158329674399744\,{x}^{14}+
469153502673764352\,{x}^{12}{y}^{2}-14465716639116558336\,{x}^{10}{y}^
{4}+70118067260381724672\,{x}^{8}{y}^{6}-139200248813388300288\,{x}^{6
}{y}^{8}+123331089024040304640\,{x}^{4}{y}^{10}-45671665675624513536\,
{x}^{2}{y}^{12}+5183040690271027200\,{y}^{14}-3084130115911680\,{x}^{
13}-3773463038250713088\,{x}^{11}{y}^{2}+49135006597358026752\,{x}^{9}
{y}^{4}-158402304530102353920\,{x}^{7}{y}^{6}+231593231173476679680\,{
x}^{5}{y}^{8}-141179764176001695744\,{x}^{3}{y}^{10}+
29953016692972978176\,x{y}^{12}+3216831720456192\,{x}^{12}+
15913681463959093248\,{x}^{10}{y}^{2}-110067536968204419072\,{x}^{8}{y
}^{4}+264218350040905678848\,{x}^{6}{y}^{6}-284291446439694827520\,{x}
^{4}{y}^{8}+108088834974485053440\,{x}^{2}{y}^{10}-
13654134374508003328\,{y}^{12}+298979459522887680\,{x}^{11}-
39615577897360687104\,{x}^{9}{y}^{2}+174237151222548135936\,{x}^{7}{y}
^{4}-330863891813719080960\,{x}^{5}{y}^{6}+242242147522191556608\,{x}^
{3}{y}^{8}-48966154269726081024\,x{y}^{10}-2132998239434047488\,{x}^{
10}+60737468293120524288\,{x}^{8}{y}^{2}-207506464597447016448\,{x}^{6
}{y}^{4}+312203958027162746880\,{x}^{4}{y}^{6}-127672280296483454976\,
{x}^{2}{y}^{8}+17922286164400668672\,{y}^{10}+6564306327886102528\,{x}
^{9}-58892920083012648960\,{x}^{7}{y}^{2}+195076184296488173568\,{x}^{
5}{y}^{4}-206261726242340339712\,{x}^{3}{y}^{6}+32860880827384332288\,
x{y}^{8}-10432606285219233792\,{x}^{8}+40660545111173627904\,{x}^{6}{y
}^{2}-147741599316188332032\,{x}^{4}{y}^{4}+74058664818128191488\,{x}^
{2}{y}^{6}-10106230522754629632\,{y}^{8}+7987653229758382080\,{x}^{7}-
28008745012461305856\,{x}^{5}{y}^{2}+85041662198783410176\,{x}^{3}{y}^
{4}-2023434471830716416\,x{y}^{6}-1357465116912074752\,{x}^{6}+
21107281103406907392\,{x}^{4}{y}^{2}-27640381831814283264\,{x}^{2}{y}^
{4}+389697666630647808\,{y}^{6}-1481623791350906880\,{x}^{5}-
12373991979513348096\,{x}^{3}{y}^{2}-666456918859776000\,x{y}^{4}+
476388724672954368\,{x}^{4}+4276799734537617408\,{x}^{2}{y}^{2}+
368722254228062208\,{y}^{4}+85538840944874240\,{x}^{3}-
155934646676264448\,x{y}^{2}-33449372458335744\,{x}^{2}+
30737114622630144\,{y}^{2}-1608355473256704\,x+722219801960291
\end{dmath*}}
\medskip

\noindent
 and its curve of normals is given by 
 
 {\tiny
\begin{dmath*}
0=u^{9}+256\,u^{8}v+16896\,u^{7}v^{2}+65792\,u^{6}v^{3 }+98304\,u^{5}v^{4}+65536\,u^{4}v^{5}+16384\,u^{3}v^{6}+1009\,u^{7}+3328\, u^{6}v-43008\,u^{5}v^{2}-126976\,u^{4}v^{3}-131072\,u^{3}v^{4}-65536\,u^{2}v^{5}-16384\,u\,v^{6}-1973\,u^{5}-3328\,u^{4}v+44544\,u^{3}v^{2}+64768\,u^{2}v^{3}+32768\,u\,v^{4}+899\,u^{3}-256\,u^{2}v-14336\,u\,v^{2}+512\,v^{3}.
\end{dmath*}}
\medskip

\noindent
{\bf VII.}
For the ampersand curve given by the equation  $(-1 + x) (-3 + 2 x) (-x^2 + \frac{y^2}{2}) - 4 (-2 x + x^2 + \frac{y^2}{2})^2=0$ we find that the evolute satisfies the equation 

{\tiny
\begin{dmath*}
0=5511466450944000000 x^{18} + 69009662803968000000 x^{16}y^2 + 
 454041885818880000000x^{14}y^4 + 1880330535124992000000x^{12}y^6 + 
 5345655877287936000000x^{10}y^8 + 10390802432311296000000x^8y^{10} + 
 13339555563405312000000x^6y^{12} + 9139138881454080000000x^4y^{14} + 
 516470220914688000000x^2y^{16} - 4172637165060096000000y^{18} - 
 64556046168883200000x^{17} - 789841921808793600000x^{15}y^2 - 
 4717469091314073600000x^{13}y^4 - 17310072284754739200000x^{11}y^6 - 
 41419320989589504000000x^9y^8 - 64119262672325836800000x^7y^{10} - 
 56510219189315174400000x^5y^{12} - 
 16514783351105126400000x^3y^{14} + 14189132952974131200000 xy^{16} + 
 479273840698982400000x^{16} + 2768688090835845120000x^{14}y^2 + 
 9361119428837498880000x^{12}y^4 + 53036338585741885440000x^{10}y^6 + 
 198593670536522588160000x^8y^8 + 
 385976265612887900160000x^6y^{10} + 
 156102041598933565440000x^4y^{12} - 
 165122138814001643520000x^2y^{14} + 20569725762943057920000y^{16} - 
 2441112028320890880000x^{15} - 1297725291667390464000x^{13}y^2 + 
 59677698364616171520000x^{11}y^4 + 15411963692995633152000x^9y^6 - 
 868539197495937761280000x^7y^8 - 
 1318818013502133583872000x^5y^{10} + 
 161133138253823016960000x^3y^{12} + 
 170398879576067211264000 xy^{14} + 9770202768808083456000x^{14} - 
 49456225953852044083200x^{12}y^2 - 
 328074665789973717504000x^{10}y^4 + 
 93878636823930729369600x^8y^6 + 
 2795509334576630486016000x^6y^8 + 
 464954283105194856038400x^4y^{10} + 
 188958757229861621760000x^2y^{12} - 165200704420888874188800y^{14} - 
 31529832219803878686720x^{13} + 303544424692845998899200x^{11}y^2 + 
 926371061110394275829760x^9y^4 - 
 3747819321749535408230400x^7y^6 - 
 2973429851944718044200960x^5y^8 + 
 1477642254834900016742400x^3y^{10} - 
 381211747138806044590080xy^{12} + 85614841611796218904576x^{12} - 
 1078016613583757824819200x^{10}y^2 - 
 1041849832068904881193728x^8y^4 + 
 11346560638136539363105280x^6y^6 + 
 895893179712472723166208x^4y^8 - 
 2696999420593780028467200x^2y^{10} + 427788300003766762512384y^{12} - 
 198357132132989480730624x^{11} + 2628618143845183404392448x^9y^2 - 
 2036013473969817278971392x^7y^4 - 
 13204872129378101105256960x^5y^6 + 
 492689269773585678532608x^3y^8 + 
 1316567506701481365510144xy^{10} + 396550012793183158714368x^{10} - 
 4588956519916286903657472x^8y^2 + 
 11219652644662107620252928x^6y^4 + 
 3354625992301094644254336x^4y^6 + 
 1190920572545964905977344x^2y^8 - 623038688469087703386624y^{10} - 
 691507745750015024807936x^9 + 5543028047756058225863424x^7y^2 - 
 22590031523777430946108800x^5y^4 + 
 5636759081952628005950720x^3y^6 - 555746023848741721855488xy^8 + 
 1045450945995318948120576x^8 - 3706401234038334720496128x^6y^2 + 
 23662827441239781026569824x^4y^4 - 
 6981194929097981792227488x^2y^6 + 325739861292108362174208y^8 - 
 1380102793606261280311296x^7 - 666183695858358116066496x^5y^2 - 
 11030931021992461708485024x^3y^4 + 
 2376500257894024307732640xy^6 + 1564154216003823193668160x^6 + 
 4579931612162068642790016x^4y^2 + 
 1259188508355594228237792x^2y^4 - 311109184315558946826424y^6 - 
 1506839814034251881049792x^5 - 5131962297288198923626224x^3y^2 + 
 1079660070840409026907104xy^4 + 1215111170084394604620528x^4 + 
 2113492574765255484253248x^2y^2 - 325085551620906170070891y^4 - 
 749312183644081938387648x^3 - 232367367731112288440676xy^2 + 
 368436954348821351723724x^2 - 43158007590429842976492y^2 - 
 72453454503241228456644x + 4591488798441627367329.
\end{dmath*}}

 and its curve of normals is given by 
 
 {\tiny
\begin{dmath*}
0=87362u^{10}+615296u^9v+1715408u^8v^2+2426304u^7v^3+1849632u^6v^4+725760u^5v^5+115200u^4v^6-364287u^8-1181120u^7v-1324760u^6v^2-242464u^5v^3+560784u^4v^4+362880u^3v^5+57600u^2v^6+721335u^6+1284976u^5v+816008u^4v^2+306720u^3v^3+255536u^2v^4+120960uv^5+19200v^6-456152u^4-305200u^3v-215984u^2v^2-124800uv^3-37632v^4+31692u^2-2592uv+13248v^2.
\end{dmath*}}
\medskip
\noindent
{\bf VIII.}
For the cross curve with equation $x^2y^2-4x^2-y^2=0$ we find that the evolute satisfies the equation 

{\tiny
\begin{dmath*}
0=729\,{x}^{10}{y}^{8}+2916\,{x}^{8}{y}^{10}-15552\,{x}^{10}{y}^{6}-
63909\,{x}^{8}{y}^{8}-15552\,{x}^{6}{y}^{10}+193536\,{x}^{10}{y}^{4}+
813888\,{x}^{8}{y}^{6}+405162\,{x}^{6}{y}^{8}+255744\,{x}^{4}{y}^{10}-
1048576\,{x}^{12}-5259264\,{x}^{10}{y}^{2}-4760448\,{x}^{8}{y}^{4}-
5609600\,{x}^{6}{y}^{6}-5464338\,{x}^{4}{y}^{8}-697344\,{x}^{2}{y}^{10
}-1048576\,{y}^{12}+5898240\,{x}^{10}+39121920\,{x}^{8}{y}^{2}+
40709376\,{x}^{6}{y}^{4}+68111424\,{x}^{4}{y}^{6}+17387325\,{x}^{2}{y}
^{8}+25896960\,{y}^{10}-93229056\,{x}^{8}-530841600\,{x}^{6}{y}^{2}-
374813568\,{x}^{4}{y}^{4}-231180480\,{x}^{2}{y}^{6}-350999001\,{y}^{8}
+182255616\,{x}^{6}+2202992640\,{y}^{2}{x}^{4}+3063536640\,{x}^{2}{y}^
{4}+2416483584\,{y}^{6}-1560674304\,{x}^{4}-10510663680\,{x}^{2}{y}^{2
}-9406347264\,{y}^{4}-2293235712\,{x}^{2}+5159780352\,{y}^{2}-
764411904
\end{dmath*}}

 and its curve of normals is given by 
 
 {\tiny
\begin{dmath*}
0=4u^{10} + u^8 - 27u^6v^2 + 3u^4v^4 - u^2v^6 - 32u^6 + 36u^4v^2 + 12u^2v^4 - 8u^4 - 72u^2v^2 + 64u^2 + 16
\end{dmath*}}

{\bf IX.}
For the bean curve with equation $x^4+x^2y^2+y^4-x(x^2+y^2)=0$ we find that the evolute satisfies the equation 

{\tiny
\begin{dmath*}
0=-31539456 x^2 + 350977536 x^3 - 1917858816 x^4 + 7090366464 x^5 -
19873460224 x^6 + 43800502272 x^7 - 77198426112 x^8 +
109664534528 x^9 - 124091891712 x^{10} + 108723437568 x^{11} -
71203553280 x^{12} + 33124515840 x^{13} - 9937354752 x^{14} +
1528823808 x^{15} - 31539456 y^2 + 361490688 x y^2 -
1748083680 x^2 y^2 + 1949253120 x^3 y^2 + 13233944448 x^4 y^2 -
95767079424 x^5 y^2 + 392916556800 x^6 y^2 - 1153118462976 x^7 y^2 +
2541323182080 x^8 y^2 - 4316317065216 x^9 y^2 +
5711357509632 x^{10} y^2 - 5865968222208 x^{11} y^2 +
4611386474496 x^{12} y^2 - 2707355860992 x^{13} y^2 +
1130565206016 x^{14} y^2 - 304427040768 x^{15} y^2 +
41278242816 x^{16} y^2 + 289103904 y^4 - 5905759248 x y^4 +
48750458715 x^2 y^4 - 238306655304 x^3 y^4 + 1040955570732 x^4 y^4 -
3732221426928 x^5 y^4 + 9671916451200 x^6 y^4 -
17913206267904 x^7 y^4 + 24089412325824 x^8 y^4 -
23687011761408 x^9 y^4 + 16982457449472 x^{10} y^4 -
8815705694208 x^{11} y^4 + 3347436616704 x^{12} y^4 -
942421340160 x^{13} y^4 + 154067816448 x^{14} y^4 + 11144688831 y^6 -
82761154920 x y^6 + 539350526268 x^2 y^6 - 2993483501120 x^3 y^6 +
10353776994672 x^4 y^6 - 22032469027008 x^5 y^6 +
30444692917312 x^6 y^6 - 28682276426496 x^7 y^6 +
19309670210304 x^8 y^6 - 10385947063296 x^9 y^6 +
5162451812352 x^{10} y^6 - 2016306597888 x^{11} y^6 +
385024720896 x^{12} y^6 + 10389105792 y^8 - 561650718672 x y^8 +
3482787853200 x^2 y^8 - 9177621332544 x^3 y^8 +
12228895076352 x^4 y^8 - 8069718410496 x^5 y^8 +
3055349272320 x^6 y^8 - 3206575558656 x^7 y^8 +
4143727632384 x^8 y^8 - 2573903278080 x^9 y^8 +
610314227712 x^{10} y^8 + 215802164448 y^{10} - 896116472448 x y^{10} +
412614261120 x^2 y^{10} + 3286274637312 x^3 y^{10} -
6260831013888 x^4 y^{10} + 2994602876928 x^5 y^{10} +
2094082025472 x^6 y^{10} - 2464365441024 x^7 y^{10} +
739237072896 x^8 y^{10} - 186640115968 y^{12} + 1214872899840 x y^{12} -
3115413377280 x^2 y^{12} + 3164174807040 x^3 y^{12} -
223384200192 x^4 y^{12} - 1483045687296 x^5 y^{12} +
631825256448 x^6 y^{12} + 37629709056 y^{14} + 88194438144 x y^{14} -
4616331264 x^2 y^{14} - 621021966336 x^3 y^{14} +
411913506816 x^4 y^{14} - 20777472000 y^{16} - 92565504000 x y^{16} +
170201088000 x^2 y^{16} + 46656000000 y^{18}
\end{dmath*}}

 and its curve of normals is given by 
 
 {\tiny

\begin{dmath*}
0=4 u^2 + 4 u^4 + 24 u^6 + 20 u^8 + 4 u^{10} - 8 u v + 68 u^3 v +
104 u^5 v + 56 u^7 v - 4 u^9 v + 136 u^2 v^2 + 63 u^4 v^2 +
51 u^6 v^2 - 55 u^8 v^2 + 64 u v^3 + 40 u^3 v^3 + 120 u^5 v^3 -
20 u^7 v^3 + 32 v^4 + 276 u^2 v^4 + 312 u^4 v^4 + 220 u^6 v^4 +
336 u v^5 + 336 u^3 v^5 + 336 u^5 v^5 + 144 v^6 + 144 u^2 v^6 +
144 u^4 v^6.
\end{dmath*}}
\smallskip

  \noindent
{\bf X.}
For the trifolium with equation $(x^2+y^2)^2-x^3+3xy^2=0$ we find that the evolute satisfies the equation 

{\tiny
\begin{dmath*}
0=128000x^{10} + 603136x^8y^2 + 1230848x^6y^4 + 1288192x^4y^6 + 
 656384x^2y^8 + 123904y^{10} + 614400x^9 - 110592x^7y^2 - 
 3280896x^5y^4 - 5148672x^3y^6 - 1806336xy^8 + 1175040x^8 - 
 1520640x^6y^2 + 4976640x^4y^4 + 8156160x^2y^6 + 483840y^8 + 
 684288x^7 - 684288x^5y^2 - 3421440x^3y^4 - 2052864xy^6 - 
 1632960x^6 + 14136768x^4y^2 - 17589312x^2y^4 + 482112y^6 - 
 3452544x^5 + 6905088x^3y^2 + 10357632xy^4 - 1527984x^4 - 
 3055968x^2y^2 - 1527984y^4 + 2344464x^3 - 7033392xy^2 + 
 3306744x^2 + 3306744y^2 - 1594323
\end{dmath*}}
\smallskip
  \noindent

 and its curve of normals is given by 
 
 {\tiny
\begin{dmath*}
0=81u^6-108u^5v-540u^4v^2+384u^3v^3+1024u^2v^4-486u^4+216u^3v - 1080u^2v^2-1152uv^3+1024v^4+729u^2+324uv-540v^2
\end{dmath*}}
\smallskip

  \noindent
{\bf XI.}
For the quadrifolium given by the equation $(x^2 + y^2)^3 - 4x^2y^2=0$ we find that the evolute satisfies the equation 

{\tiny
\begin{dmath*}
0=746496x^{14} + 5085504x^{12}y^2 + 14983137x^{10}y^4 + 
 24747363x^8y^6 + 24747363x^6y^8 + 14983137x^4y^{10} + 
 5085504x^2y^{12} + 746496y^{14} - 1741824x^{12} + 6796224x^{10}y^2 + 
 40595580x^8y^4 + 64220040x^6y^6 + 40595580x^4y^8 + 
 6796224x^2y^10 - 1741824y^12 + 2598912x^10 - 8320320x^8y^2 + 
 38078208x^6y^4 + 38078208x^4y^6 - 8320320x^2y^8 + 
 2598912y^10 - 4220928x^8 + 66646080x^6y^2 - 168386688x^4y^4 + 
 66646080x^2y^6 - 4220928y^8 + 4338688x^6 - 19968000x^4y^2 - 
 19968000x^2y^4 + 4338688y^6 - 3228672x^4 + 11126784x^2y^2 - 
 3228672y^4 + 2555904x^2 + 2555904y^2 - 1048576
\end{dmath*}}
\smallskip
  \noindent

 and its curve of normals is given by 

 {\tiny
\begin{dmath*}
0=64u^6v^2 - 432u^4v^4 + 729u^2v^6 - 64u^6 + 336u^4v^2 - 1296u^2v^4 + 729v^6 + 128u^4 + 336u^2v^2 - 432v^4 - 64u^2 + 64v^2 
\end{dmath*}}
\smallskip

  \noindent
{\bf XII.}
For the D\"urer folium given by the equation  $(x^2 + y^2) (2(x^2 + y^2) - 1)^2 - x^2=0$ we find that the evolute satisfies the equation 
\smallskip
  \noindent
  
{\tiny
\begin{dmath*}
0=9000000x^{10} + 45432000x^8y^2 + 91733184x^6y^4 + 
  92607552x^4y^6 + 46743552x^2y^8 + 9437184y^{10} - 330000x^8 - 
  2364960x^6y^2 - 5116176x^4y^4 - 4457472x^2y^6 - 1376256y^8 - 
  31175x^6 - 14424x^4y^2 + 57792x^2y^4 + 40960y^6 + 1149x^4 - 
  11868x^2y^2 + 1536y^4 + 27x^2 - 48y^2 - 1
\end{dmath*}}
\smallskip
  \noindent

 and its curve of normals is given by 
\smallskip
  \noindent 
  
 {\tiny
\begin{dmath*}
0=16u^4v^2 - 432u^2v^4 + 2916v^6 - u^4 + 41u^2v^2 - 540v^4 - u^2 + 25v^2.
\end{dmath*}}
\smallskip
  \noindent
{\bf XV.}
For the ranunculoid we find that the evolute satisfies the equation 

{\tiny
\begin{dmath*}
0=823543x^{12} + 4941258x^{10}y^2 + 12353145x^8y^4 + 
 16470860x^6y^6 + 12353145x^4y^8 + 4941258x^2y^{10} + 
 823543y^{12} - 17647350x^{10} - 88236750x^8y^2 - 176473500x^6y^4 - 
 176473500x^4y^6 - 88236750x^2y^8 - 17647350y^{10} - 4501875x^8 - 
 18007500x^6y^2 - 27011250x^4y^4 - 18007500x^2y^6 - 
 4501875y^8 - 88812500x^6 - 266437500x^4y^2 - 266437500x^2y^4 - 
 88812500y^6 - 7290000000x^5 + 72900000000x^3y^2 - 
 36450000000 xy^4 - 1787109375x^4 - 3574218750x^2y^2 - 
 1787109375y^4 - 36621093750x^2 - 36621093750y^2 - 762939453125
\end{dmath*}}
\smallskip
  \noindent

  and its curve of normals is given by 
  \smallskip
  \noindent
  
 {\tiny
\begin{dmath*}
0=78125u^7-109375u^6v+35000u^4v^3-2800u^2v^5+64v^7-703125u^5-328125u^4v+70000u^2v^3-2800v^5-390625u^3-328125u^2v+35000v^3+390625u-109375v 
\end{dmath*}}
\smallskip
  \noindent
\bigskip


\newpage
\noindent {\bf References}

\begin{biblist}



\bib{Ar1}{book}{
   author={Arnol\cprime d, V. I.},
   title={Huygens and Barrow, Newton and Hooke},
   note={Pioneers in mathematical analysis and catastrophe theory from
   evolvents to quasicrystals;
   Translated from the Russian by Eric J. F. Primrose},
   publisher={Birkh\"{a}user Verlag, Basel},
   date={1990},
   pages={118},
   isbn={3-7643-2383-3},
}

\bib{Ar2}{book}{
   author={Arnol\cprime d, V. I.},
   title={Topological invariants of plane curves and caustics},
   series={University Lecture Series},
   volume={5},
   note={Dean Jacqueline B. Lewis Memorial Lectures presented at Rutgers
   University, New Brunswick, New Jersey},
   publisher={American Mathematical Society, Providence, RI},
   date={1994},
   pages={viii+60},
   isbn={0-8218-0308-5},
}

\bib{Ba}{article}{
   author={Baillif, Mathieu},
   title={Curves of constant diameter and inscribed polygons},
   journal={Elem. Math.},
   volume={64},
   date={2009},
   number={3},
   pages={102--108},
   issn={0013-6018},
}


\bib{Br}{article}{
   author={Brusotti,Luigi},
   title={Sulla ``piccola variazone'' di una curva piana algebraica reali},
   journal={Rend. Rom. Ac. Lincei},
   volume={30},
   date={1921},
   number={5},
   pages={375--379},
   issn={},
}

\bib{CT}{article}{
   author={Catanese, Fabrizio},
   author={Trifogli, Cecilia},
   title={Focal loci of algebraic varieties. I},
   note={Special issue in honor of Robin Hartshorne},
   journal={Comm. Algebra},
   volume={28},
   date={2000},
   number={12},
   pages={6017--6057},
   issn={0092-7872},
}

\bib{Co}{book}{
   author={Coolidge, Julian Lowell},
   title={A treatise on algebraic plane curves},
   publisher={Dover Publications, Inc., New York},
   date={1959},
   pages={xxiv+513},
}

\bib{DiEkWe}{article}{
   author={Di Rocco, Sandra},
   author={Eklund, David},
   author={Weinstein, Madeleine},
   title={The bottleneck degree of algebraic varieties},
   journal={SIAM J. Appl. Algebra Geom.},
   volume={4},
   date={2020},
   number={1},
   pages={227--253},
}

\bib{DHOST}{article}{
   author={Draisma, Jan},
   author={Horobe\c{t}, Emil},
   author={Ottaviani, Giorgio},
   author={Sturmfels, Bernd},
   author={Thomas, Rekha R.},
   title={The Euclidean distance degree of an algebraic variety},
   journal={Found. Comput. Math.},
   volume={16},
   date={2016},
   number={1},
   pages={99--149},
   issn={1615-3375},
}

\bib{Fu}{article}{
   author={Fuchs, Dmitry},
   title={Evolutes and involutes of spatial curves},
   journal={Amer. Math. Monthly},
   volume={120},
   date={2013},
   number={3},
   pages={217--231},
   issn={0002-9890},
}

\bib{FuTa}{article}{
   author={Fukunaga, T.},
   author={Takahashi, M.},
   title={Evolutes of fronts in the Euclidean plane},
   journal={J. Singul.},
   volume={10},
   date={2014},
   pages={92--107},
}

\bib{Ga}{article}{
   author={Garay, Mauricio D.},
   title={On vanishing inflection points of plane curves},
   language={English, with English and French summaries},
   journal={Ann. Inst. Fourier (Grenoble)},
   volume={52},
   date={2002},
   number={3},
   pages={849--880},
}

\bib{M2}{article}{
          author = {Grayson, Daniel R.}, 
          author = {Stillman, Michael E.},
          title = {Macaulay2, a software system for research in algebraic geometry},
        }

\bib{Gu}{article}{
   author={Gudkov, D. A.},
   title={The topology of real projective algebraic varieties},
   language={Russian},
   note={Collection of articles dedicated to the memory of Ivan Georgievi\v{c}
   Petrovski\u{\i}, II},
   journal={Uspehi Mat. Nauk},
   volume={29},
   date={1974},
   number={4(178)},
   pages={3--79},
}

\bib{Gu2}{article}{
   author={Gudkov, D. A.},
   title={ Bifurcation of simple double points and cusps of real, plane, algebraic curves}, 
     language={Russian},
   journal={  Dokl. Akad. Nauk SSSR}, 
   volume={142}, 
     date={1962}, 
      pages={990--993}, 
      }

\bib{Hi}{book}{
   author={Hilton, Harold},
   title={Plane algebraic  curves},
   publisher={Oxford University Press},
   date={1932},
   pages={xv+390},
}

\bib{Huy}{book}{
   author={Huygens,Christiaan},
   title={Horologium Oscillatorium: Sive de Motu Pendulorum ad Horologia Aptato Demonstrationes Geometricae},
   publisher={https://archive.org/details/B-001-004-158/Horologium-oscillatorium/page/n11/mode\-/2up},
   date={1673},
   }

\bib{HuyB}{book}{
   author={Huygens, Christiaan},
   author={Blackwell, Richard J.},
   title={Christiaan Huygens' the pendulum clock, or, Geometrical demonstrations concerning the motion of pendula as applied to clocks},
   publisher={Ames: Iowa State University Press},
   date={1986},
   }

\bib{JoPe}{article}{
   author={Josse, Alfrederic},
   author={P\`ene, Fran\c{c}oise},
   title={On the degree of caustics by reflection},
   journal={Comm. Algebra},
   volume={42},
   date={2014},
   number={6},
   pages={2442--2475},
   issn={0092-7872},
}


\bib{JoPe2}{article}{
  author={Josse, Alfrederic},
   author={P\`ene, Fran\c{c}oise},
   title={Normal class and normal lines of algebraic hypersurfaces},
  journal={hal-00953669v4},
    date={2014},
   pages={pp.~24},
  }

\bib{Kl}{article}{
   author={Klein, Felix},
   title={Eine neue Relation zwischen den Singularit\"{a}ten einer algebraischen
   Curve},
   language={German},
   journal={Math. Ann.},
   volume={10},
   date={1876},
   number={2},
   pages={199--209},
   issn={0025-5831},
}


\bib{LSS}{article}{
   author={Lang, Lionel},
   author={Shapiro,Boris},
   author={Shustin, Eugenii},
   title={On the number of intersection points of the contour of an amoeba with a line},
   journal={Indiana Univ. Math. J},
   date={to appear},
}


\bib{Pi1}{article}{
   author={Piene, Ragni},
   title={Polar classes of singular varieties},
   journal={Ann. Sci. \'{E}cole Norm. Sup. (4)},
   volume={11},
   date={1978},
   number={2},
   pages={247--276},
   issn={0012-9593},
}

\bib{Pi}{article}{
   author={Piene, Ragni},
   title={Higher order polar and reciprocal polar loci},
   journal={To appear in \emph{Facets of Algebraic Geometry: A Volume in honour of William Fulton's 80th Birthday,}
   Cambridge University Press, 2021; arXiv:2002.03691},
}


\bib{Pu1}{article}{
   author={Pushkar\cprime, Petr E.},
   title={Diameters of immersed manifolds and of wave fronts},
   language={English, with English and French summaries},
   journal={C. R. Acad. Sci. Paris S\'{e}r. I Math.},
   volume={326},
   date={1998},
   number={2},
   pages={201--205},
   issn={0764-4442},
}


\bib{Pu2}{article}{
   author={Pushkar\cprime, Petr E.},
   title={A generalization of Chekanov's theorem. Diameters of immersed
   manifolds and wave fronts},
   language={Russian},
   journal={Tr. Mat. Inst. Steklova},
   volume={221},
   date={1998},
   pages={289--304},
   issn={0371-9685},
   translation={
      journal={Proc. Steklov Inst. Math.},
      date={1998},
      number={2(221)},
      pages={279--295},
      issn={0081-5438},
   },
}

\bib{Sa}{book}{
   author={Salmon, George},
   title={A treatise on the higher plane curves: intended as a sequel to ``A
   treatise on conic sections''},
   series={3rd ed},
   publisher={Chelsea Publishing Co., New York},
   date={1960},
   pages={xix+395},
}

\bib{ST}{article}{
   author={Salarinoghabi, Mostafa},
   author={Tari, Farid},
   title={Flat and round singularity theory of plane curves},
   journal={Q. J. Math.},
   volume={68},
   date={2017},
   number={4},
   pages={1289--1312},
   issn={0033-5606},
}

\bib{ScT}{article}{
   author={Dias, Fabio Scalco},
   author={Tari, Farid},
   title={On vertices and inflections of plane curves},
   journal={J. Singul.},
   volume={17},
   date={2018},
   pages={70--80},
}


\bib{Sch}{article}{
   author={Schuh, Fred.},
   title={An equation of reality for real and imaginary plane curves with higher singularities},
   journal={KNAW, Proceedings, 6, 1903--1904, Amsterdam},
   date={1904},
   number={6},
   pages={764--773},
   }


\bib{Ta}{article}{
   author={Tabachnikov, Serge},
   title={The four-vertex theorem revisited---two variations on the old
   theme},
   journal={Amer. Math. Monthly},
   volume={102},
   date={1995},
   number={10},
   pages={912--916},
   issn={0002-9890},
}

\bib{Ya}{book}{
   author={Yates, Robert C.},
   title={A Handbook on Curves and Their Properties},
   publisher={J. W. Edwards, Ann Arbor, Mich.},
   date={1947},
   pages={x+245},
}

\bib{Yo}{book}{
   author={Yoder, Joella G.},
   title={Unrolling time},
   note={Christiaan Huygens and the mathematization of nature},
   publisher={Cambridge University Press, Cambridge},
   date={1988},
   pages={xii + 238},
   isbn={0-521-34140-X},
}

\bib{Vi}{article}{
   author={Viro, O. Ya.},
   title={Some integral calculus based on Euler characteristic},
   conference={
      title={Topology and geometry---Rohlin Seminar},
   },
   book={
      series={Lecture Notes in Math.},
      volume={1346},
      publisher={Springer, Berlin},
   },
   date={1988},
   pages={127--138},
}

\bib{Wa}{article}{
   author={Wall, C. T. C.},
   title={Duality of real projective plane curves: Klein's equation},
   journal={Topology},
   volume={35},
   date={1996},
   number={2},
   pages={355--362},
   issn={0040-9383},
}
		
\bib{Wi}{webpage}{
  author={Wikipedia},
  title={Evolute},
  url={https://en.wikipedia.org/wiki/Evolute},
  accessdate={20. October 2021}
  }
  
  \end{biblist}
\end{document}